\documentclass[a4paper, 10pt, twoside]{article}

\usepackage{exscale, fullpage}
\usepackage[centertags]{amsmath}
\usepackage{amssymb}
\usepackage{amsthm}
\usepackage{dsfont}
\usepackage[all]{xy}
\usepackage{lscape}
\usepackage{pdflscape}
\usepackage[hypertex]{hyperref}

\newtheorem{definition}{Definition}[section]
\newtheorem{theorem}[definition]{Theorem}
\newtheorem{proposition}[definition]{Proposition}
\newtheorem{lemma}[definition]{Lemma}
\newtheorem{corollary}[definition]{Corollary}

\newtheorem{question}[definition]{Question}
\newtheorem{deflemma}[definition]{Definition-Lemma}

\newcommand{\nd}{\noindent}

\newcommand{\dR}{{\mathds R}}
\newcommand{\dC}{{\mathds C}}
\newcommand{\dQ}{{\mathds Q}}
\newcommand{\dN}{{\mathds N}}

\newcommand{\dZ}{{\mathds Z}}
\newcommand{\dP}{{\mathds P}}

\newcommand{\dH}{{\mathbb H}}
\newcommand{\dL}{{\mathbb L}}

\newcommand{\bD}{{\mathbb D}}
\newcommand{\dD}{{\mathds D}}

\newcommand{\cC}{\mathcal{C}}
\newcommand{\cD}{\mathcal{D}}
\newcommand{\cE}{\mathcal{E}}
\newcommand{\cF}{\mathcal{F}}
\newcommand{\cG}{\mathcal{G}}
\newcommand{\cH}{\mathcal{H}}
\newcommand{\cI}{\mathcal{I}}

\newcommand{\cK}{\mathcal{K}}
\newcommand{\cL}{\mathcal{L}}
\newcommand{\cM}{\mathcal{M}}
\newcommand{\cN}{\mathcal{N}}
\newcommand{\cO}{\mathcal{O}}

\newcommand{\cQ}{\mathcal{Q}}
\newcommand{\cR}{\mathcal{R}}
\newcommand{\cS}{\mathcal{S}}

\newcommand{\cU}{\mathcal{U}}
\newcommand{\cV}{\mathcal{V}}

\newcommand{\cZ}{\mathcal{Z}}

\newcommand{\SC}{\scriptstyle}

\DeclareMathOperator{\Spec}{\textup{Spec}\,}
\DeclareMathOperator{\Proj}{\textup{Proj}\,}

\DeclareMathOperator{\Aut}{\textup{Aut}}
\DeclareMathOperator{\diag}{\textup{diag}}
\DeclareMathOperator{\vol}{\textup{vol}}

\DeclareMathOperator{\DR}{\mathit{DR}}

\DeclareMathOperator{\car}{\textup{char}}
\DeclareMathOperator{\pr}{\textup{pr}}

\DeclareMathOperator{\FL}{\overline{\textup{FL}}}
\DeclareMathOperator{\id}{\textup{id}}
\DeclareMathOperator{\gr}{\textup{gr}}

\DeclareMathOperator{\iso}{\textup{iso}}
\DeclareMathOperator{\Conv}{\textup{Conv}}

\DeclareMathOperator{\Int}{\textup{Int}}
\newcommand{\BL}{{_0\!}\widehat{\cM}^{loc}_{\!\widetilde{A}}}
\newcommand{\qM}{\cQ\!\cM^{loc}_{\!\widetilde{A}}}
\newcommand{\qMBL}{{_0\!}\cQ\!\cM^{loc}_{\!\widetilde{A}}}
\newcommand{\qMBLlog}{{_0\!}\cQ\!\cM_{\!\widetilde{A}}}
\newcommand{\qMBLhat}{\widehat{{_0\!}\cQ\!\cM}_{\!\widetilde{A}}}
\newcommand{\qMlog}{\cQ\!\cM_{\!\widetilde{A}}}
\newcommand{\qMhat}{\widehat{\cQ\!\cM}_{\!\widetilde{A}}}
\newcommand{\Leff}{\dL_{\textup{eff}}}
\DeclareMathOperator{\trTLEP}{\textup{trTLEP(n)}}
\DeclareMathOperator{\trTLEPlog}{\textup{log-trTLEP(n)}}
\newcommand{\XSigA}{X_{\!\Sigma_{\!A}}}
\newcommand{\E}{D[\partial_{\lambda_0}^{-1}]}

\begin{document}

\title{Logarithmic Frobenius manifolds, hypergeometric systems and quantum $\cD$-modules}
\author{Thomas Reichelt and Christian Sevenheck}

\maketitle

\begin{abstract}
We describe mirror symmetry for weak Fano toric manifolds as an equivalence
of filtered $\cD$-modules. We discuss in particular the logarithmic degeneration
behavior at the large radius limit point, and express the mirror correspondence as an isomorphism
of Frobenius manifolds with logarithmic poles. The main tool is an identification of the Gau\ss-Manin
system of the mirror Landau-Ginzburg model with a hypergeometric $\cD$-module,
and a detailed study of a natural filtration defined on this differential system.
We obtain a solution of the Birkhoff problem for lattices defined by this filtration
and show the existence of a primitive form, which yields the construction of Frobenius
structures with logarithmic poles associated to the mirror Laurent polynomial.
As a final application, we show the existence of a pure polarized non-commutative Hodge structure
on a Zariski open subset of the complexified K\"ahler moduli space of the variety.
\end{abstract}
\renewcommand{\thefootnote}{}
\footnote{
Th.R. is supported by the DFG grant He 2287/2-2. Ch.S. is supported by a DFG Heisenberg fellowship (Se 1114/2-1).\\
Both authors acknowledge partial support by the
ANR grant ANR-08-BLAN-0317-01 (SEDIGA).
}

\tableofcontents

\section{Introduction}
\label{sec:Introduction}

In this paper we study the differential systems that occur
in the mirror correspondence for smooth toric weak Fano varieties. On the so-called A-side
of mirror symmetry, which is mathematically expressed as the quantum
cohomology of this variety, these systems has been known since quite some time
as \emph{quantum $\cD$-modules}. A striking fact which makes their study attractive is
that the integrability of the corresponding connection encodes many properties of
the quantum product, in particular, the associativity, usually expressed by the
famous WDVV-equations. It is well-known (see, e.g., \cite{Manin})
that the quantum $\cD$-module (or first structure connection)
is essentially equivalent to the Frobenius structure defined by the quantum product on the cohomology space of the variety.

The main subject of this paper is to establish the same kind of structures for
the $B$-side, also called the Landau-Ginzburg model, of such a variety. This problem is related to more classical
objects in the theory of singularities of holomorphic or algebraic functions: namely,
period integrals, vanishing cycles and the Gau\ss-Manin connection in its various forms.
A by-now well-known construction going back to K.~Saito and M.~Saito endows the
semi-universal unfolding space of an isolated hypersurface singularity $f:(\dC^n,0)\rightarrow (\dC,0)$ with a Frobenius structure.
There are two main ingredients in constructing these structures: a very precise analysis
of the Hodge theory of $f$, which is done using the the so-called Brieskorn lattice,
and which culminates in a solution of the Birkhoff problem (also called a \emph{good basis} of the Brieskorn lattice). The second step is
to show that there is a specific section of the Brieskorn lattice, called \emph{primitive and homogeneous}
(which is also known as the ``primitive form'').

However, these Frobenius manifolds will never appear as the mirror of the quantum cohomology of
some variety. Sabbah has shown in a series of papers (partly joint with Douai,
see \cite{Sa1}, \cite{Sa2}, \cite{DS}) that the above results can be adapted if one starts with an algebraic function $f:U\rightarrow \dC$
defined on a smooth affine variety $U$. Besides the isolatedness of the critical locus of $f$, one is forced to impose
a stronger condition, known as tameness. Roughly speaking, it states that no change of the topology of the fibres
comes from critical points at infinity. The need for this condition reflects the fact that
the Gau\ss-Manin system of such a function, and other related objects, are not simply direct sums of the
corresponding local objects at the critical points. For tame functions, it is known that the
Birkhoff problem for the Brieskorn lattice always has a solution, similarly to the local case, one uses
information coming from the Hodge theory of $f$ to show this result. One the other
hand, the existence of a primitive (and homogeneous) form is a quite delicate problem which
is not known in general. It has been shown for certain tame polynomials in \cite{Sa2}, for convenient and non-degenerate Laurent polynomials
in \cite{DS} (and later with different methods in \cite{Dou2}) and also for some other particular
cases of tame functions (e.g., \cite{dGMS}). In any case, the outcome of these constructions is a germ
of a Frobenius structure on the deformation space of a single function. The general construction
in \cite{DS} does not give much information on how these Frobenius manifolds vary for families of, say, Laurent polynomials.
Notice also that the Frobenius structure associated to a Laurent polynomial (or even to a local singularity) is not at all unique, it
depends on both the choice of a good basis and a primitive (and homogeneous) form. However, there is a canonical choice
of a solution of the Birkhoff problem, predicted by the use of Hodge theory (more precisely, it is defined by Deligne's $I^{p,q}$-splitting of the Hodge filtration
associated to $f$), but in general Frobenius structure coming from this solution will not behave well in families.

For some special kind of Fano varieties like the projective spaces (see \cite{Baran}) or, more generally, for some
orbifolds like weighted projective spaces (\cite{Mann1}, \cite{DouMann}), it is possible to find explicit solutions to the Birkhoff problem
and to carry out the construction of the Frobenius manifold rather directly. Then one may compare the Brieskorn lattices
(or their extension using good bases) to the quantum $\cD$-module by an explicit identification of bases.
This yields isomorphisms of Frobenius manifolds and even some results on their degeneration behavior near the large radius limit
(see \cite{DouMann}), but of course this method is limited if one wants to attack more general classes
of examples.

In the present paper, we obtain such an identification of Frobenius manifolds for all weak Fano toric manifolds,
using Givental's $I\!=\!J$-theorem (\cite{Giv7}). We do not rely on the results of \cite{DS},
instead, we identify the family of Gau\ss-Manin systems attached to the Landau-Ginzburg model of our variety
with a certain hypergeometric $\cD$-module (also called Gelfand-Kapranov-Zelevinski-(GKZ)-system) by a purely
algebraic argument. This makes available some known results and constructions from the theory of these special $\cD$-modules,
and we are able to deduce a finiteness and a duality statement for the family of Brieskorn lattices. The tameness assumption from
above is used via an adaption of a result in \cite{Adolphson}, who has calculated the characteristic variety of
a hypergeometric $\cD$-module. In general, this tameness will hold on a Zariski open subspace of
the parameter space, and we show that if our variety is genuine Fano, then this is the whole parameter space.
An important point in the construction is to extend the family of Brieskorn lattices
on the K\"ahler moduli space of the variety to a certain partial compactification including the large radius limit point.
This compactification depends on a choice of coordinates on the complexified K\"ahler moduli space, that is, on a choice of a basis of nef classes
of the second cohomology of our variety. Once we have this logarithmic extension, we can apply \cite{Reich1} which
yields the construction of a \emph{logarithmic Frobenius manifold}, that is, a Frobenius structure on a manifold which
is the complement of a normal crossing divisor, and such that both multiplication and metric are defined on the sheaf
of logarithmic vector fields. At any point inside the K\"ahler moduli space, this restricts to
a germ of a Frobenius manifold constructed in \cite{DS}. In this sense our mirror statement also generalizes
the equivalence of Frobenius structures (at fixed points of the K\"ahler moduli space) known in particular cases like $\dP^n$.

Let us give a short overview on the content of this paper:
In section \ref{sec:HyperGM} we study in some detail
various differential systems associated to toric data defined by a smooth toric weak Fano variety $\XSigA$ (where $\Sigma_A$ is the defining fan), parts of
the results hold even more generally for a given set of vectors in a lattice.
In particular, we obtain an identification of a certain hypergeometric $\cD$-module with the Gau\ss-Manin system
of a generic family of Laurent polynomials defined by the toric data, more precisely,
with a partial Fourier-Laplace transformation of it (theorem \ref{theo:GM-GKZUp}).
We next study a natural filtration of this Gau\ss-Manin system, prove a finiteness result (theorem \ref{theo:BrieskornLatticeFree})
and show that it satisfies a compatibility condition with respect to the duality functor (proposition \ref{prop:FiltrationDuality}).

The actual Landau-Ginzburg model is a subfamily of the family of generic Laurent polynomials studied in section \ref{sec:HyperGM},
parameterized by the K\"ahler moduli space, i.e., by a $\dim H^2(\XSigA,\dC)$-dimensional torus.
In section \ref{sec:LogQDMod}, we first identify the Gau\ss-Manin system of the Landau-Ginzburg model of $\XSigA$
with a GKZ-system on the K\"ahler moduli space (corollary \ref{cor:GM-GKZDown}). In the second part of this section, we extend this module to a vector bundle with an integrable
connection having logarithmic poles along the boundary divisor of an appropriate compactification of the
K\"ahler moduli space (theorem \ref{theo:LogExtQDMod}). From this object we can derive, using a a method which goes back to \cite{Guest},
a specific basis defining a solution to the Birkhoff problem in family in the sense of \cite{DS}.
This is a family of $\dP^1$-bundles which extends the GKZ-$\cD$-module mentioned above.
An important new point is that this construction works taking into account the logarithmic degeneration behavior near the large radius
limit point. As a consequence, we can construct a canonical logarithmic Frobenius manifold associated to
the Landau-Ginzburg model of $\XSigA$, which has an algebraic structure on the subspace corresponding to
the compactified K\"ahler moduli space (theorem \ref{theo:logFrob-BSide}). One may speculate that it restricts to the canonical
Frobenius structure considered in \cite{DS} in a small neighborhood of any point of the K\"ahler moduli space
(question \ref{quest:CanonicalFrob}).

In section \ref{sec:AModel} we first recall very briefly the construction of the quantum $\cD$-module
of a projective variety, and then show that it is isomorphic, in the toric weak Fano case, to the family of $\dP^1$-bundles with connection constructed
in section \ref{sec:LogQDMod}.
From this we deduce (theorem \ref{theo:MirrorSymmetry}) an isomorphism of logarithmic Frobenius manifolds by invoking the main result
from \cite{Reich1}.

In the final section \ref{sec:ncHodge} we show (theorem \ref{theo:GKZunderliesNCHodge}), using the fundamental result from \cite{Sa8}
that the quantum $\cD$-module is equipped with the structure of a variation of pure polarized
non-commutative Hodge structures in the sense of \cite{KKP}. As there are several versions of this notion around, we briefly recall
the basic definitions and show how they apply in our context. This result strengthens a theorem
of Iritani (\cite{Ir3}), who directly shows the existence of $tt^*$-geometry in quantum cohomology, however,
he uses an asymptotic argument, whereas our approach gives the existence of an ncHodge structure
wherever the small quantum product is convergent and the mirror map is defined. We also deduce from the
construction of a logarithmic Frobenius manifold that this $tt^*$-geometry behaves quite nicely
along the boundary divisor of the K\"ahler moduli space, namely, that the corresponding harmonic bundle
is tame along this divisor (theorem \ref{theo:TameHarmonic}).

We finish this introduction by some remarks on how our work relates to other
papers concerning mirror symmetry for Fano varieties and hypergeometric differential systems:
As already mentioned above, our main result relies on Givental's $I=J$-theorem.
It is  certainly well-known to specialists (and it is briefly mentioned at some places in \cite{Giv7} and also in
subsequent papers) that the $I$-function is related to oscillating integrals and hence to
the Fourier-Laplace transformation of the Gau\ss-Manin system of the mirror Laurent polynomial,
but to the best of our knowledge, a thourough treatment of these issues is missing in the literature.
More recently, Iritani has given in \cite{Ir2} an analytic description of the differential system
associated to the Landau-Ginzburg model and discussed its relation to hypergeometric $\cD$-modules.
He considers the more general case of toric weak Fano orbifolds, however, solutions to the Birkhoff problem
resp. Frobenius structures are not treated in loc.cit.
Passing through the analytic category one also looses
the algebraic nature of the objects involved, which may be an obstacle in some situations. As an example,
one cannot apply the general results on formal decomposition of meromorphic bundles with connection
from \cite{Mo5-0} and \cite{Mo5} for non-algebraic bundles. Nevertheless, some of the techniques used
here are also present in \cite{Ir2}, and at some points our presentation is (without explicit mentioning) similar to
that of loc.cit.

Finally, let us notice that although
one may think of an extension of some of our results (like those in section \ref{sec:HyperGM}) to the orbifold case,
there is a serious obstacle in the construction of a logarithmic Frobenius structure associated to the Landau-Ginzburg model
of a weak Fano toric orbifold. This is mainly due to the
fact that the ``limit'' orbifold cup product does not satisfy an ``$H^2$-generation condition'', in contrast
to the case of toric manifolds (see also the preprint \cite{DouMann} for a discussion of this
phenomenon for the case of weighted projected spaces).

\section{Hypergeometric $\cD$-modules and filtered Gau\ss-Manin systems}
\label{sec:HyperGM}

In this section we study Gau\ss-Manin systems associated to generic families
of Laurent polynomials. We show that (a partial Fourier-Laplace transformation
of) these $\cD$-modules always have a hypergeometric structure, i.e.,
are isomorphic to (a partial Fourier-Laplace transformation
of) a certain GKZ-system. Moreover, both Gau\ss-Manin systems and
GKZ-systems carry natural filtrations by $\cO$-modules. For the Gau\ss-Manin system,
these are the so-called Brieskorn lattices, as studied, for more general
polynomial functions, in \cite{Sa2}. We show that the above identification also works at the level of lattices.
As an application, we prove that if the family of Laurent polynomials
is associated to a fan of a smooth toric weak Fano manifold, then outside a certain ``bad part'' of the parameter
space, the family of Brieskorn lattices is $\cO$-locally free. This will be needed later in
the construction of Frobenius manifolds associated to these special families of Laurent polynomials.
Finally, we study the holonomic dual of the Gau\ss-Manin system and obtain (up to a shift of the homological degree)
an isomorphism of this dual to the Gau\ss-Manin system itself. The way of constructing
this isomorphism is purely algebraic, using a resolution called Euler-Koszul complex
of the hypergeometric $\cD$-module which is isomorphic to the Gau\ss-Manin system.
This proof differs from \cite{Sa2} or \cite{DS}, where the duality isomorphism
is obtain in a topological way. We could also give a topological proof
along the lines of the quoted papers, by using a partial compactifications of the family
of Laurent polynomials and a smoothness property at infinity (see the proof of proposition \ref{prop:DiscFinite}
for a description of this partial compactification). However, our algebraic
approach gives almost for free that the above mentioned filtration is compatible (up to a shift),
with the duality isomorphism. This fact is also needed for the construction of Frobenius structures.

\subsection{Hypergeometric systems and Gau\ss-Manin systems}
\label{subsec:GKZandGM}

We start with the following set of data: Let $N$ be a finitely generated free abelian group of rank $n$,
for which we choose once and for all a basis which identifies it with $\dZ^n$. Let
$\underline{a}_1,\ldots,\underline{a}_m$ be elements of $N$, which we also see as vectors of $\dZ^n$.
We suppose that $\underline{a}_1,\ldots,\underline{a}_m$ generates $N$,
if we only have
$\sum_{i=1}^n \dQ \underline{a}_i =N_\dQ:=N\otimes \dQ$, then some of our results can be adapted, see
proposition \ref{prop:NnotGenerated} below.
In order to orient the reader, let us point out from the very beginning that
the case we are mostly interested in is when these vectors are the
primitive integral generators of the rays of a fan $\Sigma_A$ in
$N_\dR:=N\otimes \dR$ defining a smooth projective toric variety $\XSigA$ which is \emph{weak Fano}, that
is, such that the anticanonical divisor $-K_{\XSigA}$ is numerically effective (nef). The Fano case,
i.e., when $-K_{\XSigA}$ is ample is of particular importance and will sometime be treated apart, as
there are cases in which we obtain stronger statements for genuine Fano varieties. See also the proof of proposition \ref{prop:ClosedEmb_NormalCMGorenstein},
the proof of lemma \ref{lem:FanoNonDeg} and the beginning of section \ref{sec:LogQDMod} for toric characterizations of the weak Fano condition.
We will abbreviate this case by saying that $\underline{a}_1,\ldots,\underline{a}_m$ \emph{are defined by toric data}.
We write $\dL$ for the
module of relations between $\underline{a}_1,\ldots,\underline{a}_m$, i.e., $\underline{l}\in\dL\subset \dZ^m$ iff
$\sum_{i=1}^m l_i \underline{a}_i=0$.
We will denote by $S_0$ the $n$-dimensional torus $\Spec \dC[N]$ with coordinates $y_1,\ldots, y_n$ and by $W'$ the $ m$-dimensional affine space
$\Spec \dC[\oplus_{i=1}^m \dN \underline{a}_i]$ with coordinates $w_1,\ldots,w_m$. We are slightly pedantic in this latter definition in order to make a clear difference with the dual space, called $W$, which will appear later.

An important point in the arguments used below will be to consider the
following set of extended vectors: Put $\widetilde{N}:=\dZ \times N\cong \dZ^{n+1}$,
$\widetilde{\underline{a}}_i:=(1,\underline{a}_i)\in \widetilde{N}$ for all $i=1,\ldots, m$ and $\widetilde{\underline{a}}_0:=(1,\underline{0})\in\widetilde{N}$. Write $\widetilde{A}=(\widetilde{\underline{a}}_0,
\widetilde{\underline{a}}_1,\ldots,\widetilde{\underline{a}}_m)$. Notice that the module of relations
of $\widetilde{A}$ is isomorphic to $\dL$, any $\underline{l}=(l_1,\ldots,l_m)\in \dL$ gives in a unique way rise to the relation $(-\sum_{i=1}^m l_i) \widetilde{\underline{a}}_0+
\sum_{i=1}^m l_i \widetilde{\underline{a}}_i=0$. By abuse of notation, we also write $\dL$ for
the module of relations of $\widetilde{A}$. As another piece of notation, we put $\overline{l}:=\sum_{i=1}^m l_i$.
Let $V'=\Spec \dC[\oplus_{i=0}^m \dN \widetilde{\underline{a}}_i]$ with coordinates $w_0,\ldots,w_m$ and $V$ the
dual space, with coordinates $\lambda_0,\ldots,\lambda_m$.
We also need the $m$-dimensional torus $S_1:=\Spec\dC\left[(\oplus_{i=1}^m \dZ \underline{a}_i)^\vee\right]$,
with inclusion map $j:S_1\hookrightarrow W$.
Moreover, put $\widehat{V}:=\Spec\dC[\dN \widetilde{\underline{a}}_0]\times W$
and $\widehat{T}:=\Spec\dC[\dN \widetilde{\underline{a}}_0] \times S_1$,
we still denote the map $\widehat{T} \hookrightarrow \widehat{V}$ by $j$. We put $\tau=-w_0$ so that
$(\tau,\lambda_1,\ldots,\lambda_m)$ gives coordinates on $\widehat{V}$ resp. $\widehat{T}$.
We will also write $\dC_\tau$ for $\Spec \dC[\dN \widetilde{\underline{a}}_0]$ and
$\dC^*_\tau$ for $\Spec \dC[\dZ \widetilde{\underline{a}}_0]$.
Later we will consider algebraic $\cD_{\widehat{V}}$- (resp. $\cD_{\widehat{T}}$)-modules which are localized along $\tau=0$, and in this
case we also use the variable $z:=\tau^{-1}$. Sometimes we will implicitly identify such modules
with their restriction to $\dC^*_\tau \times W$ resp. to $\dC^*_\tau \times S_1$.

The first geometric statement about these data is the following proposition.
\begin{proposition}
\label{prop:ClosedEmb_NormalCMGorenstein}
\begin{enumerate}
\item
Consider the map
$$
\begin{array}{rcl}
k:S_0  & \longrightarrow & W'\\ \\
(y_1,\ldots,y_n) & \longmapsto & (w_1,\ldots,w_m):=(\underline{y}^{\underline{a}_1},\ldots,\underline{y}^{\underline{a}_m}),
\end{array}
$$
where $\underline{y}^{\underline{a}_i}:=\prod_{k=1}^n y_k^{a_{ki}}$.
Suppose that $\underline{0}$ lies in the interior of
$\Conv(\underline{a}_1,\ldots,\underline{a}_m)$,
where for any subset $K\subset N$, $\Conv(K)$ denotes
the convex hull of $K$ in $N_\dR$. Then $k$ is a closed embedding.
\item
Suppose that $\underline{a}_1,\ldots,\underline{a}_m$ are defined by toric data.
In particular, the completeness of $\Sigma_A$ implies that $0$ is an interior point of
$\Conv(\underline{a}_1,\ldots,\underline{a}_m)$.
Let $\dN \widetilde{A}=\sum_{i=0}^m \dN \widetilde{\underline{a}}_i$, then
$\dN \widetilde{A}$ is a normal semigroup, i.e. it satisfies
$\widetilde{N} \cap C(\widetilde{A})= \dN \widetilde{A}$ and positive, i.e., the origin is the only unit in $\dN \widetilde{A}$.
Here for a finite set $\{\widetilde{\underline{x}}_1,\ldots,\widetilde{\underline{x}}_k\}$ we write $C(\{\widetilde{\underline{x}}_1,\ldots,\widetilde{\underline{x}}_k\})$ for the
cone $\sum_{j=1}^k \dR_{\geq 0} \widetilde{\underline{x}}_j$. The associate semigroup ring $\Spec\dC[\dN \widetilde{A}]$ is normal, Cohen-Macaulay and Gorenstein.
\end{enumerate}
\end{proposition}
\begin{proof}
\begin{enumerate}
\item
The condition that the origin is a interior point of the convex hull of the vectors $\underline{a}_i$ translates
into the existence of a relation $\underline{l}=(l_1,\ldots,l_m)\in \dL\cap \dZ^m_{>0}$ between $\underline{a}_1,\ldots,\underline{a}_m$ consisting of positive integers. On the other hand, the closure of the image of the map $k$ is contained in the vanishing locus of the so-called toric ideal
$$
I=\left(\prod_{i:l_i <0} w_i^{-l_i} - \prod_{i:l_i>0 } w_i^{l_i}\right)_{\underline{l}\in\dL} \subset \cO_{W'}.
$$
From the existence of $\underline{l}\in\dL\cap\dZ^m_{>0}$ we deduce that
the function $\prod_{i=1}^m w_i^{l_i} -1$ lies in $I$. This shows that for any point $\underline{w}=(w_1, \ldots,w_m) \in \overline{\mathit{Im}(k)}\subset V(I)\subset W'$, we have $w_i\neq 0$,
i.e., $\underline{w}\in \mathit{Im}(k)$.
\item
First we show the normality property:
Consider any integer vector $\widetilde{\underline{x}}=(x_0,x_1,\ldots,x_n)\in C(\widetilde{A}) \cap \widetilde{N}$. We have
\begin{equation}\label{eq:Deckel}
C(\widetilde{A}) \cap \left(\{1\} \times N_\dR\right) =\bigcup_{\lambda_i \in \dR_{\geq 0} ; \sum_{i=0}^m \lambda_i =1} \lambda_i \widetilde{\underline{a}}_i = \{ 1 \} \times \Conv(\underline{a}_1, \ldots ,\underline{a}_m)
\end{equation}
Now define
\[
P(\Sigma_A) = \bigcup_{\langle \underline{a}_{i_1}, \ldots , \underline{a}_{i_n} \rangle \in \Sigma_A(n)} \Conv(\underline{0},\underline{a}_{i_1}, \ldots , \underline{a}_{i_n})
\]
We have the following reformulation of the weak Fano condition (see, e.g., \cite[page 268]{Wisniewski}):
\[
-K_{\XSigA}\;\; \textup{ is nef } \Longleftrightarrow P(\Sigma_A) \;\; \textup{ is convex.}
\]
Hence by assumption we know that $P(\Sigma_A)$ is convex. We claim that $P(\Sigma_A) = \Conv(\underline{a}_1, \ldots , \underline{a}_m)$. The inclusion $\subset$ follows from the fact $\underline{0}, \underline{a}_{i_1}, \ldots , \underline{a}_{i_n} \in \Conv(\underline{a}_1, \ldots ,\underline{a}_m)$ for $\langle \underline{a}_{i_1}, \ldots , \underline{a}_{i_n} \rangle \in \Sigma_A(n)$. The other inclusion follows from $\underline{a}_1, \ldots , \underline{a}_m \in P(\Sigma_A)$ and the convexity of $P(\Sigma_A)$. From the claim and
    equality \eqref{eq:Deckel} we get the following decomposition of the cone $C(\widetilde{A})$:
\[
C(\widetilde{A}) = \bigcup_{\langle \underline{a}_{i_1}, \ldots , \underline{a}_{i_n} \rangle \in \Sigma_A(n)} C(\{ \widetilde{\underline{a}}_0, \widetilde{\underline{a}}_{i_1}, \ldots , \widetilde{\underline{a}}_{i_n}\})
\]
Using this decomposition, we see that $\widetilde{x}$ lies in a cone $C(\widetilde{\underline{a}}_0, \widetilde{\underline{a}}_{j_1}, \ldots ,\widetilde{\underline{a}}_{j_n})$, that is, there
are  $\lambda_0, \lambda_{j_1},\ldots, \lambda_{j_n} \in \dR_{\geq 0}$ such that $\widetilde{x} = \lambda_0 \widetilde{\underline{a}}_0 + \sum_{k=1}^n \lambda_{j_k}\widetilde{\underline{a}}_{j_k}$. Notice that $\widetilde{\underline{a}}_0, \widetilde{\underline{a}}_{j_1}, \ldots , \widetilde{\underline{a}}_{j_n}$ is $\dZ$-basis of $\widetilde{N}$, as $\underline{a}_{j_1}, \ldots , \underline{a}_{j_n}$ is a $\dZ$-basis of $N$ which follows from the smoothness of $\Sigma_A$. From this follows $\widetilde{x} \in \dN \widetilde{A}$. Notice also that the ``exterior boundary'' $\partial C(\widetilde{\underline{a}}_0, \widetilde{\underline{a}}_{j_1}, \ldots ,\widetilde{\underline{a}}_{j_n})\cap
\partial C(\widetilde{A})$ equals $\sum_{k=1}^n \dR_{\geq 0} \widetilde{\underline{a}}_{i_k}$ so that
$\widetilde{\underline{x}}\in \Int(C(\widetilde{A}))$ precisely iff the coefficient $\lambda_0$ in the above sum is positive.

From the fact that $\dN \widetilde{A}$ is normal it follows that $\Spec\dC[\dN \widetilde{A}]$ is Cohen-Macaulay by a classical
result due to Hochster (\cite[theorem 1]{Hoch}).
That $\dN \widetilde{A}$ is positive is equally easy to see: it follows (see, e.g., \cite[lemma 7.12]{MillSturm}) from the fact that $C(\widetilde{A})$ is pointed, i.e., that the vectors $(\widetilde{\underline{a}}_i)_{i=0,\ldots,m}$ are contained in the half-space
$\{\widetilde{\underline{x}}\in\dR^{n+1}\,|\, \widetilde{x}_0>0\}$.

It remains to show that $\Spec\dC[\dN \widetilde{A}]$ is Gorenstein: We use \cite[corollary 6.3.8]{HerzBr} stating that  this property is equivalent,
for normal positive semigroup rings, to the fact that that there is a vector $\widetilde{\underline{c}}\in\Int(\dN \widetilde{A})$ with
$$
\Int(\dN \widetilde{A}) = \widetilde{\underline{c}}+ \dN \widetilde{A}.
$$
From the above proof of the normality of $\dN \widetilde{A}$ we see that
$\Int(\dN \widetilde{A}) = \widetilde{N}\cap \Int(C(\widetilde{A}))$. On the other hand,
the map $\widetilde{N}\rightarrow \widetilde{N}$ which sends $\widetilde{\underline{x}}$ to
$\widetilde{\underline{x}}+(1,\underline{0})$ induces a bijection
from $C(\widetilde{A})$ to $\Int(C(\widetilde{A}))$, this follows from the characterization of $C(\widetilde{A})$ given above.

\end{enumerate}
\end{proof}

In order to state our first main result, we will associate (several variants of) a $\cD$-module) to the set of
vectors $\underline{a}_1,\ldots,\underline{a}_m$ above. This construction is a special case of the well-known
\emph{A-hypergeometric systems} (also called hypergeometric $\cD$-modules or GKZ-systems). We recall first the
general definition.
\begin{definition}[\cite{GKZ1}, \cite{Adolphson}]
\label{def:GKZ}
Consider a lattice $\dZ^t$ and vectors $\underline{b}_1,\ldots,\underline{b}_s\in\dZ^t$ which we also write as
a matrix $B=(\underline{b}_1,\ldots,\underline{b}_s)$. Moreover, let $\beta=(\beta_1,\ldots,\beta_t)\in\dC^t$.
Write (as above) $\dL$ for the module of relations of $B$ and
$\cD_{\dC^s}$ for the sheaf of rings of algebraic differential operators on $\dC^s$ (where
we choose $x_1,\ldots,x_s$ as coordinates).
Define
$$
\cM^\beta_B:=\cD_{\dC^s}/\left((\Box_{\underline{l}})_{\underline{l}\in\dL}+(Z_k)_{k=1,\ldots t}\right),
$$
where
$$
\begin{array}{rcl}
\Box_{\underline{l}} & := & \prod_{i:l_i<0} \partial_{x_i}^{-l_i} -\prod_{i:l_i>0} \partial_{x_i}^{l_i} \\ \\
Z_k & := & \sum_{i=1}^s b_{ki} x_i \partial_{x_i}+\beta_k
\end{array}
$$
$\cM^\beta_B$ is called hypergeometric system.
\end{definition}
We will use at several places in this paper the Fourier-Laplace transformation for algebraic $\cD$-modules. In order to introduce a convenient notation for this operation, let $X$ be a smooth algebraic variety, and $\cM$ a $\cD_{\dC^s\times X}$-module, where we have coordinates $(x_1,\ldots,x_s)$ on $\dC^s$. Then
we write $\FL^{y_1,\ldots,y_s}_{x_1,\ldots,x_s}\cM$ for the $\cD_{(\dC^s)^\vee \times X}$-module, which is
the same as $\cM$ as a $\cD_X$-module, and where $y_i$ acts as $-\partial_{x_i}$ and $\partial_{y_i}$ acts
as $x_i$, here $y_1,\ldots,y_s$ are the dual coordinates on $(\dC^s)^\vee$. One could also work
with the functor $\textup{FL}^{y_1,\ldots,y_s}_{x_1,\ldots,x_s}$, where
$y_i$ acts as $\partial_{x_i}$ and $\partial_{y_i}$ acts as $-x_i$, this would lead to slightly uglier formulas.
\begin{definition}
\label{def:GKZ_ExtGKZ_FlGKZ}
Let $\cD_V$, $\cD_{\widehat{V}}$ and $\cD_{\widehat{T}}$ be the sheaves of algebraic differential operators on $V$, $\widehat{V}$ and $\widehat{T}$, respectively.
\begin{enumerate}
\item
Consider the hypergeometric system $\cM^\beta_{\widetilde{A}}$ associated to the
vectors $\widetilde{\underline{a}}_0,\widetilde{\underline{a}}_1,\ldots,\widetilde{\underline{a}}_m$.
More explicitly, $\cM^\beta_{\widetilde{A}}:=\cD_V/\cI$, where $\cI$ is the sheaf
of left ideals in $\cD_V$ defined by
$$
\cI:=\cD_V(\Box_{\underline{l}})_{\underline{l}\in\dL}+\cD_V(Z_k)_{k\in\{1,\ldots,n\}}+\cD_V E,
$$
where
$$
\begin{array}{rcrcrcr}
\Box_{\underline{l}} & := & \partial_{\lambda_0}^{\overline{l}} \cdot \prod\limits_{i:l_i<0} \partial_{\lambda_i}^{-l_i} & - & \prod\limits_{i:l_i>0} \partial_{\lambda_i}^{l_i}  & \textup{ if } & \overline{l} \geq 0,\\ \\
\Box_{\underline{l}} & := & \prod\limits_{i:l_i<0} \partial_{\lambda_i}^{-l_i} & - & \partial_{\lambda_0}^{-\overline{l}} \cdot \prod\limits_{i:l_i>0} \partial_{\lambda_i}^{l_i}  & \textup{ if } & \overline{l}<0,  \\ \\
Z_k & := & \sum_{i=1}^m a_{ki} \lambda_i\partial_{\lambda_i} + \beta_k,\\ \\
E   & := & \sum_{i=0}^m \lambda_i\partial_{\lambda_i}+\beta_0,
\end{array}
$$
here $\underline{a}_i =(a_{1i},\ldots,a_{ni})$ when seen as a vector in $\dZ^n$.
\item
Let $\widehat{\cM}^{\beta}_{\widetilde{A}}$ be the $\cD_{\widehat{V}}$-module $\FL^{w_0}_{\lambda_0}(\cM^\beta_{\widetilde{A}})[\tau^{-1}]$.
In other words, $\widehat{\cM}^\beta_{\widetilde{A}}=\cD_{\widehat{V}}[\tau^{-1}]/\widehat{\cI}$, where $\widehat{\cI}$ is the left ideal generated
by the Fourier-Laplace transformed operators $\widehat{\Box}_{\underline{l}}$, $\widehat{Z}_k$ and $\widehat{E}$, i.e.,
$$
\begin{array}{rcl}
\widehat{\Box}_{\underline{l}} & := & \tau^{\overline{l}} \cdot \prod\limits_{i:l_i<0} \partial_{\lambda_i}^{-l_i}  -  \prod\limits_{i:l_i>0} \partial_{\lambda_i}^{l_i}  =
z^{-\overline{l}} \cdot \prod\limits_{i:l_i<0} \partial_{\lambda_i}^{-l_i}  -  \prod\limits_{i:l_i>0} \partial_{\lambda_i}^{l_i},  \\ \\
\widehat{Z}_k & := & \sum_{i=1}^m a_{ki} \lambda_i\partial_{\lambda_i} + \beta_k,\\ \\
\widehat{E}   & := & \sum_{i=1}^m \lambda_i\partial_{\lambda_i} -\tau\partial_\tau - 1+\beta_0,
= \sum_{i=1}^m \lambda_i\partial_{\lambda_i} + z\partial_z - 1+\beta_0.
\end{array}
$$
\item
Define $\widehat{\cM}^{\beta,loc}_{\widetilde{A}}:=j^* \widehat{\cM}^{\beta}_{\widetilde{A}}$ to be the restriction of $\widehat{\cM}^{\beta}_{\widetilde{A}}$ to $\widehat{T}$. We will
use the presentations $\cD_{\widehat{T}}[\tau^{-1}]/\widehat{\cI}'$ and $\cD_{\widehat{T}}[\tau^{-1}]/\widehat{\cI}''$ of
$\widehat{\cM}^{\beta,loc}_{\widetilde{A}}$
where $\widehat{\cI}'$ resp. $\widehat{\cI}''$ is the sheaf of left ideals generated by
$\widehat{\Box}'_{\underline{l}}$, $\widehat{Z}_k$ and $\widehat{E}$
resp.
$\widehat{\Box}''_{\underline{l}}$, $\widehat{Z}_k$ and $\widehat{E}$, where
$$
\widehat{\Box}'_{\underline{l}} :=  z^{\sum_{i:l_i > 0}l_i}\cdot \widehat{\Box}_{\underline{l}}
\quad\textup{ and }\quad
\widehat{\Box}''_{\underline{l}} := \prod_{i:l_i > 0} (z\cdot \lambda_i)^{l_i} \cdot \widehat{\Box}_{\underline{l}},
$$
so that
$$
\widehat{\Box}'_{\underline{l}} = \prod\limits_{i:l_i<0} (z\partial_{\lambda_i})^{-l_i}  -  \prod\limits_{i:l_i>0} (z\partial_{\lambda_i})^{l_i},
$$
and, using the formula $\lambda_i^j\partial_{\lambda_i}^j = \prod_{\nu=0}^{j-1} (\lambda_i\partial_{\lambda_i}-\nu)$,
$$
\widehat{\Box}''_{\underline{l}}\;=\;\prod_{i=1}^m \lambda_i^{l_i} \cdot \prod_{i:l_i<0} \prod_{\nu=0}^{-l_i-1} \left(z\lambda_i\partial_{\lambda_i}-\nu z\right)
- \prod_{i:l_i>0} \prod_{\nu=0}^{l_i-1} \left(z\lambda_i\partial_{\lambda_i}-\nu z\right).
$$
Notice that obviously $\widehat{\cI}'=\widehat{\cI}''$, but we will later need the
two different explicit forms of the generators of this ideal, for that reason, two different
names are appropriate.
\item
Write $\cM_{\widetilde{A}}:=\cM^{(1,\underline{0})}_{\widetilde{A}}$,
$\widehat{\cM}_{\widetilde{A}}:=\widehat{\cM}^{(1,\underline{0})}_{\widetilde{A}}$
and $\widehat{\cM}^{loc}_{\widetilde{A}}:=\widehat{\cM}^{(1,\underline{0}),loc}_{\widetilde{A}}$.
\end{enumerate}
In order to avoid too heavy notations, we will sometimes identify $\widehat{\cM}^\beta_{\widetilde{A}}$ resp.
$\widehat{\cM}^{\beta,loc}_{\widetilde{A}}$ with the corresponding modules over either $\dC^*_\tau\times W$ resp. $\dC^*_\tau\times S_1$
or $\dP^1_z\times W$ resp. $\dP^1_z\times S_1$, here $\dP^1_z$ is $\dP^1$ with $0$ defined by $z=0$.
\end{definition}

The first main result is a comparison of these $\cD$-modules to some Gau\ss-Manin systems
associated to families of Laurent polynomials. When this paper was written, a similar result
appeared in \cite{AdolphSperber}. The techniques of loc.cit. are not too far from those used
in the proof of the next theorem, however, it seems not to be more efficient to translate their
result into our situation than to give a direct proof.
\begin{theorem}
\label{theo:GM-GKZUp}
Let $\underline{a}_1,\ldots,\underline{a}_m\in N$ such that
$\sum_{i=1}^m \dZ \underline{a}_i = N$. Consider the family of Laurent
polynomials $\varphi: S_0 \times W \rightarrow \dC_t\times W$ defined by
$$
\varphi((y_1,\ldots,y_n),(\lambda_1,\ldots,\lambda_m))= (\sum_{i=1}^m \lambda_i \underline{y}^{\underline{a}_i},\underline{\lambda}) =\left(\sum_{i=1}^m \lambda_i \prod_{k=1}^n y_k^{a_{ki}},(\lambda_1,\ldots,\lambda_m)\right)=:\left(t,\lambda_1,\ldots,\lambda_m\right).
$$
Then there is an isomorphism
$$
\phi:\widehat{\cM}_{\widetilde{A}} \longrightarrow
\FL_t^\tau(\cH^0 \varphi_+\cO_{S_0 \times W})[\tau^{-1}] =:G
$$
of $\cD_{\widehat{V}}$-modules.
\end{theorem}
Before entering into the proof, let us recall the following well-known
description of the Fourier-Laplace transformation of the Gau\ss-Manin system.
\begin{lemma}\label{lem:DirectImageSimplified}
Write $\varphi=(F,\pi)$, where $F:S_0\times W \rightarrow \dC_t$, $(\underline{y},\underline{\lambda})\mapsto\sum_{i=1}^m \lambda_i \underline{y}^{\underline{a}_i}$
and $\pi:S_0\times W\rightarrow W$ is the projection. Then there
is an isomorphism of $\cD_{\widehat{V}}$-modules
$$
G\cong \cH^0\left(\pi_*\Omega^{\bullet+n}_{S_0\times W / W}[z^\pm],d-z^{-1} \cdot d F\wedge\right),
$$
where $d$ is the differential in the relative de Rham complex $\pi_* \Omega^\bullet_{S_0\times W/W}$.
The structure of a $\cD_{\widehat{V}}$-module on the right hand side is defined as follows
$$
\begin{array}{rcl}
\partial_z (\omega \cdot z^i) & := & i\cdot \omega\cdot z^{i-1} - z^{-2}F\cdot \omega \cdot z^i, \\ \\
\partial_{\lambda_i} (\omega \cdot z^i) & := & \partial_{\lambda_i}(\omega)\cdot z^i + \partial_{\lambda_i} F \cdot \omega \cdot z^{i-1}
= \partial_{\lambda_i}(\omega)\cdot z^i + \underline{y}^{\underline{a}_i}\cdot \omega \cdot z^{i-1}, \\ \\
\end{array}
$$
where $\omega\in\Omega^n_{S_0\times W / W}$.
\end{lemma}
\begin{proof}
The identification of both objects as $\cD_{\widehat{V}}/\cD_{W}$-modules is well-known (see, e.g., \cite[proposition 2.7]{DS}, where
the result is stated, for a proof, one uses \cite[lemma 2.4]{SM}). The proof of the formulas for the action of the vector fields $\partial_{\lambda_i}$
can be found, in a similar situation, in \cite[lemma 7]{Sev1}.
\end{proof}

\begin{proof}[Proof of the theorem]
Throughout the proof, we will use the following notation: Let $X$ be a smooth algebraic variety, and $f$ a meromorphic function
on $X$ with pole locus $D:=g^{-1}(\infty)\subset X$, then we denote by $\cO_X(*D)\cdot e^f$ the locally free $\cO_X(*D)$-module of rank one with connection operator $\nabla:=d+df\wedge$. The $\cD_X$-module thus obtained has irregular singularities along $D$, notice that this irregularity locus may
lay in a boundary of a smooth projective compactification $\overline{X}$ of $X$ if $f\in\cO_X$. For any $\cD_X$-module $\cM$, we write
$\cM\cdot e^f$ for the tensor product $\cM\otimes_{\cO_X} \cO_X(*D)\cdot e^f$.

Put $T_0:=\Spec\dC[\widetilde{N}]$ with coordinates $y_0,y_1,\ldots,y_n$, and define
$$
\begin{array}{rcl}
\widetilde{k}: T_0 & \longrightarrow & \dC^*\times W' \subset V'\\ \\
(y_0,y_1,\ldots,y_n) & \longmapsto & \left(w_0:=y_0,(w_i:=y_0\cdot \underline{y}^{\underline{a}_i})_{i=1,\ldots,m}\right),
\end{array}
$$
where, as before, we write $\underline{y}^{\underline{a}_i}$ for the product $\prod_{k=1}^n y_k^{a_{ki}}$.
It is an obvious consequence of the first point of proposition \ref{prop:ClosedEmb_NormalCMGorenstein}, that $\widetilde{k}$ is again a closed
embedding from $T_0$ to $\dC^*\times W'$.
Write moreover
$p$ for the projection $\dC_\tau^*\times S_0 \times W \twoheadrightarrow \dC_\tau^*\times W$. We identify $T_0$ with $\dC_\tau^*\times S_0$ by the map
$(y_0,y_1,\ldots,y_n)\mapsto (-y_0,y_1,\ldots,y_n)=(\tau,y_1,\ldots,y_n)$.

First we claim that
\begin{equation}
\label{eq:ProofMainTheo-1}
G \cong \cH^0 p_+ \left(\cO_{\dC^*_\tau\times S_0 \times W}\cdot e^{-\tau\sum_{i=1}^m\lambda_i \underline{y}^{\underline{a}_i}}\right).
\end{equation}
As $p$ is a projection, the direct image $p_+$ of any module is nothing but its relative de Rham
complex, i.e.
$$
\cH^0 p_+ \left(\cO_{\dC^*_\tau\times S_0 \times W}\cdot e^{-\tau\sum_{i=1}^m\lambda_i \underline{y}^{\underline{a}_i}}\right)
\cong
\cH^0\left(p_*\Omega^{\bullet+n}_{\dC^*_\tau\times S_0\times W/\dC^*_\tau\times W}, d-\tau\cdot d F\wedge\right),
$$
and this module is the same as $G$, using lemma \ref{lem:DirectImageSimplified}.
It follows from the projection formula (\cite[corollary 1.7.5]{Hotta}) that
$$
\left((\widetilde{k}\times \id_W)_+\cO_{T_0\times W}\right) \cdot e^{\sum_{i=1}^m \lambda_i w_i} =(\widetilde{k}\times \id_{W})_+\left(\cO_{T_0\times W}\cdot e^{y_0\sum_{i=1}^m \lambda_i \underline{y}^{\underline{a}_i}}\right).
$$
This can also be shown by a direct calculation, in fact, both modules are quotients of $\cD_{\dC^*_\tau\times W' \times W}$.
Now consider the following diagram
$$
\xymatrix{
&& & \dC_\tau^*\times W'\times W \ar@{->>}[ld]_{\pi_1} \ar@{->>}[rd]^{\pi_2} \\
T_0  \ar^{\widetilde{k}}[rr]&& \dC_\tau^* \times W' && \dC_\tau^*\times W,
}
$$
where $\pi_1$ and $\pi_2$ are the obvious projections. As
$\pi_2\circ(\widetilde{k}\times\id_W) = p$, we obtain that
$$
\cH^0 p_+ \left(\cO_{S_0\times \dC_\tau^* \times W} \cdot e^{-\tau\sum_{i=1}^m \lambda_i \underline{y}^{\underline{a}_i}}\right)
=
\cH^0 \pi_{2,+}\left(((\widetilde{k}\times \id_{W})_+ \cO_{T_0\times W}) \cdot e^{\sum_{i=1}^m \lambda_i w_i}\right).
$$
On the other hand, we obviously have that $(\widetilde{k}\times \id_W)_+ \cO_{T_0\times W}=\pi_1^* \widetilde{k}_+ \cO_{T_0}$, hence
$$
\cH^0 \pi_{2,+}\left(((\widetilde{k}\times \id_W)_+ \cO_{T_0\times W}) \cdot e^{\sum_{i=1}^m \lambda_i w_i}\right) =
\cH^0 \pi_{2,+}\left((\pi_1^* \widetilde{k}_+ \cO_{T_0})\cdot e^{\sum_{i=1}^m \lambda_i w_i}\right),
$$
Now we use the following well-known description of the Fourier-Laplace transformation:
$$
\cH^0 \pi_{2,+}\left(((\pi_1)^* \widetilde{k}_+ \cO_{T_0})\cdot e^{\sum_{i=1}^m \lambda_i w_i}\right)
= \FL^{\,-\lambda_1,\ldots,-\lambda_m}_{w_1,\ldots,w_m}\left(\widetilde{k}_+\cO_{T_0}\right).
$$
We are thus left to show that the latter module equals $\widehat{\cM}_{\widetilde{A}}$.
In order to do so,
notice that the $\cD_{T_0}$-module $\cO_{T_0}$ can be written as a quotient of $\cD_{T_0}$.
The natural choice would be to mod out the left ideal generated by $(y_k\partial_{y_k})_{k=0,\ldots,n}$,
however, we will rather write
\begin{equation}
\label{eq:ChoiceOT0}
\cO_{T_0} =
\frac{\cD_{T_0}}{(y_0\partial_{y_0})+(y_k \partial_{y_k}+1)_{k=1,\ldots,n}},
\end{equation}
which we abbreviate as $\cO_{T_0}\cdot \prod_{k=1}^n y_k^{-1}$.
Now notice that $\widetilde{k}$ is a closed embedding, hence a calculation similar to
the proof of \cite[proposition 2.1]{SchulWalth2}, using the $(\cD_{T_0},\widetilde{k}^{-1}\cD_{\dC_\tau^*\times W'})$-transfer bimodule $\cD_{T_0\rightarrow \dC^*_\tau\times W'}$ shows that
the direct image $\widetilde{k}_+ \cO_{T_0}$ is given by

$$
\widetilde{k}_+ \cO_{T_0}
=
\frac{\cD_{\dC_\tau^*\times W'}}{\left(\prod_{i:l_i <0} (w_0^{-1} w_i)^{-l_i} - \prod_{i:l_i>0 } (w_0^{-1} w_i)^{l_i}\right)_{\underline{l}\in\dL}
+\left(\sum_{i=1}^m a_{ki} \partial_{w_i}w_i\right)_{k=1,\ldots,n}+\left(w_0\partial_{w_0}+\sum_{i=1}^m \partial_{w_i}w_i \right)}.
$$
Now as $w_0=-\tau$ and $\partial_{\lambda_i}=-w_i$ in $\FL^{\;-\lambda_1,\ldots,-\lambda_m}_{w_1,\ldots,w_m} \widetilde{k}_+\cO_{T_0}$, we obtain that
the latter module equals
$$
\frac{\cD_{\dC_\tau^*\times W}}{\left(\prod_{i:l_i <0} (\tau^{-1} \partial_{\lambda_i})^{-l_i} - \prod_{i:l_i>0 } (\tau^{-1} \partial_{\lambda_i})^{l_i}\right)_{\underline{l}\in\dL}
+\left(\sum_{i=1}^m a_{ki} \lambda_i \partial_{\lambda_i}\right)_{k=1,\ldots,n}+\left(w_0\partial_{w_0}-\sum_{i=1}^m \lambda_i\partial_{\lambda_i}\right)}^.
$$
so that finally
$$
\begin{array}{rcl}
\FL^{\;-\lambda_1,\ldots,-\lambda_m}_{w_1,\ldots,w_m} \widetilde{k}_+\cO_{T_0}
& = &
\frac{\cD_{\widehat{V}}[\tau^{-1}]}{\tau^{\overline{l}}\left(\prod_{i:l_i <0} \partial_{\lambda_i}^{-l_i} - \prod_{i:l_i>0 } \partial_{\lambda_i}^{l_i}\right)_{\underline{l}\in\dL}
+\left(\sum_{i=1}^m a_{ki} \lambda_i \partial_{\lambda_i}\right)_{k=1,\ldots,n}+\left(\tau\partial_\tau-\sum_{i=1}^m \lambda_i\partial_{\lambda_i}\right)} \\ \\
& = & \widehat{\cM}^{(1,\underline{0})}_{\widetilde{A}} = \widehat{\cM}_{\widetilde{A}}.
\end{array}
$$
\end{proof}
In the following proposition, we comment upon the more general case where
the vectors $\underline{a}_1,\ldots,\underline{a}_m$ only generate $N_\dQ$ over $\dQ$.
Let as before $A=(\underline{a}_1,\ldots,\underline{a}_m)$ where $\underline{a}_i$ are seen
as vectors in $\dZ^n$. Then it is a well-known fact that $A$ can be factorized as
$ B_1 \cdot C \cdot B_2$ where $B_1$ resp. $B_2$ is in $\textup{Gl}(n, \dZ)$ resp. $\textup{Gl}(m,\dZ)$ and $C$ has the form
$$
\left(\begin{matrix} e_1 & & &\\ & \ddots &  & 0\\ & & e_n & \end{matrix} \right) = \left(\begin{matrix} e_1 & & \\ & \ddots &  \\ & & e_n  \end{matrix} \right) \cdot \left(\begin{matrix} 1 & & &\\ & \ddots &  & 0\\ & & 1 & \end{matrix} \right) = D \cdot E
$$
where $e_i$ are natural numbers called elementary divisors. Set $A':= E \cdot B_2$, then $A=B_1\cdot D \cdot A'$ and the columns of $A'$ generate $N$ over $\dZ$.
\begin{proposition}
\label{prop:NnotGenerated}
We have the following isomorphism
\[
\FL_t^\tau(\cH^0 \varphi_+\cO_{S_0 \times W})[\tau^{-1}] \simeq
\bigoplus_{
\underline{j}\in I_n}
\widehat{\cM}^{(1,j_1/e_1, \ldots , j_n/e_n)}_{\widetilde{A}}.
\]
where $\underline{j}=(j_1,\ldots,j_n)\in \dN^n$ and $I_n=\prod_{k=1}^n ([0,e_k-1]\cap \dN) \subset \dN^n$.
\end{proposition}
\begin{proof}
First notice that the morphism $\varphi$ can be factorized into $\varphi'\circ (\Phi\times \id_{S_1})$,
where $\Phi$ is the  automorphism of $S_0$ defined by $B_1\in\textup{Gl}(n,\dZ)$. Hence
$\varphi_+\cO_{S_0 \times W} = \varphi'_+\cO_{S_0 \times W}$, so that
we can assume that $B_1 = \id_{\dZ^n}$, i.e., that $A = D \cdot A'$. Now one checks that the arguments
in the proof of theorem \ref{theo:GM-GKZUp} showing that
$\FL^{\lambda_1,\ldots,\lambda_m}_{w_1,\ldots,w_m}\left(\widetilde{k}_+\cO_{T_0}\right) \simeq \FL_t^\tau(\cH^0 \varphi_+\cO_{S_0 \times W})[\tau^{-1}]$
are still valid under the more general hypothesis that $A=D\cdot A'$ where only the columns of $A'$ do generate $N$ over $\dZ$.
Hence we need to compute the module $\FL^{\lambda_1,\ldots,\lambda_m}_{w_1,\ldots,w_m}\left(\widetilde{k}_+\cO_{T_0}\right)$.

The factorization of $A$ corresponds to a factorization $\widetilde{k}=\widetilde{k}'\circ c$, where
$c: (y_0, y_1, \ldots ,y_n) \mapsto (y_0,y_1^{e_1}, \ldots y_n^{e_n})$ is a covering map and $\widetilde{k}'$
is a closed embedding defined by the matrix $A'$.
Let us first compute the direct image of $\cO_{T_0}$ under $c$. To do so, we look at the one-dimensional case, i.e. a map $c_k: y_k \mapsto y_k^{e_k}$. We have
\[
c_{k,+} \cO_{\dC^*} \simeq c_{k,+} \cD_{\dC^*}/(y_k \partial_{y_k}) \simeq \bigoplus_{j=0}^{e_k-1} \cD_{\dC^*}/(y_k \partial_{y_k} + 1 - j/e_k),
\]
and moreover $c_+ \cO_{T_0} = \cO_{\dC^*} \boxtimes c_{1,+} \cO_{\dC^*} \boxtimes \ldots \boxtimes c_{n,+} \cO_{\dC^*}$ so that we get
\[
c_+ \cO_{T_0} \simeq
\bigoplus_{
\underline{j}\in I_n}
\frac{\cD_{T_0}}{y_0\partial_{y_0}+( y_k \partial_{y_k} +1 -j_k/e_k)_{k=1,\ldots,n}}.
\]
In the next step we compute the direct image under the closed embedding $\widetilde{k}'$. Similar as above, we obtain for the direct image
\begin{align}
&\widetilde{k}'_+\left(\frac{\cD_{T_0}}{y_0\partial_{y_0}+( y_k \partial_{y_k} +1 -j_k/e_k)_{k=1,\ldots,n}}\right) \notag \\
= &\frac{\cD_{\dC_\tau^*\times W'}}{\left(\prod_{i:l_i <0} (w_0^{-1} w_i)^{-l_i} - \prod_{i:l_i>0 } (w_0^{-1} w_i)^{l_i}\right)_{\underline{l}\in\dL}
+\left(\sum_{i=1}^m a_{ki} \partial_{w_i}w_i -j_k /e_k \right)_{k=1,\ldots,n}+\left(w_0\partial_{w_0}+\sum_{i=1}^m \partial_{w_i}w_i \right)}\notag\\
\end{align}
The Fourier-Laplace transformation in the variables $w_1, \ldots, w_m$ yields
\[
\FL^{\;-\lambda_1,\ldots,-\lambda_m}_{w_1,\ldots,w_m}
\left(\widetilde{k}'_+\left(\frac{\cD_{T_0}}{y_0\partial_{y_0}+( y_k \partial_{y_k} +1 -j_k/e_k)_{k=1,\ldots,n}}\right)\right)= \widehat{\cM}^{\beta}_{\widetilde{A}}
\]
where $\beta =(1,j_1/e_1, \ldots , j_n/e_n)$. Taking the direct sum this gives
\[
\FL_t^\tau(\cH^0 \varphi_+\cO_{S_0 \times W})[\tau^{-1}] \simeq \FL^{\;-\lambda_1,\ldots,-\lambda_m}_{w_1,\ldots,w_m}\left(\widetilde{k}_+\cO_{T_0}\right)\\
=
\bigoplus_{\underline{j}\in I_n}
\widehat{\cM}^{(1,j_1/e_1, \ldots , j_n/e_n)}_{\widetilde{A}}\; .
\]
\end{proof}
In the following proposition, we collect some properties of the hypergeometric $\cD$-modules introduced above.
An important tool will be the notion of non-degeneracy of a Laurent polynomial, recall
(see, e.g., \cite{Kouch} or \cite{Adolphson}) that $f:(\dC^*)^t \rightarrow \dC$, $f=\mu_1 \underline{x}^{\underline{b}_s}+\ldots+\mu_s \underline{x}^{\underline{b}_s}$
is called non-degenerate if for any proper face $\tau$ of $\Conv(\underline{0},\underline{b}_1, \ldots, \underline{b}_s)\subset \dR^t$ not containing
$\underline{0}$, $f_\tau=\sum_{\underline{b}_i\in\tau} \mu_i \underline{x}^{\underline{b}_i}$ has no critical points in $(\dC^*)^t$.
\begin{proposition}
\label{prop:ResultsClassicalGKZ}
\begin{enumerate}
\item
$\cM^\beta_{\widetilde{A}}$ (resp. $\widehat{\cM}^\beta_{\widetilde{A}}$, $\widehat{\cM}^{\beta,loc}_{\widetilde{A}}$) is a coherent and holonomic
$\cD_{V}$-module (resp. $\cD_{\widehat{V}}$-module, $\cD_{\widehat{T}}$-module).
Moreover, $\cM^\beta_{\widetilde{A}}$ has only regular singularities, included at infinity.
\item
Let as before
$F:S_0\times W \rightarrow \dC_t$, $(y_1,\ldots,y_n,\lambda_1,\ldots,\lambda_m)\mapsto\sum_{i=1}^m \lambda_i \cdot \underline{y}^{\underline{a}_i}$. Define
$$
S_1^0:=\{(\lambda_1,\ldots,\lambda_m)\in S_1\,|\,F(-,\underline{\lambda}) \textup{ is non-degenerate with respect to its Newton polyhedron}\}.
$$
Moreover, consider the following extended family
$$
\begin{array}{rcl}
\widetilde{F}:T_0\times V & \longrightarrow & \dC \\ \\
((y_0,y_1,\ldots,y_n),(\lambda_0,\lambda_1,\ldots,\lambda_m)) & \longmapsto &
y_0\cdot\left(\lambda_0+\sum_{i=1}^m\lambda_i\cdot \underline{y}^{\underline{a}_i}\right)
\end{array}
$$
and put
$$
V^0:=\{(\lambda_0,\lambda_1,\ldots,\lambda_m)\in \dC\times S_1\,|\,\widetilde{F}(-,\underline{\lambda}) \textup{ is non-degenerate with respect to its Newton polyhedron}\}.
$$
Both $S_1^0$ and $V^0$ are Zariski open subspaces of $S_1$ resp. $\dC\times S_1$
(as well as of $W$ resp. $V$). We have
\begin{enumerate}
\item
The characteristic variety of the restriction of $\cM^\beta_{\widetilde{A}}$ to $V^0$
is the zero section of $T^*V^0$, i.e., $\cM^\beta_{\widetilde{A}}$ is smooth on $V^0$.
\item
Suppose that $\underline{a}_1,\ldots,\underline{a}_m$ are defined by toric data and moreover, that the the projective variety
$\XSigA$ is genuine Fano, i.e., that its anti-canonical class is ample (and not only nef).
Then
$V\backslash V^0 \subset \Delta(F) \cup \bigcup_{i=1}^m\{\lambda_i=0\}\subset V$, where
$$
\Delta(F):=\left\{(-t,\lambda_1,\ldots,\lambda_m)\in V\,|\, F(-,\underline{\lambda})^{-1}(t)\textup{ is singular }\right\}
$$
is the discriminant of the family $-F$.
\item
The restriction of $\widehat{\cM}^{\beta,loc}_{\widetilde{A}}$
to $\dC_\tau^*\times S_1^0$ is smooth.
\end{enumerate}
\item
Suppose that $\underline{a}_1,\ldots,\underline{a}_m$ are defined by toric data. Then the
generic rank of both
$\cM^\beta_{\widetilde{A}}$ and $\widehat{\cM}^\beta_{\widetilde{A}}$ is equal to
$n!\cdot \vol(\Conv(\underline{a}_1,\ldots,\underline{a}_m)) )= (n+1)!\cdot \vol(\Conv(\widetilde{\underline{0}},\widetilde{\underline{a}}_1,\ldots,\widetilde{\underline{a}}_m))$, where
the volume of a hypercube $[0,1]^t\subset \dR^t$ is normalized to one,
and where $\widetilde{\underline{0}}$ denotes the origin in $\dZ^{n+1}$.
\end{enumerate}
\end{proposition}
Before entering into the proof, we need the following lemma.
\begin{lemma}\label{lem:FanoNonDeg}
Suppose that $\underline{a}_1,\ldots,\underline{a}_m$ are the primitive integral
generators of the rays of a fan $\Sigma_A$ defining a smooth toric Fano manifold $\XSigA$ . Then the family
$F:S_0\times S_1\rightarrow \dC_t$ is non-degenerate for any $(\lambda_1,\ldots,\lambda_m)\in S_1$.
\end{lemma}
\begin{proof}
If $\XSigA$ is Fano, then it is well known (see, e.g., \cite[lemma 3.2.1]{CK}) that $\Sigma_A$ is the fan over the proper
faces of $\Conv(\underline{a}_1, \ldots , \underline{a}_m)$. Let $\tau$ be a face of codimension $n+1-s$ and $\sigma$ the corresponding $s$-dimensional cone over $\tau$. As $\Sigma_A$ is regular, the primitive generators $\underline{a}_{\tau_1}, \ldots ,\underline{a}_{\tau_s}$ are linearly independent. We have to check that
$$
F_\tau(\underline{\lambda},\underline{y}) = \lambda_{\tau_1} \underline{y}^{\underline{a}_{\tau_1}}+ \ldots + \lambda_{\tau_s} \underline{y}^{\underline{a}_{\tau_s}}
$$
has no singularities on $S_0$ for any $(\lambda_{\tau_1},\ldots , \lambda_{\tau_s}) \in (\dC^*)^s$.
The critical point equations $y_k \partial_{y_k}F_\tau = 0$ can be written in matrix notation as
$$
\left(
\begin{matrix}
(a_{\tau_1})_1  & (a_{\tau_2})_1 & \hdots & (a_{\tau_s})_1 \\
\vdots & \vdots & & \vdots \\
(a_{\tau_1})_n  & (a_{\tau_2})_n & \hdots & (a_{\tau_s})_n
\end{matrix}
\right)
\cdot
\left(
\begin{matrix}
\lambda_{\tau_1} \cdot \underline{y}^{\underline{a}_{\tau_1}} \\
\vdots \\
\lambda_{\tau_s} \cdot \underline{y}^{\underline{a}_{\tau_s}}
\end{matrix}
\right)
= 0 \, .
$$
This matrix has maximal rank and therefore can only have the trivial solution, contradicting the fact that
$(\lambda_{\tau_1},\ldots,\lambda_{\tau_s})\in (\dC^*)^s$ and $\underline{y}\in S_0$. Hence there is
no solution at all and $F$ is non-degenerate for all $\underline{\lambda} \in S_1$.
\end{proof}
\begin{proof}[Proof of the proposition]
\begin{enumerate}
\item
The holonomicity statement for $\cM^\beta_{\widetilde{A}}$ is \cite[Theorem 3.9]{Adolphson}
(or even the older result \cite[Theorem 1]{GKZ1}, as the vectors $\widetilde{\underline{a}}_0,
\widetilde{\underline{a}}_1, \ldots, \widetilde{\underline{a}}_m$ lie in an affine hyperplane
of $\widetilde{N}$). Then also
$\widehat{\cM}^\beta_{\widetilde{A}}$ and $\widehat{\cM}^{\beta,loc}_{\widetilde{A}}$
are holonomic as this property is preserved under (partial) Fourier-Laplace transformation.
The regularity of $\cM^\beta_{\widetilde{A}}$ has been shown, e.g., in \cite[section 6]{HottaEq}.

\item
\begin{enumerate}
\item
This is shown in \cite[lemma 3.3]{Adolphson}.
\item
By lemma \ref{lem:FanoNonDeg},
$\widetilde{F}_\tau:=\sum_{i:\widetilde{\underline{a}}_i\in\tau} \lambda_i \prod_{k=0}^n y_k^{\widetilde{a}_{ki}}$
can have a critical point in $T_0$ only in the case that $\tau=
\Conv(\widetilde{\underline{a}}_0,\widetilde{\underline{a}}_1,
\ldots,\widetilde{\underline{a}}_m)$, i.e., we have
the following system of equations
$$
\begin{array}{rcl}
y_0\partial_{y_0} \widetilde{F} & = & y_0\left(\lambda_0+ \sum_{i=1}^m\lambda_i\cdot\prod_{k=1}^n y_k^{a_{ki}}\right) \stackrel{!}{=} 0, \\
\bigg(y_k\partial_{y_k} \widetilde{F} & = & y_0\sum_{i=1}^m\lambda_i\cdot a_{ki}\prod_{k=1}^n y_k^{a_{ki}}\stackrel{!}{=}0\bigg)_{k=1,\ldots,n}.
\end{array}
$$
The first equation yields $\lambda_0=-t$, where $t$ denotes the value of the family $F$, and the second one
is the critical point equation for $F$.
\item
We know that $\car(\widehat{\cM}^\beta_{\widetilde{A}})$ is included in
the variety cut out by the ideal
$$
\left(\sigma(\widehat{\Box}_{\underline{l}})\right)_{\underline{l}\in\dL}+(\sigma(\widehat{Z}_k))_{k=1,\ldots,n}
+\sigma(\widehat{E}).
$$
Write $y$ resp. $\mu_i$ for the cotangent coordinates on $T^*(\dC^*_\tau\times S_1^0)$ corresponding to
$z$ resp. $\lambda_i$. As $\sigma(\widehat{E})=zy+\sum_{i=1}^n \lambda_i \mu_i$, it suffices to show that
the sub-variety of $\dC^*_\tau\times T^* S_1^0$ defined by the ideal
$$
\left(\sigma(\widehat{\Box}_{\underline{l}})\right)_{\underline{l}\in\dL}+(\sigma(\widehat{Z}_k))_{k=1,\ldots,n}
$$
equals the zero section.
Write $\beta=(\beta_0,\beta')$ with $\beta'\in N_\dC$.
Notice that for any $\underline{l}\in \dL$, if  $\overline{l}\neq 0$, then
either $\sigma(\widehat{\Box}_{\underline{l}})$ or
$\sigma(z^{\overline{l}}\widehat{\Box}_{\underline{l}})$ belongs
to $\dC[\mu_1,\ldots,\mu_m]$ and equals the symbol of one of the operators defining $\cM^{\beta'}_A$.
Similarly, if $\overline{l}=0$, then already $\Box_{\underline{l}}$ itself
is independent of $z$ and equal to an operator from $\cM^{\beta'}_A$.
This shows that \cite[lemma 3.1 to lemma 3.3]{Adolphson} holds for
$\widehat{\cM}^\beta_{\widetilde{A}}$, and hence $\widehat{\cM}^\beta_{\widetilde{A}}$ is
smooth on $\dC_\tau^*\times S_1^0$.

\end{enumerate}
\item
For the $\cD_V$-module $\cM^\beta_{\widetilde{A}}$ this is \cite[corollary 5.21]{Adolphson} as $\Spec\dC[\dN \widetilde{A}]$ is Cohen-Macaulay by proposition \ref{prop:ClosedEmb_NormalCMGorenstein}, 2.,
notice that the Cohen-Macaulay condition is needed only for the ring $\Spec\dC[\dN \widetilde{A}]$, not for any of its subrings as the only face $\tau$ occurring in loc.cit.
that does not contain the origin is the one spanned by the vectors $\widetilde{\underline{a}}_0,\widetilde{\underline{a}}_1\ldots,\widetilde{\underline{a}}_m$.

Similarly, \cite[corollary 5.21]{Adolphson} shows that the generic rank of
$\cM^{\beta'}_A$ equals $n!\cdot\vol(\Conv(\underline{a}_1,\ldots,\underline{a}_m))$:
Here we have to use the fact that all cones $\sigma\in \Sigma_A$ are smooth, so that
the semigroup rings generated by their primitive integral generators are normal and Cohen-Macaulay.
Now it follows from the calculation of the characteristic variety from 2(c)
that this is then also the generic rank of $\widehat{\cM}^\beta_{\widetilde{A}}$.
\end{enumerate}
\end{proof}

For later purpose, we need a precise statement on the regularity resp. irregularity of the module
$\widehat{\cM}^{\beta}_{\widetilde{A}}$, at least in the case of main interest
where $\underline{a}_1,\ldots,\underline{a}_m$ are defined by toric data.
As a preliminary step, we show in the following proposition a finiteness result for
the singular locus of $\cM^\beta_{\widetilde{A}}$.
\begin{proposition}\label{prop:DiscFinite}
Suppose that $\underline{a}_1,\ldots,\underline{a}_m$ are defined by toric data.
Let $p:V\rightarrow W$ be the projection forgetting the first component. Then for any
$\underline{\lambda}=(\lambda_1,\ldots,\lambda_m)\in S_1^0$, there is a small
analytic neighborhood $U_{\underline{\lambda}} \subset S_1^{0,an}$ such that
the restriction
$$
p_{|\Delta(F)^{an}\cap p^{-1}(U_{\underline{\lambda}})}:
\Delta(F)^{an}\cap p^{-1}(U_{\underline{\lambda}}) \longrightarrow U_{\underline{\lambda}}
$$
is finite, i.e., proper with finite fibres. In particular $p_{|\Delta(F)\cap p^{-1}(S_1^0)}:\Delta(F)\cap p^{-1}(S_1^0) \rightarrow S_1^0$ is finite.
\end{proposition}
\begin{proof}
Write $P_{\underline{\lambda}}$ for the restriction $p_{|\Delta(F)^{an}\cap p^{-1}(U_{\underline{\lambda}})}$.
The quasi-finiteness of $P_{\underline{\lambda}}$ is obvious, as for any $\underline{\lambda}\in S^0_1$,
$F(-,\underline{\lambda})$ has only finitely many critical values. Hence we need to show that $P_{\underline{\lambda}}$
is proper. Take any compact subset $K$ in $U_{\underline{\lambda}}$. Suppose that
$P_{\underline{\lambda}}^{-1}(K)$ is not compact, then it must be unbounded in $V\cong\dC^{m+1}$ for the standard metric.
Hence there is a sequence $(\lambda_0^{(i)},\underline{\lambda}^{(i)})\in P^{-1}_{\underline{\lambda}}(K)$ with
$\lim_{i\rightarrow \infty} |\lambda_0^{(i)}| =\infty$, as $K$ is closed and bounded in $W\cong\dC^m$. Consider the projection $\pi:V\rightarrow \dP(V)=\Proj \dC[\lambda_0,\lambda_1,\ldots,\lambda_m]$,
then (possibly after passing to a subsequence), we have $\lim_{i\rightarrow \infty}\pi(\lambda_0^{(i)},\underline{\lambda}^{(i)}) = (1:0:\ldots:0)$.

In order to construct a contradiction, we will need to consider a partial compactification of the family $F$, or rather
of the morphism $\varphi:S_0\times S_1\rightarrow \dC_t\times S_1$. This is done as follows (see, e.g., \cite{DF1} and \cite{Kh}): Write $X_B$ for the
projective toric variety defined by the polytope $\Conv(\underline{a}_1,\ldots,\underline{a}_m)$
(under the assumption that $\XSigA$ is weak Fano, this is a reflexive polytope in the sense of \cite{Bat3}) then
$X_B$ embeds into $\dP(V')$ and contains the closure of the image of the morphism $k$ from proposition \ref{prop:ClosedEmb_NormalCMGorenstein}.
Write $Z=\{\sum_{i=0}^m \lambda_i \cdot w_i =0\}\subset \dP(W')\times\dP(W)$ for the universal hypersurface
and put $Z_B:=\left(X_B\times\dP(W)\right)\cap Z$. Consider the map
$\pi:X_B\times (\dC_t\times S_1) \rightarrow X_B\times \dP(W)$, let $\widetilde{Z}_B:=\pi^{-1}(Z_B)$,
and write $\phi$ for the restriction of the projection $X_B\times (\dC_t\times S_1)\twoheadrightarrow \dC_t\times S_1$
to $\widetilde{Z}_B$. Then $\phi$ is proper, and restricts to $\varphi$ on $S_0\times S_1\cong \Gamma_\varphi\subset \widetilde{Z}_B$.
There is a natural stratification
of $X_B$ by torus orbits and this gives a product stratification on $X_B\times(\dC_t\times S_1)$.
Now consider the restriction
$\phi'$ of $\phi$ to $\widetilde{Z}'_B:=\phi^{-1}(\dC_t \times S_1^0)$, then one checks that the non-degeneracy of $F$ on $S_1^0$ is equivalent
to the fact that $Z$ cuts all strata of $(X_B\backslash S_0) \times (\dC_t\times S_1^0)$ transversal.
Hence we have a natural Whitney stratification $\Sigma$ on (the analytic space associated to) $\widetilde{Z}'_B$.
If we write
$\textup{Crit}_\Sigma(\phi')$ for the $\Sigma$-stratified critical locus of $\phi'$, i.e.,
$\textup{Crit}_\Sigma(\phi'):=\bigcup_{\Sigma_\alpha\in\Sigma} \textup{Crit}(\phi'_{|\Sigma_\alpha})$, then we have
$\textup{Crit}_\Sigma(\phi')=\textup{Crit}(\varphi')$, where $\varphi':=\varphi_{|S_0\times S_1^0}$.
On the other hand, Whitney's (a)-condition implies that $\textup{Crit}_\Sigma(\phi')$ is closed
in $\widetilde{Z}'_B$, and so is $\textup{Crit}(\varphi')$.

Now consider the above sequence $(\lambda_0^{(i)},\underline{\lambda}^{(i)})\in P_{\underline{\lambda}}^{-1}(K) \subset \Delta(F)^{an}$,
then the fact that the projection from the critical locus of $\varphi$ to the discriminant is onto shows that there is a sequence $((w_0^{(i)},\underline{w}^{(i)}),(\lambda_0^{(i)},\underline{\lambda}^{(i)})\in\textup{Crit}(\varphi')\subset S_0\times K$
projecting under $\varphi'$ to $(\lambda_0^{(i)},\underline{\lambda}^{(i)})$. Consider the first component
of the sequence $\pi((w_0^{(i)},\underline{w}^{(i)}),(\lambda_0^{(i)},\underline{\lambda}^{(i)}))$, then this is
a sequence $(w_0^{(i)},\underline{w}^{(i)})$ in $X_B$ which converges (after passing possibly again to a subsequence)
to a limit $(0:w_1^{\textup{lim}},\ldots,w_m^{\textup{lim}})$ (this is forced
by the incidence relation $\sum_{i=0}^m w_i\lambda_i=0$), in other words, this limit lies in $X_B\backslash S_0$.
However, we know that $\lim_{i\rightarrow \infty}
((w_0^{(i)},\underline{w}^{(i)}),(\lambda_0^{(i)},\underline{\lambda}^{(i)})$ exists in $\textup{Crit}_\Sigma(\phi')$
as the latter space is closed. This is a contradiction, as we have seen that $\phi$ is non-singular outside $S_0\times(\dC_t\times S_1)$, i.e.,
that $\textup{Crit}_\Sigma(\phi')=\textup{Crit}(\varphi') \subset S_0\times S_1^0$.
\end{proof}
Now the regularity result that we will need later is the following.
\begin{lemma}\label{lem:MhatRegInfty}
Consider $\widehat{\cM}^{\beta}_{\widetilde{A}}$ as a $\cD_{\dP^1_z\times \overline{W}}$-module, where
$\overline{W}$ is a smooth projective compactification of $W$. Then $\widehat{\cM}^{\beta}_{\widetilde{A}}$ is regular outside
$(\{z=0\}\times \overline{W})\cup(\dP^1_z\times(\overline{W}\backslash S_1^0))$.
\end{lemma}
\begin{proof}
It suffices to show that any $\underline{\lambda}=(\lambda_1,\ldots,\lambda_m)\subset S_1^0$
has a small analytic neighborhood $U_{\underline{\lambda}} \subset S_1^{0,an}$ such that
the partial analytization $\widehat{\cM}^{\beta,loc}_{\widetilde{A}}\otimes_{\cO_{\dC^*_\tau\times S_1}} \cO^{an}_{U_{\underline{\lambda}}}[\tau,\tau^{-1}]$
is regular on $\dC_\tau \times U_{\underline{\lambda}}$ (but not at $\tau=\infty$). This is precisely the statement
of \cite[theorem 1.11 (1)]{DS}, taking into account the regularity of $\cM^{\beta}_{\widetilde{A}}$ (i.e., proposition \ref{prop:ResultsClassicalGKZ}, 1.),
the fact that on $\dC_{\lambda_0} \times U_{\underline{\lambda}}$, the singular locus of $\cM^{\beta}_{\widetilde{A}}$ coincides with $\Delta(F)$ (see the proof of
proposition \ref{prop:ResultsClassicalGKZ}, 2(b))
as well as the last proposition (notice that the non-characteristic assumption
in loc.cit. is satisfied, see, e.g., \cite[page 281]{Ph1}).
\end{proof}

\subsection{Brieskorn lattices}
\label{subsec:Brieskorn}

The next step is to study natural lattices that exist in
$G$ and in $\widehat{\cM}^\beta_{\widetilde{A}}$. To avoid endless
repetition of hypotheses, we will assume throughout this subsection that our vectors
$\underline{a}_1,\ldots,\underline{a}_m$ are defined by toric data.
In order to discuss lattices in $\widehat{\cM}^\beta_{\widetilde{A}}$, we start with definition.
\begin{definition}\label{def:LatticeGKZ}
\begin{enumerate}
\item
Consider the ring
$$
R := \dC[\lambda^\pm_1,\ldots,\lambda^\pm_m,z]\langle z \partial_{\lambda_1},\ldots, z \partial_{\lambda_m},z^2\partial_z\rangle,
$$
i.e. the quotient of the free associative $\dC[\lambda^\pm_1,\ldots,\lambda^\pm_m,z]$-algebra generated by
$z \partial_{\lambda_1}, \ldots, z \partial_{\lambda_m}, z^2\partial_z$ by the left ideal generated
by the relations
$$
\begin{array}{c}
[z \partial_{\lambda_i}, z] = 0, \quad [z \partial_{\lambda_i}, \lambda_j]=\delta_{ij} z ,
\quad [z^2\partial_z,\lambda_i]=0,\quad \\ \\
{[}z^2\partial_z, z]= z^2, \quad [z \partial_{\lambda_i}, z \partial_{\lambda_j} ]= 0, \quad [z^2\partial_z, z \partial_{\lambda_i}] = z\cdot z \partial_{\lambda_i}.
\end{array}
$$
Write $\cR$ for the associated sheaf of quasi-coherent $\cO_{\dC_z\times S_1}$-algebras
which restricts to $\cD_{\dC^*_\tau\times S_1}$ on $\{(z\neq 0\}$.
We also consider the subring $R':=\dC[\lambda^\pm_1,\ldots,\lambda^\pm_m,z]\langle z \partial_{\lambda_1},\ldots, z \partial_{\lambda_m}\rangle$
of $R$, and the associated sheaf $\cR'$. The
inclusion $\cR'\hookrightarrow \cR$ induces a functor from the category of $\cR$-modules
to the category of $\cR'$-modules, which we denote by $\textup{For}_{z^2\partial_z}$ (``forgetting the $z^2\partial_z$-structure'').
\item
Choose $\beta\in \widetilde{N}_\dC$, consider the ideal $\cI:=\cR(\widehat{\Box}'_{\underline{l}})_{\underline{l}\in\dL}+\cR(z\cdot\widehat{Z}_k)_{k=1,\ldots,n}+\cR(z\cdot\widehat{E})$ in $\cR$
and write ${_0\!}\widehat{\cM}^{\beta,loc}_{\widetilde{A}}$ for the quotient $\cR/\cI$. We have
$\textup{For}_{z^2\partial_z}({_0\!}\widehat{\cM}^{\beta,loc}_{\widetilde{A}}) = \cR'/
((\widehat{\Box}'_{\underline{l}})_{\underline{l}\in\dL}+(z\cdot\widehat{Z}_k)_{k=1,\ldots,n})$,
and the restriction of ${_0\!}\widehat{\cM}^{\beta,loc}_{\widetilde{A}}$ to $\dC_\tau^*\times S_1$ equals $\widehat{\cM}^{\beta,loc}_{\widetilde{A}}$.
Again we put $\BL:={_0\!}\widehat{\cM}^{(1,\underline{0}),loc}_{\widetilde{A}}$.
\end{enumerate}
\end{definition}
\begin{corollary}\label{cor:IdentBrieskorn}
Consider the restriction of the isomorphism $\phi$ from theorem \ref{theo:GM-GKZUp} to $\dC^*_\tau\times S_1$.
\begin{enumerate}
\item
$\phi$ sends the class of the section $1$ in $\widehat{\cM}^{loc}_{\widetilde{A}}$ to
class of the (relative) volume form $\omega_0:=dy_1/y_1\wedge\ldots\wedge dy_n/y_n \in \Omega^n_{S_0\times S_1/S_1}$.
\item
The morphism
$\phi$ maps
$\BL$ isomorphically to
$$
G_0:=\frac{\pi_*\Omega^n_{S_0\times S_1 / S_1}[z]}{\left(zd-dF\wedge\right)\pi_*\Omega^{n-1}_{S_0\times S_1 / S_1}[z]}.
$$
\end{enumerate}
\end{corollary}
\begin{proof}
\begin{enumerate}
\item
Following the identifications in the proof of theorem \ref{theo:GM-GKZUp}, this is evident,
if one takes into account that due to the choice in formula \eqref{eq:ChoiceOT0}, we have actually computed
$$
G_{|\dC_\tau^*\times S_1}=
\FL^\tau_t\left(\cH^0\varphi_+\cO_{S_0\times S_1}\frac{1}{y_1\cdot\ldots\cdot y_n}\right)
=\FL^\tau_t\left(\cH^0\varphi_+\frac{\cD_{S_0\times S_1}}{\cD_{S_0\times S_1}(y_k\partial_{y_k}+1)_{k=1,\ldots,n}}\right)
$$
\item
First notice that due to 1. and the formulas in lemma \ref{lem:DirectImageSimplified}, we have
$\phi\left(\BL\right) \subset G_0$. To see that it is surjective, take
any representative $s=\sum_{i\geq 0} \omega^{(i)} z^i$ of a class in $G_0$.
As an element of $G$,
$s$ has a unique preimage under $\phi$, which is an operator $P\in \widehat{\cM}^{loc}_{\widetilde{A}}$ and
we have to show that actually $P\in\BL$. By linearity of $\phi$, it is sufficient to do it for the case where $\omega^{(0)}\neq 0$.
There is a minimal $k\in \dN$ such that $z^kP\in
\BL$, and then the class of $z^kP$
in $\BL/z\cdot \BL$ does not vanish. Suppose that $k > 0$, then
the class of $\phi(z^kP)=z^k s$ vanishes in $G_0/z G_0$, which contradicts the next lemma. Hence
$k=0$ and $P\in\BL$.
\end{enumerate}
\end{proof}
\begin{lemma}\label{lem:BatyrevRing}
\begin{enumerate}
\item
The quotient
$\BL/z\cdot \BL$ is the sheaf of commutative
$\cO_{S_1}$-algebras associated to
$$
\frac{\dC[\lambda_1^\pm,\ldots,\lambda_m^\pm,\mu_1,\ldots,\mu_m]}{(\prod_{l_i<0} \mu_i^{-l_i} - \prod_{l_i>0} \mu_i^{l_i} )_{\underline{l}\in\dL} +
(\sum_{i=1}^m a_{ki}\lambda_i \mu_i)_{k=1,\ldots,n}}
$$
\item
The induced map
$$
[\phi]: \BL/z\cdot \BL \longrightarrow G_0/zG_0 \cong \pi_*\Omega^n_{S_0\times S_1/S_1}/d_{\underline{y}}F\wedge \pi_*\Omega^{n-1}_{S_0\times S_1/S_1}
$$
is an isomorphism.
\end{enumerate}
\end{lemma}
\begin{proof}
\begin{enumerate}
\item
Letting $\mu_i$ be the class of $z\partial_{\lambda_i}$ in $\BL/z\cdot \BL$,
we see that the commutator $[\mu_i,\lambda_i]$ vanishes in this quotient.
\item
This can be shown along the lines of \cite[theorem 8.4]{Bat2}. Namely, consider the
morphism of $\dC[\lambda_1^\pm,\ldots,\lambda_1^\pm]$-algebras
$$
\begin{array}{rcl}
\psi:\dC[\lambda_1^\pm,\ldots,\lambda_m^\pm,\mu_1,\ldots,\mu_m] & \longrightarrow & \dC[\lambda_1^\pm,\ldots,\lambda^\pm_m,y_1^\pm,\ldots,y_n^\pm] \\ \\
\mu_i & \longmapsto & \underline{y}^{\underline{a}_i}
\end{array}
$$
From the completeness and smoothness of $\Sigma_A$ we deduce that $\psi$ is surjective.
Moreover, we have $\mathit{ker}(\psi)=
(\prod_{l_i<0} \mu_i^{-l_i} - \prod_{l_i>0} \mu_i^{l_i} )_{\underline{l}\in\dL}$ (for a proof, see, e.g., \cite[theorem 7.3]{MillSturm}), and obviously
$\psi(\sum_{i=1}^m a_{ki}\lambda_i \mu_i) = y_k\partial_{y_k} F$ for all $k=1,\ldots,n$.
One easily checks that the induced map
$$
\psi:
\dfrac{\dC[\lambda_1^\pm,\ldots,\lambda_m^\pm,\mu_1,\ldots,\mu_m]}{(\prod_{l_i<0} \mu_i^{-l_i} - \prod_{l_i>0} \mu_i^{l_i} )_{\underline{l}\in\dL} +
(\sum_{i=1}^m a_{ki}\lambda_i \mu_i)_{k=1,\ldots,n}}
\longrightarrow
\dfrac{\dC[\lambda_1^\pm,\ldots,\lambda^\pm_m,y_1^\pm,\ldots,y_n^\pm]}{(y_k\partial_{y_k} F)_{k=1,\ldots,n}}
$$
coincides with the map $[\phi]$ induced by $\phi$, notice that
$$
\dfrac{\dC[\lambda_1^\pm,\ldots,\lambda^\pm_m,y_1^\pm,\ldots,y_n^\pm]}{(y_k\partial_{y_k} F)_{k=1,\ldots,n}} \cong
\dfrac{\pi_*\Omega^n_{S_0\times S_1/S_1}}{d F\wedge \pi_* \Omega^{n-1}_{S_0\times S_1/S_1}}.
$$
by multiplication with the relative volume form $dy_1/y_1 \wedge\ldots \wedge dy_n/y_n$.
\end{enumerate}
\end{proof}

Following the terminology of \cite{Sa2} and \cite{DS} (going back to \cite{SM}, and, of course, to \cite{Brie}), we call $G_0$ (and, using the last result, also $\BL$)
the (family of) Brieskorn lattice(s) of the morphism $\varphi$.
For the case of a single Laurent polynomial $F_{\underline{\lambda}}:=\varphi(-,\underline{\lambda}):S_0\rightarrow \dC$, it follows
from the results of \cite{Sa2} that the module $\Omega^n_{S_0}[z]/(zd-dF_{\underline{\lambda}}\wedge)\Omega^{n-1}_{S_0}[z]$ is $\dC[z]$-free
provided that $\underline{\lambda}\in S_1^0$, recall that $S_1^0$ denotes
the Zariski open subset of $S_1$ of parameter values $\underline{\lambda}$ such that $F(-,\underline{\lambda})$ is non-degenerate
with respect to its Newton polyhedron. However, this does not directly
extend to a finiteness (and freeness) result for the Brieskorn lattice $G_0$ of the family $\varphi:S_0\times S^0_1\rightarrow \dC_t\times S^0_1$.
We can now prove this freeness using corollary \ref{cor:IdentBrieskorn}.
\begin{theorem}
\label{theo:BrieskornLatticeFree}
The module $\cO_{\dC_z\times S_1^0}\otimes_{\cO_{\dC_z\times S_1}}\BL$ (and hence also the module $\cO_{\dC_z\times S_1^0} \otimes_{\cO_{\dC_z\times S_1}} G_0$) is
$\cO_{\dC_z\times S^0_1}$-locally free.
\end{theorem}
\begin{proof}
The main argument in the proof is very much similar to the proof of proposition \ref{prop:ResultsClassicalGKZ}, 2.c).
It is actually sufficient to show that $\cO_{\dC_z\times S_1^0}\otimes_{\cO_{\dC_z\times S_1}}\BL$ is
$\cO_{\dC_z\times S^0_1}$-coherent. Namely, we know that the restriction $\cO_{S_1^0}
\otimes_{\cO_{S_1}}\left(\BL / z\cdot \BL \right)$ equals the Jacobian algebra of $\varphi_{|S_0\times S_1^0}$, which is
$\cO_{S^0_1}$-locally free of rank equal to the Milnor number of $\varphi_{|S_0\times S_1^0}$, that it, equal
to $n!\cdot\vol(\Conv(\underline{a}_1,\ldots,\underline{a}_m))$, see \cite[th\'eor\`eme 1.16]{Kouch}. Moreover, the restriction
$\cO_{\dC^*_\tau\times S_1^0} \otimes_{\cO_{\dC_z\times S_1}}\BL= \cO_{\dC^*_\tau\times S_1^0}\otimes_{\cO_{\dC^*_\tau\times S_1}}\widehat{\cM}_{\widetilde{A}}$
is locally free of the same rank and equipped with a flat structure, so that $\BL\otimes_{\cO_{\dC_z\times S_1}}\cO_{\dC_z\times S_1^0}$ can only have the same rank everywhere, provided that it is coherent.

It will be sufficient to show the coherence of
$\cN:=\cO_{\dC_z\times S_1^0}\otimes_{\cO_{\dC_z\times S_1}}\textup{For}_{z^2\partial_z}(\BL)$ only, as this is the same
as $\cO_{\dC_z\times S_1^0}\otimes_{\cO_{\dC_z\times S_1}}\BL$ when considered as an $\cO_{\dC_z\times S^0_1}$-module. Let us denote by
$F_\bullet$ the natural filtration on $\cR'$
defined by
$$
F_k \cR':=
\left\{P\in\cR'\,\left|\,P=\sum_{|\underline{\alpha}|\leq k} g_{\underline{\alpha}}(z,\underline{\lambda})(z \partial_{\lambda_1})^{\alpha_1}\cdot\ldots\cdot(z \partial_{\lambda_m})^{\alpha_r}\right.\right\}.
$$
This filtration induces a filtration $F_\bullet$ on $\cN$ which is good, in the sense
that
$F_k \cR'\cdot F_l \cN = F_{k+l} \cN$.
Obviously, for any $k$, $F_k \cN$ is $\cO_{\dC_z\times S^0_1}$-coherent, so that
it suffices to show that the filtration $F_\bullet$ become eventually stationary. The ideal
generated by the symbols of all operators in the ideal defining $\cN$,
that is, by the highest order terms with respect to the filtration $F_\bullet$,
cut out a subvariety of $\dC_z\times T^*S^0_1$, and it suffices to show that this subvariety equals $\dC_z\times S_1^0$, then by the usual
argument the filtration $F_\bullet$ stabilizes for some sufficiently large index.
However, for any of the operators $\widehat{\Box}'_{\underline{l}}$ and  $\widehat{Z}_k$ in $\cI'$, its symbol
with respect to the above filtration $F_\bullet$ is precisely the same as the symbol
of $\widehat{\Box}_{\underline{l}},\widehat{Z}_k$ with respect to the ordinary filtration
on $\widehat{\cM}^{loc}_{\widetilde{A}}$, hence, the same argument as in the proof of
proposition \ref{prop:ResultsClassicalGKZ}, 2.c) (that is, the arguments in
\cite[lemma 3.1 to lemma 3.3]{Adolphson}) shows that the above mentioned subvariety is
the zero section $\dC^*\times S_1^0$.
\end{proof}

\subsection{Duality and Filtrations}
\label{subsec:DualAndBrieskorn}

In this section, we discuss
the holonomic dual of the hypergeometric system
$\cM_{\widetilde{A}}$, from which we deduce a
self-duality property of the module $\widehat{\cM}_{\widetilde{A}}$.
Moreover, we study the natural
good filtration on $\cM_{\widetilde{A}}$ by order of operators,
and show that it is preserved, up to shift, by the
duality isomorphism. We obtain an induced filtration
on $\widehat{\cM}^{loc}_{\widetilde{A}}$ by $\cO_{S_1}[z]$-modules
(which is not a good filtration on this module). Its zeroth step
turns out to coincide with the lattice $\BL$ considered in the last subsection.
This shows that we obtain a non-degenerate pairing on $\BL$, a fact
that we will need later in the construction of Frobenius structures.

We start by describing the holonomic dual of the $\cD_V$-module
$\cM_{\widetilde{A}}$. This description is based on the
local duality theorem for the Gorenstein ring $\Spec \dC[\dN \widetilde{A}]$.
If we were only interested in the description of this dual module,
we could simply refer to \cite[proposition 4.1]{Walther1},
however, as we need later a more refined version taking into account filtrations,
we recall the techniques using Euler-Koszul homology that leads to this duality result.

We suppose throughout this section that the vectors $\underline{a}_1,\ldots, \underline{a}_m$
are defined by toric data.
\begin{theorem}\label{theo:Duality}
\begin{enumerate}
\item
For any holonomic left $\cD_V$-module $\cN$, write $\bD\cN$ for the left
$\cD_V$-module associated to the right $\cD_V$-module ${\cE\!}xt^{m+1}_{\cD_V}(\cN,\cD_V)$,
where we use, as $V$ is an affine space, the canonical identification $\cO_V\cong\Omega^{m+1}_V$
given by multiplying functions with the volume form $d\lambda_0\wedge\ldots\wedge d\lambda_m$.
Then we have
$$
\bD \cM^\beta_{\widetilde{A}} =  \cM^{-\beta+(1,\underline{0})}_{\widetilde{A}},
$$
in particular
$$
\bD \cM_{\widetilde{A}} =  \cM^{(0,\underline{0})}_{\widetilde{A}},
$$
\item
We have the following isomorphisms of holonomic left $\cD_{\widehat{V}}$- (resp. $\cD_{\widehat{T}}$)-modules
$$
\begin{array}{rcl}
\widehat{\cM}_{\widetilde{A}} & \cong & \iota^* \bD\widehat{\cM}_{\widetilde{A}}\\ \\
\widehat{\cM}^{loc}_{\widetilde{A}} & \cong & \iota^* \bD\widehat{\cM}^{loc}_{\widetilde{A}},
\end{array}
$$
here $\iota:\widehat{V}\rightarrow \widehat{V}$ resp. $\iota:\widehat{T}\rightarrow \widehat{T}$
is the automorphism sending $(z,\lambda_1,\ldots,\lambda_m)$ to $(-z,\lambda_1,\ldots,\lambda_m)$.
\end{enumerate}
\end{theorem}
Before giving the proof of this result, we need to introduce some notations.
The basic ingredient for the proof is an explicit resolution
of $\cM^\beta_{\widetilde{A}}$ by the so-called Euler-Koszul complex.
We recall the description of this complex from \cite{MillerWaltherMat}.
In order to be consistent with the notations used in loc.cit., we will rather
work with rings and modules than with sheaves. Therefore, put
$R=\dC[w_0,\ldots,w_m]$ and
$S=R/I$ where $I$ is the toric ideal of $\widetilde{\underline{a}}_0,\widetilde{\underline{a}}_1,\ldots,\widetilde{\underline{a}}_m$,
i.e., the ideal generated by
$$
\begin{array}{rl}
w_0^{\overline{l}}\cdot\prod\limits_{i:l_i<0} w_i^{-l_i}  -  \prod\limits_{i:l_i>0} w_i^{l_i} &
\textup{for any } \underline{l}\in\dL\textup{ with } \overline{l}\geq 0 \\ \\
\prod\limits_{i:l_i<0} w_i^{-l_i}  -  w_0^{\overline{l}}\cdot\prod\limits_{i:l_i>0} w_i^{l_i} &
\textup{for any } \underline{l}\in\dL\textup{ with } \overline{l} < 0
\end{array}
$$
Both rings are $\dZ^{n+1}$-graded, where $\deg(w_i):=-\widetilde{\underline{a}}_i\in \dZ^{n+1}$ (more invariantly, they are $\widetilde{N}$-graded),
notice that the homogeneity of $I$ follows from the fact that $\dL$ is the kernel of the surjection
$\dZ^{m+1}\twoheadrightarrow \widetilde{N}$ given by the matrix $\widetilde{A}$.
We write $D=\Gamma(V,\cD_V)$ for the ring of algebraic differential operators
on $V$. However, using the Fourier-Laplace isomorphism $D\cong \Gamma(V',\cD_{V'})$ given
by $\partial_{\lambda_i} \mapsto -w_i$ and $\lambda_i\mapsto \partial_{w_i}$, we can also view
$D$ as the ring of differential operators on the dual space,
and we shall do so if $D$-modules are considered as $R$-modules.
We have a natural $\dZ^{n+1}$-grading on $D$ defined by $\deg(\lambda_i)=\widetilde{\underline{a}}_i$ and $\deg(\partial_{\lambda_i})=-\widetilde{\underline{a}}_i$,
and the Fourier-Laplace isomorphism gives rise to an injective $\dZ^{n+1}$-graded ring homomorphism $R\hookrightarrow D$ sending $w_i$ to $-\partial_{\lambda_i}$.
Again in order to match our notations with those from \cite{MillerWaltherMat}, let us put
$E_0:=\sum_{i=0}^m \lambda_i\partial_{\lambda_i} \in D$ and $E_k:=\sum_{i=1}^m a_{ki}\lambda_i\partial_{\lambda_i} \in D$ for all $k=1,\ldots,n$.
Let $P$ be any $\dZ^{n+1}$-graded $D$-module, and $\alpha\in \dC^n$ arbitrary, then by putting $(E_k -\alpha_k) \circ y := (E_k -\alpha_k - \deg_k(y))(y)$
for $k=0,\ldots,n$ and for any homogeneous element $y\in P$ and by extending $\dC$-linearly, we obtain a $D$-linear endomorphism of $P$.
We also have that the commutator $[(E_i-\alpha_i)\circ,(E_j-\alpha_j)\circ]$ vanishes for any $i,j\in\{0,\ldots,m\}$.
Hence we can define the Euler-Koszul complex $\cK_\bullet(E-\alpha,P)$, a complex of $\dZ^{n+1}$-graded left $D$-modules, to be the Koszul complex of the endomorphisms
$(E_0-\alpha_0)\circ,\ldots,(E_n-\alpha_n)\circ$ on $P$. Notice that here $E$ is an abbreviation for the vector $(E_0,E_1,\ldots,E_n)$ and should
not be confused with the single vector field $\sum_{i=0}^m \lambda_i\partial_{\lambda_i}+\beta_0$ used in the definition of the modules $\cM^{\beta}_{\widetilde{A}}$.
The definition of the Euler-Koszul complex applies in particular to the case $P:=D\otimes_R T$, where
$T$ is a so-called toric $R$-module (see \cite[definition 4.5]{MillerWaltherMat}), in which case
we also write $\cK_\bullet(E-\alpha, T)$ for the Euler-Koszul complex. Similarly one defines the
Euler-Koszul cocomplex, denoted by $\cK^\bullet(E-\alpha, P)$ resp. $\cK^\bullet(E-\alpha, T)$, where
$\cK^i(E-\alpha, P)=\cK_{n+1-i}(E-\alpha, P)$ and the signs of the differentials are changed accordingly. In particular, we have $H^i(\cK^\bullet(E-\alpha, P))=H_{n+1-i}(\cK_\bullet(E-\alpha, P))$.
We will mainly use the construction of the Euler-Koszul complex resp. cocomplex
in the case of the toric $R$-module $S$, or for shifted version $S(\widetilde{\underline{c}})$, where $\widetilde{\underline{c}}\in\dZ^{n+1}$.

The main result on the Euler-Koszul homology and holonomic duality that we need is the following.
For any $D$-module $M$, consider a $D$-free resolution $L_\bullet \twoheadrightarrow M$, then
we write $\dD M$ for the complex of left $D$-modules associated to $\mathit{Hom}_D(L_\bullet, D)$.
\begin{lemma}[{\cite[theorem 6.3]{MillerWaltherMat}}]\label{lem:SpektralSeqDualityDMod}
Put $\underline{\varepsilon}_{\widetilde{\underline{a}}}:=\sum_{k=0}^m \widetilde{\underline{a}}_i \in \dZ^{n+1}$. Then there is a spectral sequence
\begin{equation}
\label{eq:SpectralDuality}
E^{p,q}_2 = H^q(\cK^\bullet(E+\alpha, Ext^p_R(S,\omega_R))(-\underline{\varepsilon}_{\widetilde{\underline{a}}}) \Longrightarrow   H^{p+q}\dD\left(H_{p+q-(m+1)}(\cK_\bullet(E-\alpha, S))\right)^{-}.
\end{equation}
Here $(-)^{-}$ is the auto-equivalence of $D$-modules induced by the involution $\lambda_i\mapsto-\lambda_i$ and $\partial_{\lambda_i}\mapsto-\partial_{\lambda_i}$.
Notice that it is shown in \cite[lemma 6.1]{MillerWaltherMat} that $Ext^p_R(S,\omega_R)$ is toric.
Notice also that the dualizing module $\omega_R$ is nothing but the ring $R$, placed in $\dZ^{n+1}$-degree $\underline{\varepsilon}_{\widetilde{\underline{a}}}$
(see, e.g., \cite[definition 12.9 and corollary 13.43]{MillSturm} or \cite[corollary 6.3.6]{HerzBr}
for this).
\end{lemma}
In our situation, the relevant $Ext$-group occurring in the spectral sequence of this lemma is actually
rather simple to calculate, as the next result shows.
\begin{lemma}\label{lem:GorensteinDuality}
There is an isomorphism of $\dZ^{n+1}$-graded $R$-modules
$Ext^{m-n}_R(S,\omega_R)\cong \omega_S \cong S((1,\underline{0}))$.
\end{lemma}
\begin{proof}
First it follows from a change of ring property that
$Ext^{m-n}(S,\omega_R)=\omega_S$ (see \cite[proposition 3.6.12]{HerzBr}). We are thus
reduced to compute a canonical module for the ring $S$.
Remark that $S$ is nothing but the semigroup ring $\dC[\dN \widetilde{A}]$ from proposition
\ref{prop:ClosedEmb_NormalCMGorenstein} (see again \cite[theorem 7.3]{MillSturm}), and its canonical module is the ideal in $S$ generated
by the monomials corresponding to the interior points of $\dN \widetilde{A}$. We have seen in proposition \ref{prop:ClosedEmb_NormalCMGorenstein}, 2., that
the set of these interior points is given as $(1,\underline{0})+\dN \widetilde{A}$, i.e., we have that
$\omega_S=S((1,\underline{0}))$, recall that $S$ is a quotient of $R=\dC[w_0,w_1,\ldots,w_m]$ and that $\deg(w_i)=-\widetilde{\underline{a}}_i$.

\end{proof}

\begin{proof}[Proof of the theorem]
In order to use lemma \ref{lem:SpektralSeqDualityDMod} for the computation of the holonomic
dual of $\cM^\beta_{\widetilde{A}}$,
write $M^\beta_{\widetilde{A}}:=H^0(V,\cM^\beta_{\widetilde{A}})$  and
notice that the homology group $H_0(\cK_\bullet(E-\alpha,S))$, seen both as a $R$-module
and a $D$-module,
is nothing but $\Gamma(V',\FL^{w_0,\ldots,w_m}_{\lambda_0,\ldots,\lambda_m}(\cM^{-\alpha}_{\widetilde{A}}))$.
Hence by putting $\alpha:=-\beta$, we have an equality $\bD M^\beta_{\widetilde{A}} = H^{m+1}(\dD H_0(\cK_\bullet(E+\beta,S)))^{-}$
of $D$-modules. Notice that the duality functor and the Fourier-Laplace transformation commutes
only up to a sign (see, e.g., \cite[paragraph 1.b]{DS}), for this reason, the right hand side of the last formula
is twisted by the involution $(-)^{-}$.
\begin{enumerate}
\item
As the ring $S$ is Cohen-Macaulay, $\mathit{Ext}^p_R(S,\omega_R)$ can only be non-zero
if $p=\mathit{codim}_R(S)=m+1-(n+1)=m-n$. This implies that the spectral sequence
\eqref{eq:SpectralDuality} degenerates at the $E_2$-term, so that
$E_2^{m-n,q} = H^{m-n+q}\dD\left(H_{(m-n)+q-(m+1)}(\cK_\bullet(E-\alpha, S))\right)^{-}$.
On the other hand, we deduce from lemma \ref{lem:GorensteinDuality} that
$$
E_2^{m-n,q} = H^q(\cK^\bullet(E+\alpha,S((1,\underline{0}))))\cong H_{n+1-q}(\cK_\bullet(E+\alpha+(1,\underline{0}),S))(1,\underline{0}).
$$
where we have used the equality
$$
\cK_\bullet(E+\alpha,S(\widetilde{\underline{c}})) = \cK_\bullet(E+\alpha+\widetilde{\underline{c}},S)(\widetilde{\underline{c}})
$$
of complexes of $\dZ^{n+1}$-graded $D$-modules.
As noticed in \cite[remark 6.4]{MillerWaltherMat} the CM-property of $S$ also implies that
the Euler-Koszul complex $\cK_\bullet(E-\alpha, S)$ can only have homology in degree zero, hence
$E_2^{m-n,q}=0$ unless $q=n+1$. This is consistent with the fact that due to the holonomicity
of $\cM^\beta_{\widetilde{A}}$, the right hand side of the spectral sequence \eqref{eq:SpectralDuality}
can only be non-zero for $p+q=m+1$. Summarizing, we obtain an isomorphism of $\dZ^{n+1}$-graded $D$-modules
$$
H^{m+1}(\dD H_0(\cK_\bullet(E-\alpha,S)))^{-}=
H_0(\cK_\bullet(E+\alpha+(1,\underline{0}),S))(1,\underline{0}),
$$
from which we deduce an isomorphism of sheaves of $\cD_V$-modules (recall that $\alpha=-\beta$)
$$
\bD \cM^\beta_{\widetilde{A}} \cong \cM^{-\beta+(1,\underline{0})}_{\widetilde{A}},
$$
as required.
\item
Put $\beta=(1,\underline{0})\in \dZ^{n+1}\cong \widetilde{N}$, then it follows from 1. that
we have a morphism
$$
\begin{array}{rcl}
\phi:M_{\widetilde{A}}^\beta & \longrightarrow &  \bD M_{\widetilde{A}}^\beta  \\ \\
m & \longmapsto & a\cdot\partial_{\lambda_0}
\end{array}
$$
Write $\E$ for the partial (polynomial) microlocalization $\dC[\lambda_0,\lambda_1,\ldots,\lambda_m]\langle \partial_{\lambda_0},
\partial_{\lambda_0}^{-1},\partial_{\lambda_1},\ldots,\partial_{\lambda_m}\rangle$.
Then $\phi$ induces an isomorphism $\E \otimes_D M^\beta_{\widetilde{A}}\stackrel{\cong}{\rightarrow} \E \otimes_D (\bD M^\beta_{\widetilde{A}})$.
On the other hand, it follows from the $D$-flatness of $\E$ that $\mathit{Ext}^{m+1}_{\E}(M,\E)=\mathit{Ext}^{m+1}_D(M,D)\otimes\E$ for any left $D$-module $M$, hence
we obtain an isomorphism
$$
\E \otimes_D M^\beta_{\widetilde{A}} \stackrel{\cong}{\longrightarrow} \bD(\E \otimes_D M^\beta_{\widetilde{A}})
$$
where the symbol $\bD$ on the right hand side denotes the composition of $\mathit{Ext}^{m+1}_{\E}(-,\E)$ with the transformation
of right $\E$ to left $\E$-modules. Performing a partial Fourier-Laplace transformation, we obtain an isomorphism (still denoted by $\phi$)
$$
\phi:\FL^\tau_{\lambda_0}(M^\beta_{\widetilde{A}})[\tau^{-1}]\stackrel{\cong}{\longrightarrow} \FL^{\tau}_{\lambda_0}\left(\bD (\E\otimes_D M^\beta_{\widetilde{A}})\right),
$$
which is given by right multiplication with $\tau=z^{-1}$.
On the other hand, it is known (see, e.g., \cite[paragraph 1.f]{DS}) that for any $\E$-module $N$, we have an isomorphism $\FL^{\tau}_{\lambda_0}(\bD N)\cong \iota^* \bD(\FL^{\tau}_{\lambda_0} (N))$
which gives us
$$
\FL^\tau_{\lambda_0}(M^\beta_{\widetilde{A}})[\tau^{-1}]\stackrel{\cong}{\longrightarrow} \iota^*\bD \left(\FL^{\tau}_{\lambda_0}(\E\otimes_D M^\beta_{\widetilde{A}})\right)
\cong
\iota^*\bD \left(\FL^{\tau}_{\lambda_0}(M^\beta_{\widetilde{A}})[\tau^{-1}]\right)
$$
from which we deduce the isomorphism $\widehat{\cM}_{\widetilde{A}}  \cong \iota^* \bD\widehat{\cM}_{\widetilde{A}}$ of $\cD_{\widehat{V}}$-modules
resp. the isomorphism $\widehat{\cM}^{loc}_{\widetilde{A}} \cong  \iota^* \bD\widehat{\cM}^{loc}_{\widetilde{A}}$ of $\cD_{\widehat{T}}$-modules.
\end{enumerate}
\end{proof}

The next step is to investigate a natural good filtration defined on the sheaf $\cM^\beta_{\widetilde{A}}$.
We write $M^\beta_{\widetilde{A}}:=H^0(V,\cM^\beta_{\widetilde{A}})$ which is isomorphic
to $H_0(\cK_\bullet(E+\beta,S))$ as a $D$-module.
\begin{proposition}\label{prop:FiltrationDuality}
\begin{enumerate}
\item
Write $F_\bullet$ for the natural filtration on $D$ by order of $\partial_{\lambda_i}$-operators and
denote the induced filtration on $M^\beta_{\widetilde{A}}$ also by $F_\bullet$.
There is a resolution $L_\bullet$ of $M^\beta_{\widetilde{A}}$ by free $D$-modules which is equipped with a strict
filtration $F_\bullet^{L_\bullet}$ and we have a filtered quasi-isomorphism $(L_\bullet,F^{L_\bullet})
\twoheadrightarrow (M^\beta_{\widetilde{A}},F_\bullet)$.
\item
Consider the case $\beta=(1,\underline{0})$, i.e., $M^\beta_{\widetilde{A}}=M_{\widetilde{A}}$. Write $F^{\bD}_\bullet \,\bD M_{\widetilde{A}}$ for the dual filtration of $F_\bullet M_{\widetilde{A}}$, i.e.,
$\bD(M_{\widetilde{A}},F_\bullet) = (\bD M_{\widetilde{A}}, F^{\bD}_\bullet)$ (see, e.g., \cite[page 55]{SaitoOnMHM}), then we have
$$
F_k M^{(0,\underline{0})}_{\widetilde{A}}
 = F^{\bD}_{k-n+(m+1)} \bD M_{\widetilde{A}}.
$$
\item
For any $\beta\in\dZ^{n+1}$, $F_\bullet M^\beta_{\widetilde{A}}$ induces a filtration $G^\beta_\bullet$ by $\cO_{\dC_z\times S_1}$-modules on the
$\cD_{\widehat{T}}$-module $\widehat{\cM}^{\beta,loc}_{\widetilde{A}}$ and we have an isomorphism of
$\cO_{\dC_z\times S_1}$-modules
$$
G_0\widehat{\cM}^{\beta,loc}_{\widetilde{A}} \cong {_0\!}\widehat{\cM}^{\beta,loc}_{\widetilde{A}},
$$
in particular
$$
G_0\widehat{\cM}^{loc}_{\widetilde{A}}\cong \BL.
$$
Moreover, for any $k$, $\cO_{\dC_z\times S_1^0}\otimes_{\cO_{\dC_z\times S_1}} G_k\widehat{\cM}^{loc}_{\widetilde{A}} $ is
$\cO_{\dC_z\times S_1^0}$-locally free.

For $\beta=(1,\underline{0})$ we obtain from the dual filtration $F^{\bD}_\bullet$ on $\bD M_{\widetilde{A}}$ a filtration $G^{\bD}_\bullet$ by $\cO_{\dC_z\times S_1}$-modules on $\widehat{\cM}^{(0,\underline{0})}_{\widetilde{A}}$.
\item
Consider the isomorphism
$$
\phi: \FL^{w_0}_{\lambda_0}(\cM_{\widetilde{A}})[\tau^{-1}]=\widehat{\cM}_{\widetilde{A}} \longrightarrow
\iota^*\FL^{w_0}_{\lambda_0}(\bD\cM_{\widetilde{A}})[\tau^{-1}] = \widehat{\cM}^{(0,\underline{0})}_{\widetilde{A}}
$$
from the proof of theorem \ref{theo:Duality}, 2., which is
given by multiplication with $z^{-1}$. Then we have $\phi(G_\bullet) = G^{\bD}_{\bullet+m+2-n}\widehat{\cM}^{(0,\underline{0})}_{\widehat{A}}$.
\end{enumerate}
\end{proposition}
\begin{proof}
\begin{enumerate}
\item
The free resolution $L_\bullet \twoheadrightarrow M^\beta_{\widetilde{A}}$ is obtained as in the
proof of
\cite[theorem 6.3]{MillerWaltherMat} as the total complex $\textup{Tot }\cK_\bullet(E+\beta,\cF_\bullet)$
of a resolution of the Euler-Koszul complex
obtained from a $R$-free $\dZ^{n+1}$-graded resolution $\cF_\bullet$ of $S$. In particular,
this resolution is $\dZ$-graded for the grading of $R=\dC[w_0,w_1,\ldots,w_m]=\dC[\partial_{\lambda_0},\partial_{\lambda_1},
\ldots,\partial_{\lambda_m}]$ for which $\deg(w_i)=\deg(\partial_{\lambda_i})=1$. On the other hand,
the differentials of the Euler-Koszul complex are constructed from linear differential operators. Hence
by putting on each term of the above total complex (which is $D$-free) a filtration which is on each factor
of such a module the order filtration on $D$, shifted appropriately, we obtain a strict resolution of $(M^\beta_{\widetilde{A}},F_\bullet)$.
\item
From the construction of the resolution $L_{\bullet}\twoheadrightarrow M_{\widetilde{A}}$, from point 1., we see that
$L_k=0$ for all $k> m+1$ (notice that we write this resolution such that $d:L_k\rightarrow L_{k-1}$ so that $M_{\widetilde{A}}=H_0(L_\bullet,d)$)
and $L_{m+1}=D$. We have seen that the filtration on $L_{m+1}$ is the order filtration on $D$, shifted appropriately
and we have to determine this shift. It is the sum of the length of the Euler-Koszul complex (i.e., $n+1$) and
the degree (with respect to the grading of $R$ for which $\deg(\partial_{\lambda_i})=1$) of
$\mathit{Ext}^{n-m}_R(S,\omega_R)$. The latter is equal to $m$,
which is the first component of the difference between the canonical degree of $R$ (i.e., $\varepsilon_A$)
and the canonical degree of $S$ (i.e., $(1,\underline{0})$). Hence the filtration on $L_{m+1}$ is $F_{\bullet-(n+m+1)} D$.
Now by definition (see, e.g., \cite[page 55]{SaitoOnMHM}), we have
$$
\bD(M_{\widetilde{A}},F_\bullet) = H^{m+1} \mathit{Hom}_D\left((L_\bullet,F^{L_\bullet}_\bullet),((D\otimes\Omega_V^{m+1})^\vee,F_{\bullet-2(m+1)}D\otimes(\Omega_V^{m+1})^\vee)\right)
$$
and this implies the formula for $F^{\bD}_\bullet \bD M_{\widetilde{A}}$.
\item
We will consider the $\partial_{\lambda_0}^{-1}$-saturation of the filtration steps $F_k M_{\widetilde{A}}$. More precisely,
consider again $M_{\widetilde{A}}[\partial_{\lambda_0}^{-1}]:= \E\otimes_D M_{\widetilde{A}}$, and the natural localization morphism $\widehat{\textup{loc}}:M_{\widetilde{A}}\rightarrow M_{\widetilde{A}}[\partial_{\lambda_0}^{-1}]$. Put $F_kM_{\widetilde{A}}[\partial_{\lambda_0}^{-1}]:=\sum_{j\geq 0} \partial_{\lambda_0}^{-j}\widehat{\textup{loc}}(F_{k+j} M_{\widetilde{A}})$.
Then we easily see that
$$F_k M_{\widetilde{A}}[\partial_{\lambda_0}^{-1}] = \mathit{Im}\left(\partial_{\lambda_0}^{k}\dC[\lambda_0,\lambda_1,\ldots,\lambda_m]\langle \partial_{\lambda_0}^{-1},\partial_{\lambda_0}^{-1}\partial_{\lambda_1},\ldots,
\partial_{\lambda_0}^{-1}\partial_{\lambda_m}\rangle\right) \mbox{ in }M_{\widetilde{A}}[\partial_{\lambda_0}^{-1}].
$$

The filtration $F_\bullet M_{\widetilde{A}}[\partial_{\lambda_0}^{-1}]$ induces a
filtration $G_\bullet$ on $\widehat{M}^{loc}_{\widetilde{A}}=\Gamma(\widehat{T},\widehat{\cM}^{loc}_{\widetilde{A}})$, with
$$
G_k \widehat{M}^{loc}_{\widetilde{A}} =
\mathit{Im}\left(z^{-k}\dC[z,\lambda_1^\pm,\ldots,\lambda_m^\pm]\langle z\partial_{\lambda_1},\ldots,
z\partial_{\lambda_m},z^2\partial_z\rangle\right) \mbox{ in }\widehat{M}^{loc}_{\widetilde{A}}
$$
Hence we obtain a filtration $G_\bullet$ on the sheaf $\widehat{\cM}^{loc}_{\widetilde{A}}$ and we have
$G_0\widehat{\cM}^{loc}_{\widetilde{A}} = \BL$, as required.
Moreover, $z^k\cdot: G_p\widehat{\cM}^{loc}_{\widetilde{A}} \stackrel{\cong}{\longrightarrow} G_{p-k}\widehat{\cM}^{loc}_{\widetilde{A}}$, and it follows from
theorem \ref{theo:BrieskornLatticeFree} that $\cO_{\dC_z\times S_1^0}\otimes_{\cO_{\dC_z\times S_1}} G_0\widehat{\cM}^{loc}_{\widetilde{A}} $, and hence all
$\cO_{\dC_z\times S_1^0}\otimes_{\cO_{\dC_z\times S_1}} G_p\widehat{\cM}^{loc}_{\widetilde{A}}$ are $\cO_{\dC_z\times S^0_1}$-locally free.
Notice however that $G_\bullet$ is in general not a good filtration on $\widehat{M}^{loc}_{\widetilde{A}}$, as $\partial_z G_k \widehat{M}^{loc}_{\widetilde{A}}\subset G_{k+2}\widehat{M}^{loc}_{\widetilde{A}}$ whereas $\partial_{\lambda_i} G_k\widehat{M}^{loc}_{\widetilde{A}}\subset G_{k+1}\widehat{M}^{loc}_{\widetilde{A}}$.

Concerning the filtration $G_\bullet^{\bD}$, notice that due to the definition
of $F_k M_{\widetilde{A}}[\partial_{\lambda_0}^{-1}]$, the strictly filtered resolution of
$(M_{\widetilde{A}},F_\bullet)$ from part 2 from above yields a strictly filtered resolution
of the filtered module $(M_{\widetilde{A}}[\partial_{\lambda_0}^{-1}], F_\bullet M_{\widetilde{A}}[\partial_{\lambda_0}^{-1}])$,
and the dual complex is then also strictly filtered and defines a filtration $G_\bullet^{\bD}$ on $\bD (M_{\widetilde{A}}[\partial_{\lambda_0}^{-1}])$,
which is nothing but the $\partial_{\lambda_0}^{-1}$-saturation of the dual filtration $F_\bullet^\bD$ from point 2. from above.
Hence we obtain a filtration $G_\bullet$ by $\cO_{\dC_z\times S_1}$-modules on $\bD\widehat{\cM}_{\widetilde{A}}=\widehat{\cM}^{(0,\underline{0})}_{\widetilde{A}}$.
\item
This is a direct consequence of 2. and 3.

\end{enumerate}
\end{proof}
As a consequence, we obtain the existence of a non-degenerate pairing on the lattice
$\BL$ considered above.
\begin{corollary}\label{cor:PairingPUp}
\begin{enumerate}
\item
There is a non-degenerate flat $(-1)^n$-symmetric pairing
$$
P:\left(\cO_{\dC^*_\tau\times S_1^0}\otimes_{\cO_{\dC_\tau\times S_1}} \widehat{\cM}^{loc}_{\widetilde{A}}\right)\otimes \iota^*\left(\cO_{\dC^*_\tau\times S_1^0}\otimes_{\cO_{\dC_\tau\times S_1}}
 \widehat{\cM}^{loc}_{\widetilde{A}}\right)\rightarrow \cO_{\dC^*_\tau\times S_1^0}.
$$
\item
We have that $P(\BL,\BL)\subset z^n \cO_{\dC_z\times S_1}$,
and $P$ is non-degenerate on $\cO_{\dC_z\times S_1^0}\otimes_{\cO_{\dC_z\times S_1}} \BL$, i.e., it induces a non-degenerate symmetric pairing
$$
[z^{-n}P]:\left[\cO_{S_1^0}\otimes_{\cO_{S_1}}\dfrac{\BL}{z\cdot\BL}\right]\otimes\left[\cO_{S_1^0}
\otimes_{\cO_{S_1}}\dfrac{\BL}{z\cdot \BL}\right]\rightarrow\cO_{S^0_1}.
$$
\end{enumerate}
\end{corollary}
\begin{proof}
\begin{enumerate}
\item
The statement can be reformulated as the existence of an isomorphism
$$
\psi:\left(\cO_{\dC^*_\tau\times S_1^0}\otimes_{\cO_{\dC_\tau\times S_1}} \widehat{\cM}^{loc}_{\widetilde{A}}\right) \stackrel{\cong}{\longrightarrow}
\iota^*\left(\cO_{\dC^*_\tau\times S_1^0}\otimes_{\cO_{\dC_\tau\times S_1}} \widehat{\cM}^{loc}_{\widetilde{A}}\right)^*
$$
where $(-)^*$ denotes the dual meromorphic bundle with its dual connection.
We deduce from  \cite[lemma A.11]{DS} (see also \cite[2.7]{SM}) that
$\bD(\cO_{\dC^*_\tau\times S_1^0}\otimes_{\cO_{\dC_\tau\times S_1}}\widehat{\cM}^{loc}_{\widetilde{A}})(*(\{0,\infty\}\times S_1^0))
=(\cO_{\dC^*_\tau\times S_1^0}\otimes_{\cO_{\dC_\tau\times S_1}}\widehat{\cM}^{loc}_{\widetilde{A}})^*$.
On the other hand, theorem \ref{theo:Duality}, 2. gives an isomorphism
$\cO_{\dC^*_\tau\times S_1^0}\otimes_{\cO_{\dC_\tau\times S_1}}\widehat{\cM}^{loc}_{\widetilde{A}}\cong \iota^*\bD(\cO_{\dC^*_\tau\times S_1^0}\otimes_{\cO_{\dC_\tau\times S_1}}\widehat{\cM}^{loc}_{\widetilde{A}})$ so that the latter module is already localized, i.e., equal to
$(\cO_{\dC^*_\tau\times S_1^0}\otimes_{\cO_{\dC_\tau\times S_1}}\widehat{\cM}^{loc}_{\widetilde{A}})^*$, which gives the existence of the isomorphism $\psi$ from above.
\item
We have seen in point 1. that the duality isomorphism
$$
\phi=z^{-1}\cdot : \FL^{w_0}_{\lambda_0}(\cM_{\widetilde{A}})[\tau^{-1}] \longrightarrow
\FL^{w_0}_{\lambda_0}(\bD \cM_{\widetilde{A}})[\tau^{-1}]
$$
yields an isomorphism
$$
\psi:\left(\cO_{\dC^*_\tau\times S_1^0}\otimes_{\cO_{\dC_\tau\times S_1}} \widehat{\cM}^{loc}_{\widetilde{A}}\right) \stackrel{\cong}{\longrightarrow}
\iota^*\left(\cO_{\dC^*_\tau\times S_1^0}\otimes_{\cO_{\dC_\tau\times S_1}} \widehat{\cM}^{loc}_{\widetilde{A}}\right)^*
$$
of meromorphic bundles with connection. Now it follows from \cite[formula 2.7.5]{SM} that we have
$$
{\cH}\!om_{\cO_{\dC_z\times S^0_1}}\left(\cO_{\dC_z\times S^0_1}\otimes_{\cO_{\dC_z\times S_1}} G_k \widehat{\cM}^{loc}_{\widetilde{A}},\cO_{\dC_z\times S^0_1}\right)
=\cO_{\dC^*_\tau\times S_1^0}\otimes_{\cO_{\dC_z\times S_1}} G^{\bD}_{k+(m+2)} \widehat{\cM}^{(0,\underline{0}),loc}_{\widetilde{A}}.
$$
Hence by proposition
\ref{prop:FiltrationDuality}, 4. from above we conclude that $\psi$ sends the module
$$
\cO_{\dC_z \times S_1^0}\otimes_{\cO_{\dC_z \times S_1}} G_0
\widehat{\cM}^{loc}_{\widetilde{A}}=\cO_{\dC_z \times S_1^0}\otimes_{\cO_{\dC_z \times S_1}} \BL
$$
isomorphically into
$$
\begin{array}{rcl}
& & {\cH}\!om_{\cO_{\dC_z\times S^0_1}} \left(\cO_{\dC_z \times S_1^0}\otimes_{\cO_{\dC_z \times S_1}} G_{-n} \widehat{\cM}^{loc}_{\widetilde{A}},\cO_{\dC_z\times S^0_1}\right)
\\ \\
& = & z^n{\cH}\!om_{\cO_{\dC_z\times S^0_1}} \left(\cO_{\dC_z \times S_1^0}\otimes_{\cO_{\dC_z \times S_1}} G_0 \widehat{\cM}^{loc}_{\widetilde{A}},\cO_{\dC_z\times S^0_1}\right) \\ \\
& =& z^n{\cH}\!om_{\cO_{\dC_z\times S^0_1}}\left(\cO_{\dC_z \times S_1^0}\otimes_{\cO_{\dC_z \times S_1}} \BL,\cO_{\dC_z\times S^0_1}\right),
\end{array}
$$
which is equivalent to the statement to be shown.
\end{enumerate}
\end{proof}

\section{$\cD$-modules with logarithmic structure and good bases}
\label{sec:LogQDMod}

In this section we apply the results of section \ref{sec:HyperGM} to
study hypergeometric $\cD$-modules on a subtorus of the $m$-dimensional
torus $S_1$. We suppose that our vectors $\underline{a}_1,\ldots,\underline{a}_m$
are defined by toric data.
In this situation, the subtorus is defined as  $S_2:=\Spec \dC[\dL]$,
where, as before, $\dL$ is the module of relations
between $\underline{a}_1,\ldots,\underline{a}_m$.
Following standard terminology, we call this torus the complexified K\"ahler moduli space of $\XSigA$.
We will consider a subfamily of Laurent polynomials
of the morphism $\varphi:S_0\times W\rightarrow \dC_t\times W$ from the last section,
parameterized by $S_2$ and we will show
that the associated Gau\ss-Manin system also has a hypergeometric structure.

For a good choice of coordinates on $S_2$ embedding it into some affine space $\dC^r$, we will construct
an extension of this hypergeometric modules to a certain lattice with logarithmic poles along the
boundary divisor $\dC^r\backslash S_2$. This ``$\cD$-module with logarithmic structure'' will play a crucial role
in the next section: on the one hand, we will see that it equals the so-called Givental connection defined
by the quantum cohomology of the variety $\XSigA$, on the other hand, we will use it to construct logarithmic Frobenius structures and express the mirror correspondence in terms of them. For that purpose, we will show that this
logarithmic extension is still a free module, and can be extended to a family of trivial bundle over $\dP^1\times \dC^r$ (or at least outside the locus where the family of mirror Laurent polynomials is degenerate) on which the connection extends with a logarithmic pole at infinity. This structure is the key ingredient
to construct a logarithmic Frobenius manifold, this will be done in section \ref{sec:AModel}.

\subsection{Landau-Ginzburg models and hypergeometric $\cD$-modules on K\"ahler moduli spaces}
\label{subsec:QDModDown}

We briefly recall the situation considered in the beginning of the last section, with the
more specific assumption that now the input data we are working with are of toric nature.
Hence, let again $N$ be a free abelian group of rank $n$ which we identify with $\dZ^n$
by chosing a basis. Let $\Sigma_A\subset N_\dR=N\otimes \dR$ be a fan
defining a smooth projective toric weak Fano variety $\XSigA$.
We write $\Sigma_A(1)$ for the set of rays (i.e., one dimensional cones) of $\Sigma_A$, we will
often denote such a ray by $v_i$. As before, $\underline{a}_1,\ldots,\underline{a}_m$
are the primitive integral generators of the rays $v_1,\ldots,v_m$ in $\Sigma_A(1)$. Consider the exact sequence
\begin{equation}
\label{eq:ExSeq}
0\longrightarrow \dL \stackrel{m}{\longrightarrow} \dZ^{\Sigma_A(1)}\cong \dZ^m \stackrel{A}{\longrightarrow} N \longrightarrow 0.
\end{equation}
Applying the functor $\mathit{Hom}_\dZ(-,\dC^*)$ yields
\begin{equation}
\label{eq:ExSeqTorFibr}
1 \longrightarrow S_0=\Spec \dC[N]\cong(\dC^*)^n \longrightarrow (\dC^*)^{\Sigma_A(1)}\cong (\dC^*)^m \stackrel{q}{\longrightarrow} S_2:=\Spec\dC[\dL] \cong \dL^\vee\otimes\dC^*\longrightarrow 1.
\end{equation}
The middle torus $(\dC^*)^{\Sigma_A(1)}$ is naturally dual to $S_1=\Spec\dC[\lambda_1^\pm,\ldots,\lambda_m^\pm]$, however, we
will from now on identify both (as well as the corresponding affine spaces $W$ and $W'$), so that
we denote $(\dC^*)^{\Sigma_A(1)}$ also by $S_1$.
Notice that the composition of the first map of the exact sequence \eqref{eq:ExSeqTorFibr} with the open embedding  $(\dC^*)^m \hookrightarrow \dC^m$ is nothing
but the map $k$ from proposition \ref{prop:ClosedEmb_NormalCMGorenstein}, which was shown to be closed.
Recall that for smooth toric varieties, $\dL^\vee$ equals the Picard group $\textup{Pic}(\XSigA)$.
Inside $\dL^\vee_\dR:= \dL^\vee \otimes \dR$ we have the K\"ahler cone $\cK_{\Sigma_A}$, which consists
of all classes $[a]\in \dL^\vee_\dR$ such that $a$, seen as a piecewise linear function on $N_\dR$ (linear
on each cone of $\Sigma_A$) is convex. The interior $\cK^0_{\Sigma_A}$ of the K\"ahler cone are the strictly convex
piecewise linear functions on $N_\dR$.  Write $D_i$ for the torus invariant divisors of $\XSigA$ associated
to the ray generated by $\underline{a}_i$, then the anti-canonical divisor of $\XSigA$ is
$\rho=\sum_{i=1}^m [D_i] \in \dL^\vee$. Recall that $\XSigA$ is Fano resp. weak Fano iff $\rho\in \cK^0_{\Sigma_A}$ resp.
$\rho\in \cK_{\Sigma_A}$.
We will choose a basis of $\dL^\vee$ consisting of classes $p_1,\ldots,p_r$ (with $r=m-n$) which lie in $\cK_{\Sigma_A}$
and such that $\rho$ lies in the cone generated by $p_1,\ldots,p_r$.
This identifies $S_2$ with $(\dC^*)^r$, and we write $q_1,\ldots,q_r$ for the coordinates defined by this identification.

The next definition describes one of the main objects of study of this paper.
\begin{definition}
Consider the linear function $W=w_1+\ldots+w_m:S_1 \rightarrow \dC_t$. The \textbf{Landau-Ginzburg model}
of the toric weak Fano variety $\XSigA$ is the restriction of the function $W$ to
the fibres of the torus fibration $q:S_1\cong (\dC^*)^m \rightarrow S_2\cong(\dC^*)^r$.
We will also sometimes call the morphism
$$
(W,q):S_1\cong(\dC^*)^m \longrightarrow \dC_t \times S_2\cong\dC_t\times(\dC^*)^r
$$
a Landau-Ginzburg model. Notice that the choice of the basis $p_1,\ldots,p_r$ (and hence the choice
of coordinates on $S_2$) are part of the data of the Landau-Ginzburg model, which would otherwise
only depend on the set of rays $\Sigma(1)$, but not on the fan $\Sigma$ itself.
\end{definition}
The choice of a basis $p_1,\ldots,p_r$ of $\dL^\vee$ also determines an open embedding $S_2 \hookrightarrow \dC^r$.
An important issue in this section will be to extend the various data defined by the Landau-Ginzburg model
of $\XSigA$ over the boundaray divisor $\dC^r\backslash S_2$. As a side remark, notice that
the K\"ahler cone of a toric Fano variety does not need to be simplicial, the simplest example being the toric del Pezzo surface obtained by blowing
up three points in $\dP^2$ in generic position. Hence the above chosen basis of $\dL^\vee$ does not necessarily generate the K\"ahler cone.

Using the dual basis $(p^\vee_a)_{a=1,\ldots,r}$ of $\dL$, the above map $m$ is given by a matrix $(m_{ia})$ with columns $\underline{m}_a$
and hence the torus fibration $q:(\dC^*)^m \twoheadrightarrow (\dC^*)^r$ is given by $q(w_1,\ldots,w_m)=(q_a=\underline{w}^{\underline{m}_a}:=\prod_{i=1}^m w_i^{m_{ia}})_{a=1,\ldots,r}$. We will also consider the product map $(\id_z,q):\dP^1_z \times S_1 \twoheadrightarrow \dP^1 \times S_2$
as well as its restriction to $\dC^*_z\times S_1 $.
Choose moreover a section $g:\dL^\vee\rightarrow \dZ^m$ of the projection $\dZ^m \twoheadrightarrow \dL^\vee$, which is given
in the chosen basis $p_1,\ldots,p_r$ of $\dL^\vee$ by a matrix $(g_{ia})$ with rows $\underline{g}_i$, so that
$\sum_{i=1}^m g_{ia} m_{ib}= \delta_{ab}$. The map $g$ induces a section
of the fibration $q$, still denoted by $g$, which is given as
$$
\begin{array}{rcl}
g: S_2 & \longrightarrow & S_1 \\ \\
(q_1,\ldots,q_r) & \longmapsto & (w_i:=\underline{q}^{g_i}:=\prod_{a=1}^r q_a^{g_{ia}})_{i=1,\ldots,m}
\end{array}
$$
Obviously, $g$ also gives a splitting of the fibration $q$, see diagram \eqref{eq:DiagramCoordChange} below.
Let us notice
that the section $g$ can be chosen such that the entries of the matrix $g_{ia}$
are non-negative integers. For this, recall (see, e.g., \cite[section 3.4.2]{CK})
the description of the K\"ahler cone as the intersection of cones in $\dL^\vee\otimes \dR$ each of which is generated by
the images under $\dZ^m\twoheadrightarrow \dL^\vee$ of some of the standard generators of $\dZ^m$ (the so-called anti-cones associated
to the cones $\sigma\in\Sigma_A$). Hence,
the chosen basis $(p_a)_{a=1,\ldots,r}$ of $\dL^\vee$ which consists of elements of $\cK_{\Sigma_A}$
can be expressed in the generators of any of these cones, and the coefficients are exactly the entries
of the matrix $(g_{ia})$, hence, non-negative. It follows that the section
$g:S_2\rightarrow S_1$ extends to a map $\overline{g}:\dC^r \rightarrow W=\Spec\dC[w_1,\ldots,w_m]$,
although the projection map $q:S_1\twoheadrightarrow S_2$ cannot be extended over the boundary
$\bigcup_{i=1}^m \{w_i=0\}\subset W$. In what follows, we will always assume that $g$ is constructed in such a way.

Write $S_2^0:=g^{-1}(S_1^0)=\{(q_1,\ldots,q_r)\in S_2\,|\, \widetilde{W}:=\sum_{i=1}^m \underline{q}^{\underline{g}_i} \underline{y}^{\underline{a}_i}
\textup{ is Newton non-degenerate }\}$. Finally, we define $\widetilde{g}=(\id_z,g):\dP^1_z\times S_2 \rightarrow \dP^1_z\times S_1$, which is a section of
the above projection map $(\id_z,q)$.
\begin{proposition}\label{prop:ReducedAHyperGeom}
The embedding $\widetilde{g}$ is non-characteristic for $\widehat{\cM}^{loc}_{\widetilde{A}}$ on
$\dP^1_z \times S_1^0$. Moreover,
the inverse image $\widetilde{g}^+\widehat{\cM}^{loc}_{\widetilde{A}}$ is given as the quotient of $\cD_{\dC_\tau\times S_2}[\tau^{-1}]/\widetilde{\cI}$,
where $\widetilde{\cI}$ is the left ideal
generated by
$$
\widetilde{\Box}_{\underline{l}}:=\prod_{a:p_a(\underline{l})>0} q_a^{p_a(\underline{l})} \prod_{i:l_i<0}
\prod_{\nu=0}^{-l_i-1} (\sum_{a=1}^r m_{ia} zq_a\partial_{q_a} -\nu z) -
\prod_{a:p_a(\underline{l})<0} q_a^{-p_a(\underline{l})} \prod_{i:l_i>0}
\prod_{\nu=0}^{l_i-1} (\sum_{a=1}^r m_{ia} zq_a\partial_{q_a} -\nu z)
$$
for any $\underline{l}\in\dL$ and by the single operator
$$
z\partial_z+\sum_{a=1}^r \rho(p^\vee_a) q_a\partial_{q_a}.
$$
Notice that $p_a$ is a linear form on $\dL$ and that we have $\sum_{i=1}^m g_{ia} l_i = \sum_{i,b} g_{ia} (m_{ib} p_b(\underline{l})) = p_a(\underline{l})\in \dZ$.
\end{proposition}
\begin{proof}
The non-characteristic condition is evident as the singular locus of $\widehat{\cM}^{loc}_{\widetilde{A}}$, seen as a $\cD_{\dP^1_z\times S_1}$-module
is contained in $(\{0,\infty\}\times S_1)\cup(\dP^1_z\times (S_1\backslash S_1^0))$. In order to calculate the inverse image, consider the following diagram
\begin{equation}
\label{eq:DiagramCoordChange}
\xymatrix@!0{
          &&&&& S_1 \ar[ddddlllll]^{\Phi^{-1}} \ar[ddddrrrrr]^{q} \\ \\ \\ \\
          S_0 \times S_2 \ar@(u,l)[uuuurrrrr]^\Phi \ar[rrrrrrrrrr]^{\pi} &&&&&&&&&&  S_2 \\
          (y_1,\ldots,y_n,q_1,\ldots,q_r) \ar@{|->}[rrrrrrrrrr] &&&&&&&&&& (q_1,\ldots,q_r) }
\end{equation}
where the coordinate change $\Phi$ is given as
$$
\Phi(\underline{y},\underline{q}) := \left(w_i=\underline{q}^{\underline{g}_i}\cdot y^{\underline{a}_i}\right)_{i=1,\ldots,m}
$$
As the diagram commutes,
the $\underline{q}$-component of $\Phi^{-1}$ is $q_a=\underline{w}^{\underline{m}_a}$.
Putting $\widetilde{\Phi}:S_0 \times \dC_\tau \times S_2 \rightarrow \dC_\tau \times S_1$, $(\underline{y},\tau,\underline{q})\mapsto(\tau,\Phi(\underline{y},\underline{q}))$
and similarly $\widetilde{\pi}: S_0 \times \dC_\tau \times S_2 \rightarrow \dC_\tau \times S_2$, $(\underline{y},\tau,\underline{q})
\mapsto(\tau,\underline{q})$, we consider the module $\widetilde{\Phi}^+ \widehat{\cM}^{loc}_{\widetilde{A}}$ which is
(using the presentation $\widehat{\cM}^{loc}_{\widetilde{A}}=\cD_{\widehat{T}}[\tau^{-1}]/\widehat{\cI}''$) equal to
the quotient of $\cD_{S_0\times \dC_\tau \times S_2}[\tau^{-1}]$ by the left ideal generated by
$$
\begin{array}{rcl}
\prod\limits_{a:p_a(\underline{l})<0} q_a^{p_a(\underline{l})} \widetilde{\Box}_{\underline{l}} & = &
\prod\limits_{a=1}^r q_a^{p_a(\underline{l})} \prod\limits_{i:l_i<0}
\prod\limits_{\nu=0}^{-l_i-1} (\sum_{a=1}^r m_{ia} zq_a\partial_{q_a} -\nu z) -
\prod\limits_{i:l_i>0}
\prod\limits_{\nu=0}^{l_i-1} (\sum_{a=1}^r m_{ia} zq_a\partial_{q_a} -\nu z) \\ \\
\widetilde{Z}_k & = & y_k\partial_{y_k}\\ \\
\widetilde{E} & = & z\partial_z+\sum_{a=1}^r (\sum_{i=1}^m m_{ia})q_a\partial_{q_a} = z\partial_z+\sum_{a=1}^r \rho(p^\vee_a) q_a\partial_{q_a}
\end{array}
$$
In other words, we have
$$
\widetilde{\Phi}^+\widehat{\cM}^{loc}_{\widetilde{A}}
=
\dfrac{\dC[z^\pm,y_1^\pm,\ldots,y_n^\pm,q_1^\pm,\ldots,q^\pm_r]\langle \partial_z, \partial_{q_1}, \ldots, \partial_{q_r}\rangle}{(\widetilde{\Box}_{\underline{l}})_{\underline{l}\in\dL}+\widetilde{E}}
$$

Obviously, the map $\widetilde{g}$ is given in the new coordinates by $\widetilde{g}(\tau,\underline{q}):=(\tau,\underline{1},\underline{q})\in \dC_\tau\times S_0\times S_2$, hence
we obtain
$$
\widetilde{g}^+\widehat{\cM}^{loc}_{\widetilde{A}}=
\dfrac{\dC[z^\pm,q_1^\pm,\ldots,q^\pm_r]\langle \partial_z, \partial_{q_1}, \ldots, \partial_{q_r}\rangle}{(\widetilde{\Box}_{\underline{l}})_{\underline{l}\in\dL}+\widetilde{E}}
$$
\end{proof}
As a consequence of this lemma, and using the comparison result in theorem \ref{theo:GM-GKZUp}, we can interpret this reduced GKZ-system as a Gau\ss-Manin-system.
\begin{corollary}
\label{cor:GM-GKZDown}
Consider the (restriction of the) Landau-Ginzburg model $(W,q):S^0_1 \rightarrow \dC_t \times S^0_2$.
Then there is an isomorphism of $\cD_{\dC_\tau\times S^0_2}$-modules
$$
\cD_{\dC_\tau\times S^0_2}[\tau^{-1}]/ \widetilde{\cI} \cong \FL^\tau_t(\cH^0((W,q)_+\cO_{S^0_1}))[\tau^{-1}]
$$
\end{corollary}
\begin{proof}
First notice that due to diagram \eqref{eq:DiagramCoordChange} we have an isomorphism
$$
\cH^0((W,q)_+\cO_{S_1^0})
\cong\cH^0((\widetilde{W},\pi)_+\cO_{S_0\times S_2^0}),
$$
recall that
$$
\widetilde{W}(\underline{y},\underline{q})= \sum_{i=1}^m \underline{y}^{\underline{a}_i} \underline{q}^{\underline{g}_i}=
\sum_{i=1}^m \left(\prod_{k=1}^n y_k^{a_{ki}}\right)\cdot\left(\prod_{a=1}^r q_a^{g_{ia}}\right).
$$
Consider the following cartesian diagram
\begin{equation}
\label{eq:DiagramBaseChange}
\xymatrix@!0{
          S_0\times S_2^0 \ar[dddd]_{\varphi':=(\widetilde{W},\pi)} \ar[rrrrrrr] &&&&&&& S_0\times S_1^0 \ar[dddd]^{\varphi} \\ \\ \\ \\
          \dC_t\times S_2^0 \ar[rrrrrrr]^{(\id_{\dC_t},g)} &&&&&&& U_1=\dC_t\times S_1^0
}
\end{equation}

Now we use the base change properties of the direct image (see, e.g., \cite[section 1.7]{Hotta}), from which we obtain
that
$$
(\id_{\dC_t},g)^+ \cH^0(\varphi_+\cO_{S_0\times S_1^0}) \cong \cH^0(\varphi'_+\cO_{S_0\times S_2^0}).
$$
This gives
$$
\FL^\tau_t(\id_{\dC_t},g)^+ \cH^0(\varphi_+\cO_{S_0\times S_1^0})[\tau^{-1}] \cong \FL^\tau_t\cH^0(\varphi'_+\cO_{S_0\times S_2^0})[\tau^{-1}],
$$
and as we have
$$
\FL^\tau_t\left((\id_{\dC_t},g)^+ \cH^0(\varphi_+\cO_{S_0\times S_1^0})\right) \cong
\widetilde{g}^+\FL^\tau_t\cH^0(\varphi_+\cO_{S_0\times S_1^0}),
$$
we finally obtain
$$
\widetilde{g}^+ \FL^\tau_t\left(\cH^0(\varphi_+\cO_{S_0 \times S_1^0})\right)[\tau^{-1}] = \FL^\tau_t\left(\cH^0(\varphi'_+\cO_{S_0 \times S_2^0})\right)[\tau^{-1}],
$$
from which the desired statement follows using proposition \ref{prop:ReducedAHyperGeom} and theorem \ref{theo:GM-GKZUp}.

\end{proof}

As a consequence of the last result, we have the following easy corollary
concerning the the family of Brieskorn lattices resp. the holonomic duality for
the Gau\ss-Manin-system of the Landau-Ginzburg model $(W,q)$.
\begin{corollary}\label{cor:BrieskornDualityDown}
\begin{enumerate}
\item
The $\cD_{\dC_\tau\times S_2^0}$-module $\qM:=\cO_{\dC_\tau\times S_2^0}\otimes_{\cO_{\dC_\tau\times S_2}}(\cD_{\dC_\tau\times S_2}[\tau^{-1}]/ \widetilde{\cI})$ is equipped with an increasing filtration $G_\bullet$ by $\cO_{\dC_z\times S^0_2}$-modules.
Moreover, for any $k\in\dN$, $G_k\qM$ is $\cO_{\dC_z\times S_2^0}$-locally free of rank $n!\cdot\vol(\Conv(\underline{a}_1,\ldots,\underline{a}_m)$.
\item
Write $\qMBL$ for the $\cO_{\dC_z\times S^0_2}$-module $G_0\qM$, then this is the restriction to $\dC_z\times S_2^0$ of the sheaf associated to
the module
$$
\frac{\dC[z,q^\pm_1,\ldots,q^\pm_r]\langle z\partial_{q_1},\ldots,z\partial_{q_r},z^2\partial_z\rangle}{(\widetilde{\Box}_{\underline{l}})_{\underline{l}\in \dL}+
(z^2\partial_z+\sum_{a=1}^r \rho(p^\vee_a) zq_a\partial_{q_a})}.
$$
\item
There is a non-degenerate flat $(-1)^n$-symmetric pairing
$$
P:\qM \otimes \iota^*\qM \rightarrow \cO_{\dC^*_\tau\times S^0_2}.
$$
\item
$P(\qMBL ,\qMBL )\subset z^n \cO_{\dC_z\times S^0_2}$,
and $P$ is non-degenerate on $\qMBL$.
\end{enumerate}
\end{corollary}

\begin{proof}
As we have seen, the closed embedding $\widetilde{g}_{|\dC_z\times S_2^0}:\dC_z\times S^0_2 \hookrightarrow \dC_z\times S^0_1$ is
non-characteristic for $\cO_{\dC_z\times S_2^0}\otimes_{\cO_{\dC_z\times S_2}}\widehat{\cM}^{loc}_{\widetilde{A}} $.
It is actually nothing else but the inverse image in the category of meromorphic bundles with connections. Hence the increasing filtration $\cO_{\dC_z\times S_1^0}\otimes(z^{-\bullet}\BL)$ on
$\widehat{\cM}^{loc}_{\widetilde{A}}$ by locally free $\cO_{\dC_z \times S_1^0}$-modules pulls back to an increasing
filtration $G_\bullet$ on $\qM$ by locally free $\cO_{\dC_z \times S_2^0}$-modules, the zeroth term of which is given by the formula in 2.
All other statements follow from
proposition \ref{prop:FiltrationDuality}.
\end{proof}

\subsection{Logarithmic extensions}
\label{subsec:LogExtension}

In this subsection, we first construct a logarithmic extension
of the hypergeometric system $\qM$ on the K\"ahler moduli space.
Recall from the last subsection that
$S_2^0$ is a Zariski open subspace of $S_2:=\Spec\dC[\dL]$ consisting of points $q$
such that the Laurent polynomial $\widetilde{W}(-,\underline{q}):S_0\rightarrow \dC_t$ is non-degenerate.
Recall also that we have chosen a basis $p_1,\ldots, p_r$ of $\dL^\vee$ of nef classes, i.e., classes lying
in the K\"ahler cone $\cK \subset \dL^\vee_\dR$. The corresponding coordinates on $S_2$ are $q_1,\ldots,q_r$, and define
an embedding of $S_2$ into $\dC^r$.

Write $\Delta_{S_2}:=S_2\backslash S_2^0$ and denote by
$\overline{\Delta_{S_2}}$ the closure of $\Delta_{S_2}$ in $\dC^r$. Finally, put $\overline{S_2^0}:=\dC^r\backslash  \overline{\Delta_{S_2}}$.
We will write $Z_a$ for the divisor $\{q_a=0\}$ in both $\dC^r$ and $\overline{S_2^0}$, and we define
$Z=\bigcup_{a=1}^r Z_a$ which is a simple normal crossing divisor in $\dC^r$ resp. $\overline{S_2^0}$.
\begin{lemma}\label{lem:LimitPointqCoord}
\begin{enumerate}
\item
The origin of $\dC^r$ is contained in $\overline{S_2^0}$.
\item
If $\XSigA$ is Fano (i.e., $\rho\in\cK^0_{\Sigma_A}$), then $\Delta_{S_2}=\emptyset$, and, hence,
$\overline{S_2^0}=\dC^r$.
\item
If $\Delta_{S_2}\neq \emptyset$, then there is a ball $B:=B_r(0) \subset (\overline{S_2^0})^{an}$ with radius equal to $r:= \inf\{ |q| : q \notin \overline{\Delta_{S_2}} \}$. We set $B:=(\dC^r)^{an}$ if $\Delta_{S_2} = \emptyset$.
\end{enumerate}
\end{lemma}
\begin{proof}
\begin{enumerate}
\item
This has been shown in \cite[appendix 6.1]{Ir2}.
\item
This follows from lemma \ref{lem:FanoNonDeg}.
\item
This is clear from 1.
\end{enumerate}
\end{proof}
We proceed with a construction which is a variant of the arguments used in the proof
of theorem \ref{theo:BrieskornLatticeFree}, however, now we also take into account the logarithmic structure
along $Z$. We first define the appropriate non-commutative algebras, and then
show that they are actually locally free $\cO$-modules, possibly after a further restriction to
some Zariski open subset of $\overline{S_2^0}$.
\begin{definition}
\label{ref:defLogQDMod}
\begin{enumerate}
\item
Consider the ring
$$
\widetilde{R} := \dC[q_1,\ldots,q_r,z]\langle z q_1\partial_{q_1},\ldots, z q_r \partial_{q_r},z^2\partial_z\rangle
$$
i.e., the quotient of the free $\dC[q_1,\ldots,q_r,z]$-algebra generated by
$z q_1 \partial_{q_1}, \ldots, z q_r \partial_{q_r}, z^2\partial_z$ by the left ideal generated
by the relations
$$
\begin{array}{c}
[z q_a \partial_{q_a}, z] = 0, \quad [z q_a \partial_{q_a}, q_b]=\delta_{ab} z q_a, \quad [z^2\partial_z,q_a]=0,\quad [z^2\partial_z, z]= z^2, \\ \\
{[}z q_a \partial_{q_a}, z q_b \partial_{q_b} ]= 0, \quad [z^2\partial_z, z q_a \partial_{q_a}] = z\cdot zq_a\partial_{q_a} \;\;
\end{array}
$$
Write $\widetilde{\cR}$ for the associated sheaf of quasi-coherent $\cO_{\dC_z\times \dC^r}$-algebras, which restricts
to $\cD_{\dC_\tau^*\times S_2}$ on $\{(q_a\neq 0)_{a=1,\ldots,r},z\neq 0\}$.

We also consider the subring $\widetilde{R}':=\dC[q_1,\ldots,q_r,z]\langle z q_1\partial_{q_1},\ldots, z q_r \partial_{q_r}\rangle$
of $\widetilde{R}$, and the associated sheaf $\widetilde{\cR}'$. The
inclusion $\widetilde{\cR}'\hookrightarrow \widetilde{\cR}$ induces a functor from the category of $\widetilde{\cR}$-modules
to the category of $\widetilde{\cR}'$-modules, which we denote by $\textup{For}_{z^2\partial_z}$ (``forgetting the $z^2\partial_z$-structure'').
\item
Let $\widetilde{\cI}$ be the ideal in $\widetilde{\cR}$ generated by the operators $\widetilde{\Box}_{\underline{l}}$
for any $\underline{l}\in \dL$ and by
$z^2\partial_z+\sum_{a=1}^r \rho(p^\vee_a) zq_a\partial_{q_a}$ and consider the quotient $\widetilde{\cR}/\widetilde{\cI}$.
We have $\textup{For}_{z^2\partial_z}(\widetilde{\cR}/\widetilde{\cI})=(\widetilde{\cR}'/(\widetilde{\Box}_{\underline{l}})_{\underline{l}\in \dL})$
and $\widetilde{\cR}/\widetilde{\cI}$ equals $\qMBL$ on $\dC_z\times S_2$ (and hence equals $\qM$ on $\dC_\tau^*\times S_2$).
\end{enumerate}
\end{definition}
The basic finiteness result about the module $\qMBLlog$ is the following.
\begin{theorem}
\label{theo:LogExtQDMod}
There is a Zariski open subset $U$ of $\overline{S_2^0}$ containing the origin in $\dC^r$ such that the module
$\qMBLlog :=\cO_{\dC_z\times U}\otimes_{\cO_{\dC_z\times \dC^r}} \widetilde{\cR}/\widetilde{\cI}$
is $\cO_{\dC_z\times U}$-coherent.
If $\XSigA$ is Fano, i.e., if $\rho\in\cK^0(\Sigma_A)$, then one can choose $U$ to be $\dC^r$ (which equals $\overline{S_2^0}$ in this case).

There is a connection operator
$$
\nabla:\qMBLlog \longrightarrow \qMBLlog \otimes z^{-1}\Omega^1_{\dC_z\times U}\left(\log \left((\{0\}\times U) \cup (\dC_z\times Z)\right)\right)
$$
extending the $\cD_{\dC_\tau^*\times (U\cap S_2)}$-structure on $(\qM)_{|\dC_\tau^*\times (U\cap S_2)}$.
\end{theorem}
\begin{proof}
The arguments used here have some similarities with the proof of theorem \ref{theo:BrieskornLatticeFree}.
We first suppose that $\XSigA$ is Fano, then we have to  show that $\qMBLlog$ is $\cO_{\dC_z\times \dC^r}$-coherent.
We will actually show the coherence of $\textup{For}_{z^2\partial_z}(\qMBLlog)$,
which is sufficient, as $\qMBLlog$ and $\textup{For}_{z^2\partial_z}(\qMBLlog)$ are equal
as $\cO_{\dC_z \times \dC^r}$-modules.
Consider the natural filtration
on $\widetilde{\cR}'$ given by order of operators, i.e., the filtration $F_\bullet\widetilde{\cR}'$ given on global sections by
$$
F_k \dC[q_1,\ldots,q_r,z]\langle zq_1\partial_{q_1},\ldots,zq_r\partial_{q_r}\rangle :=
\left\{P\,|\,P=\sum_{|\underline{s}|\leq k} g_{\underline{s}}(z,\underline{q})(z q_1 \partial_{q_1})^{s_1}\cdot\ldots\cdot(z q_r \partial_{q_r})^{s_r}\right\}.
$$
This filtration induces a filtration $F_\bullet$ on $\textup{For}_{z^2\partial_z}(\qMBLlog)$ which is good in the sense that
$$
F_k \widetilde{\cR}'\cdot F_l \textup{For}_{z^2\partial_z}(\qMBLlog) = F_{k+l} \textup{For}_{z^2\partial_z}(\qMBLlog).
$$
We have a natural identification
$$
\gr_\bullet^F(\widetilde{\cR}') = \pi_*\cO_{\dC_z\times T^* \dC^r(\log\,D)}
$$
where $T^* \dC^r(\log\,D)$ is the total space of the vector bundle associated to
the locally free sheaf $\Omega^1_{\dC^r}(\log\,D)$ and $\pi:\dC_z\times T^* \dC^r(\log\,D)\twoheadrightarrow \dC_z\times \dC^r$ is the projection. It will be sufficient to
show that the subvariety $\dC_z\times S$ of $\dC_z\times T^*\dC^r(\log\,D)$ cut out by the symbols of all operators
$\widetilde{\Box}_{\underline{l}}$ for ${\underline{l}\in\dL}$ equals $\dC_z\times \dC^r$, then by the usual argument the filtration $F_\bullet$
will become eventually stationary, and we conclude by the fact that all $F_k\textup{For}_{z^2\partial_z}(\qMBLlog)$ are
$\cO_{\dC_z\times \dC^r}$-coherent.
For the proof, we will use some elementary facts from toric geometry, namely, the notion
of primitive collections and primitive relations (see \cite{Bat1} and \cite{CoxRenesse}).
Suppose that $\underline{l}\in\dL$ corresponds to a primitive relation in the sense of loc.cit.,
then it follows that $p_a(\underline{l})\geq 0$ for all $a=1,\ldots,r$, as a primitive
relation lies in the Mori cone of $\XSigA$ and as $p_a$ is a nef class, i.e., by definition it takes
non-negative values on effective cycles. On the other hand, as $\XSigA$ is Fano,
we have that $\overline{l}=\rho(\underline{l})>0$, remember that $\rho=\sum_{i=1}^m D_i$ is the anti-canonical divisor which
by definition lies in the interior of the K\"ahler cone. Hence, $\sum_{i:l_i>0} l_i > \sum_{i:l_i<0} -l_i$,
moreover, for a primitive relation, we have $l_i=1$ for all $i$ such that $l_i>0$ (see \cite[proposition 3.1]{Bat1}). This yields
$$
\sigma(\widetilde{\Box}_{\underline{l}}) = \prod_{i:l_i=1} \left(\sum_{a=1}^r m_{ia} \sigma(zq_a\partial_{q_a})\right),
$$
Now identify $T^* \dC^r(\log\,Z)$ with the trivial bundle
$\dC^r \times X$ where $X$ is the vector space dual to the space
generated by $(\sigma(zq_a\partial_{q_a}))_{a=1,\ldots,r}$. Then the last equation shows that
the variety $S$ alluded to above is of the form $\dC^r \times Y_{red}$,
for some possibly non-reduced subvariety $Y\subset X$. We need to show that $Y_{red}=\{0\}$. First it is clear
that $Y$ is homogeneous so that it suffices to show that its
Krull dimension is zero. Recall from \cite[section 5.2, page 106]{Fulton} that
the classical cohomology ring of $\XSigA$ with complex coefficients is presented as
\begin{equation}\label{eq:CohomToric}
H^*(\XSigA,\dC)=\frac{\dC[(v_i)_{i=1,\ldots,m}]}{(\sum_{i=1}^m a_{ki} v_i)_{k=1,\ldots,n}+\left(
v_{i_1}\cdot\ldots\cdot v_{i_p} \right)}
\end{equation}
where the tuple $v_{i_1},\ldots,v_{i_p}$ runs over all primitive collections in $\Sigma_A(1)$.
However, it follows from the exactness of
the sequence \eqref{eq:ExSeq} that the spectrum of this ring equals the above subspace $Y$, in particular
the latter must be fat point, supported at the origin in the space $V$.
This shows that the variety $S$ is the zero section of $T^* \dC^r(\log\,D)$,
as required.

Now suppose only that $\XSigA$ is weak Fano, i.e., $\rho\in\cK_{\Sigma_A}$. Then it may
happen that for a primitive relation $\underline{l}$, we have $\overline{l}=\rho(\underline{l})=0$,
which implies that
$$
\sigma(\widetilde{\Box}_{\underline{l}})=
\prod_{a=1}^r q_a^{p_a(\underline{l})}\prod_{i:l_i<0} \left(\sum_{a=1}^r m_{ia} \sigma(zq_a\partial_{q_a})\right)^{-l_i}-
\prod_{i:l_i=1} \left(\sum_{a=1}^r m_{ia} \sigma(zq_a\partial_{q_a})\right),
$$
as $p_a(\underline{l})\geq 0$ for any primitive relation $\underline{l}$.
This shows that the fibre of $S$ over the point $q_1=0,\ldots,q_r=0$ is again the reduced
space of the spectrum of the cohomology algebra of $\XSigA$, i.e, it is
only the origin in the fibre of $T^*\dC^r(\log D)$ over
$\underline{q}=\underline{0}$. In particular, the projection map $S\twoheadrightarrow \dC^r$ is quasi-finite in a Zariski open neighborhood $U$ of
$\underline{0}\in \dC^r$. On the other hand, by its very definition, $S$ is homogeneous with respect to the fibre variables. Hence on $U$,
$S$ is the zero section of the projection $T^*U(\log D) \twoheadrightarrow U$, as required.

The statement concerning the connection follows directly from the definition of $\qMBLlog$, namely, $\qMBLlog$ is
invariant under the operators $\nabla_{zq_a\partial_{q_a}}$ for $a=1,\ldots,r$ and
$\nabla_{z^2 \partial_z}$.
\end{proof}
The next step is to discuss the restriction $(\qMBLlog)_{|\dC_z\times\{\underline{q}=\underline{0}\}}$, this is a coherent
$\cO_{\dC_z}$-module that we denote by $E$. It turns out that it is actually locally free, from which we deduce
the freeness of $\qMBLlog$ and certain extension properties of the pairing $P$ from corollary \ref{cor:BrieskornDualityDown}, 3.
\begin{lemma}\label{lem:ModuleE}
\begin{enumerate}
\item
There is a canonical isomorphism
$$
\alpha:\cO_{\dC_z}\otimes H^*(\XSigA,\dC)\stackrel{\cong }{\longrightarrow} E,
$$
hence, $E$ is $\cO_{\dC_z}$-free of rank $\mu:=n!\cdot\vol(\Conv(\underline{a}_1,\ldots,\underline{a}_m))$.
It comes equipped with a connection
$$
\nabla^{\mathit{res},\underline{q}}: E \longrightarrow E \otimes z^{-2}\Omega^1_{\dC_z}
$$
induced by the residue connection of $\nabla$ on $(\qMBLlog)_{|\dC^*_\tau\times \{\underline{q}=\underline{0}\}}$ along $\bigcup_{a=1}^r  \{q_a=0\}$.

Let $i:\dC_z\hookrightarrow \dC_z\times U$ be the inclusion and write $\pi:i^{-1}(\qMBLlog) \twoheadrightarrow E$ for the canonical
projection. Let $F=\pi\left(
\dC[zq_1\partial_{q_1},\ldots,zq_r\partial_{q_r}]\right) \subset E$, where we denote abusively by $\dC[zq_1\partial_{q_1},\ldots,zq_r\partial_{q_r}]$
the sheaf associated to the the image of this ring in $\Gamma(\dC_z\times U,\qMBLlog)$. Then $\alpha(1\otimes H^*(\XSigA,\dC)) = F$.

The restriction $E_{|z=0}=(\qMBLlog)_{|(0,\underline{0})}$ is canonically isomorphic, as a finite-dimensional
commutative algebra, to the cohomology ring $\left(H^*(\XSigA,\dC),\cup\right)$.
\item
$\qMBLlog$ is $\cO_{\dC_z\times U}$-free of rank $\mu$.
\item
Write $\qMlog$ for the restriction $(\qMBLlog)_{|\dC^*_\tau\times U}$. Then for any $a\in \{1,\ldots,r\}$, the residue endomorphisms $z
q_a\partial_{q_a}\in
{\cE\!}nd_{\cO_{\dC^*_\tau}}\left((\qMlog)_{|\dC^*_\tau\times \{\underline{0}\}}\right)=E_{|\dC^*_\tau}$
are nilpotent.
\item
There is a non-degenerate flat $(-1)^n$-symmetric pairing
$P:\qMBLlog \otimes \iota^*\qMBLlog \rightarrow z^n \cO_{\dC_z\times U}$,
i.e., $P$ is flat on $\dC^*_\tau\times (U\cap S_2)$, and the induced pairings
$P:(\qMBLlog/z\cdot\qMBLlog) \otimes (\qMBLlog/z\cdot\qMBLlog) \rightarrow z^n\cO_U$
and $P:(\qMBLlog/q_a\cdot\qMBLlog) \otimes \iota^*(\qMBLlog/q_a\cdot\qMBLlog) \rightarrow z^n\cO_{\dC_z\times Z_a}$ are non-degenerate.
\item
The induced pairing $P:E\otimes \iota^*E\rightarrow z^n\cO_{\dC_z}$ restricts to a pairing $P:F\times F\rightarrow z^n\dC$.
The pairing $z^{-n}P$ on $F$ coincides, under the identification made in 1., with the Poincar\'e pairing on $H^*(\XSigA,\dC)$ up to a non-zero constant.
\end{enumerate}
\end{lemma}
\begin{proof}
\begin{enumerate}
\item
In order to construct the map $\alpha$
notice first that we have
$$
\left(\widetilde{\Box}_{\underline{l}}\right)_{|\{\underline{q}=\underline{0}\}}
=\left\{
\begin{array}{rl}
\prod_{i:l_i>0} \prod_{\nu=0}^{l_i-1} (\sum_{a=1}^r m_{ia} zq_a\partial_{q_a} -\nu z)
 & \text{ if } p_a(\underline{l})\geq 0 \text{ for all } a=1,\ldots,r \\ \\
\prod_{i:l_i<0} \prod_{\nu=0}^{-l_i-1} (\sum_{a=1}^r m_{ia} zq_a\partial_{q_a} -\nu z)
 & \text{   if } p_a(\underline{l})\leq 0 \text{ for all } a=1,\ldots,r \\ \\
0 & \text{else}
\end{array}
\right.
$$
Hence
we obtain the following isomorphimsm of $\cO_{\dC_z}$-modules
$$
E=\left(\textup{For}_{z^2\partial_z}
(\qMBLlog)\right)_{|\dC_z\times\{\underline{q}=\underline{0}\}}
\cong\frac{\dC[z,zq_1 \partial_{q_1},\ldots,zq_r \partial_{q_r}]}{\left(\left\{\left(\widetilde{\Box}_{\underline{l}}\right)_{|\{\underline{q}=\underline{0}\}}
\,|\, \underline{l}\in\textup{Eff}_{\XSigA} \cap \dL\right\}\right)},
$$
where $\textup{Eff}_{\XSigA}\subset \dL_\dR$
 is the Mori cone of $\XSigA$.
Notice that if $\underline{l}\in \Leff:=\textup{Eff}_{\XSigA} \cap \dL$, then any $\left(\widetilde{\Box}_{\underline{l}}\right)_{|\{\underline{q}=\underline{0}\}}$ contains
$\prod_{i:l_i\geq 0} (\sum_{a=1}^r m_{ia}zq_a\partial_{q_a})$ as a factor. The Mori cone can be
characterized as follows  (see, e.g., the discussion in \cite[3.4.2]{CK}):
\begin{equation}\label{eq:MoriCone}
\textup{Eff}_{\XSigA}= \sum_{\sigma \in \Sigma_A(n)} C_{\sigma},
\end{equation}
where $C_{\sigma}$ is the cone generated by elements $\underline{l}\in\dL$ with $l_i \geq 0$ whenever $\dR_{\geq 0}\underline{a}_i$ is not a ray of $\sigma$.
It follows that whenever $\underline{l}\in\Leff\backslash\{0\}$, then the set $\{\underline{a}_i \,|\, l_i \geq 0\}$ cannot generate
a cone in $\Sigma_A$, for otherwise $-\underline{l}$ would also lie in $\textup{Eff}_{\XSigA}$, and thus $\underline{l}=0$.
As a consequence, for any $\underline{l}\in\Leff\backslash\{0\}$, the element
$(\widetilde{\Box}_{\underline{l}})_{|\{\underline{q}=\underline{0}\}}$ contains a factor
$\prod_{i\in I} (\sum_{a=1}^r m_{ia}zq_a\partial_{q_a})$ where $\sum_{i\in I}\dR_{\geq 0} \underline{a}_i\notin \Sigma_A$.

Now consider the case where $\underline{l}$ is primitive, in particular, $\underline{l}\in\Leff$. Then
$(\widetilde{\Box}_{\underline{l}})_{|\{\underline{q}=\underline{0}\}}$ is equal to
$\prod_{i\in I} (\sum_{a=1}^r m_{ia}zq_a\partial_{q_a})$, where $\{\underline{a}_i\,|\,i\in I\}$ is
a primitive collection.
As any set of rays $\{\underline{a}_j\,|\,j\in J\}$ which does not generate a cone
contains a primitive collection, we conclude from the above discussion that $E$ is equal to
$$
\frac{\dC[z,zq_1 \partial_{q_1},\ldots,zq_r \partial_{q_r}]}{\left(\left\{\left(\widetilde{\Box}_{\underline{l}}\right)_{|\{\underline{q}=\underline{0}\}}
\;\;|\;\; \underline{l}\textup{ primitive }\right\}\right)}
\cong
\dC[z]\otimes\frac{\dC[zq_1 \partial_{q_1},\ldots,zq_r \partial_{q_r}]}{\left(
\prod_{i\in I} (\sum_{a=1}^r m_{ia}zq_a\partial_{q_a})\right)_I},
$$
where the index set $I$ in the denominator of the right hand side runs over all subsets of $\{1,\ldots,m\}$ such that
$\{\underline{a}_i\,|\,i\in I\}$ is a primitive collection.

Now to define $\alpha$ we use again the presentation of $H^*(\XSigA,\dC)$ from formula \eqref{eq:CohomToric}. We conclude from the above discussion that
putting $\alpha(v_i):=
\sum_{a=1}^r m_{ia} z q_a \partial_{q_a}$ yields a well-defined map $\cO_{\dC_z} \otimes H^*(\XSigA,\dC) \rightarrow E$, which is obviously surjective.
We have seen in theorem \ref{theo:LogExtQDMod} that $\qMBLlog$ is coherent, and its generic rank is that of $\qM$, i.e., $\mu$. On the other
hand, $\cO_{\dC_z} \otimes H^*(\XSigA,\dC)$ is $\cO_{\dC_z}$-free of rank $\mu$, hence by semi-continuity and comparison of rank, we obtain that $\alpha$ is an isomorphism. Then
we also have that $\alpha(H^*(\XSigA,\dC))=F$. The pole order property of the connection operator $\nabla^{res,\underline{q}}$ on $E$ follows from the pole order properties
of $\nabla$ on $\qMBLlog$ as stated in theorem \ref{theo:LogExtQDMod}.

\item
This is now a standard argument: For any $I\subset\{1,\ldots,r\}$, put $\widetilde{Z}_I:=\bigcap_{a\in I}Z_a$ and consider the restriction
$(\qMBLlog)_{|\dC^*_\tau \times Z_I}$, where $Z_I:=\widetilde{Z}_I\backslash\left(\bigcup_{J\supsetneq I} \widetilde{Z}_J\right)$.
This restriction is equipped with the structure of a $\cD_{\dC^*_\tau \times Z_I}$-module, so that
it must be locally free. Hence it suffices to show freeness of $\qMBLlog$ in a neighborhood of $0\in\dC_z\times U$.
But this is clear after from point 1.: The dimension of the fibre at $0$ is
$n!\cdot\vol(\Conv(\underline{a}_1,\ldots,\underline{a}_m))$, which is also the rank  on $\dC_z \times S_2^0$. Hence it can neither be smaller nor bigger
at any point in a neighborhood of the origin in $\dC_z\times U$.
\item
Using the isomorphism $\alpha$ from 1., the residue endomorphism $[zq_a\partial_{q_a}]$ equals $\mathit{Id}_{\cO_{\dC^*_\tau}}\otimes (D_a\cup -)
\in {\cE\!}nd_{\cO_{\dC^*_\tau}}(E_{|\dC^*_\tau})$ from which its nilpotency follows easily.
\item
Using the $\cO_{\dC_z\times U}$-freeness of $\qMBLlog$ and point 5. above, this can be shown by
an argument similar to \cite[lemma 3.4]{HS4}. Namely,
consider
the canonical $V$-filtration (denoted by $V_{\underline{\bullet}}$) on $\qM$ along the normal crossing divisor $Z$. Then
the last point shows that we have
$V_{\mathbf{0}} \qM = \qMlog$ (recall that $\qMlog$ is the restriction
of $\qMBLlog$ to $\dC^*_\tau\times U$), hence,
$\gr_{\mathbf{0}}^V(\qMlog) =(\qMlog)_{|\dC_\tau^*\times\{\underline{0}\}}$.
This implies immediately (see \cite[proof of lemma 3.4 and formula 3.4]{HS4}) that $P$ extends in a non-degenerate way to
$\qMlog$. Hence we obtain a non-degenerate pairing on the restriction
$(\qMBLlog)_{|(\dC_z\times U)\backslash(\{0\}\times Z)}$. However, as $\{0\}\times Z$
has codimension two in $\dC_z\times U$, $P$ necessarily extends to a non-degenerate pairing
on $\qMBLlog$, as required.
\item
The non-degenerate pairing $P:E\otimes \iota^*E\rightarrow z^n\cO_{\dC_z}$ restricts to a pairing $P:F\times F \rightarrow z^n\cO_{\dC_z}$.
Let us show that it actually takes values in $z^n\dC$ on $F$.
Set $r_i = \text{dim} H^{2i}(\XSigA,\dC)$ and choose a homogeneous basis
$$
w_{1,0}=1, w_{1,1}, \ldots ,w_{r_1,1},\ldots ,
w_{1,n-1}, \ldots , w_{r_{n-1},n-1},w_{1,n}$$
where $w_{i,k} \in H^{2k}(\XSigA,\dC)$  and
which is adapted to the Lefschetz decomposition.
Recall that the Hard Lefschetz theorem says the following:
\[
H^m(\XSigA,\dC)= \bigoplus_i L^i H^{m-2i}(\XSigA,\dC)_p\, ,
\]
where $H^{n-k}(\XSigA, \dC)_p = ker(L^{k+1}: H^{n-k}(\XSigA,\dC) \rightarrow H^{n+k+2}(\XSigA,\dC))$ and the map $L$ is equal to cup-product with $c_1(\XSigA)$.
It follows from equation \ref{eq:ConnectAtQ0} that
\begin{align}
z \nabla_{\partial_z}^{\mathit{res},\underline{q}}(w_{i,k}) &= k \cdot w_{i,k} + \frac{1}{z} \sum_{m=1}^{r_{k+1}}\Theta_{m,i,k} w_{m,k+1} \quad \text{for} \quad k < n\, , \notag \\
z \nabla_{\partial_z}^{\mathit{res},\underline{q}}(w_{1,n}) &= n \cdot w_{1,n \notag}\, ,
\end{align}
where $\Theta_{m,i,k} := (\check{A}_0)_{u,v}$ with $u= m + \sum_{l=1}^{k}r_l$ and $v = i + \sum_{l=1}^{k-1}r_l$ and $\check{A}_0$ is the matrix with respect to the basis $w_{1,0}, \ldots, w_{1,n}$ of the endomorphism $-c_1(\XSigA) \cup$.
The first claim is that $P(w_{i,k},w_{j,l}) = c_{ikjl}z^{k+l}$ with $c_{ikjl} \in \dC$. Using the fact that $P$ takes values in $z^n\cO_{\dC_z}$ on $E$, this implies in particular $P(w_{i,k},w_{j,l}) = 0$ for $k+l < n$.
We have
\[
z \partial_z P(w_{1,n},w_{1,n}) = 2n P(w_{1,n},w_{1,n}) \in z^n \cO_{\dC_z}\, ,
\]
thus it follows that $P(w_{1,n},w_{1,n}) = c \cdot z^{2n}$ for some $c \in \dC$. Now assume that we have $P(w_{i,s},w_{j,t}) = c_{isjt} z^{s+t}$ for $c_{isjt} \in \dC$ and $s+t \geq d+1$.
We have for $k+l =d$
\begin{align}
z \partial_z P(w_{i,k}w_{j,l})&= P(k \cdot w_{i,k} + \frac{1}{z}(\sum_{m=1}^{r_{k+1}} \Theta_{m,i,k} w_{m,k+1}),w_{j,l})  \notag \\
&+ P(w_{i,k}, l \cdot w_{j,l} + \frac{1}{z}(\sum_{m=1}^{r_{l+1}}\Theta_{m,j,l}\, w_{m,l+1})) \notag \\
&= (k+l) P(w_{i,k},w_{j,l}) + c \cdot z^d \quad \text{for some} \quad c \in \dC\, ,\notag
\end{align}
where the last equality follows from the inductive assumption for $d+1$ and $d+2$. Thus we have
\[
(z \partial_z - d)^2 P(w_{i,k}w_{j,l}) = 0\, ,
\]
which shows $P(w_{i,k}w_{j,l}) - c \cdot z^d \in z^d \dC$. This shows the first claim, i.e. $P(w_{i,k},w_{j,l}) = c_{ijkl} z^{k+l}$ for $k+l \geq n$ and
$P(w_{i,k},w_{j,l}) = 0$ for $k+l <n$.\\

As a second step we want to show $P(w_{i,k},w_{j,l}) = 0$ for $k+l > n$. We prove this by descending induction, beginning with the case $k+l = 2n$. We first introduce some notation. We say $w_{i,k}$ is primitive if it is not of the form $-c_1(\XSigA) \cup v$ for some $v \in H^{2k-2}(\XSigA,\dC)$. We say ${_q}w_{i,k}\in H^{2k-2q}(\XSigA,\dC)$ is a $q$-th primitive of $w_{i,k}$ if $(-c_1(\XSigA))^q \cup {_q}w_{i,k} = w_{i,k}$. The Hard Lefschetz Theorem tells us that for $2k \geq n$ the element $w_{i,k}$ is never primitive.\\

As the base case we have to prove $P(w_{1,n},w_{1,n})=0$. Let ${_1}w_{1,n}$ be a first primitive of $w_{1,n}$. We have
\begin{align}
0=(z\partial_z- (2n-1))P({_1}w_{1,n},w_{1,n}) =&P((n-1)\cdot{_1}w_{1,n},w_{1,n}) + P(\frac{1}{z}w_{1,n},w_{1,n})\notag \\
+&P({_1}w_{1,n},n\cdot w_{1,n}) -(2n-1)P({_1}w_{1,n},w_{1,n}) \notag \\
=& \frac{1}{z}P(w_{1,n},w_{1,n})\, .\notag
\end{align}
Now assume $P(w_{i,k},w_{j,l}) = 0$ for $k+l \geq s+1$. We will prove $P(w_{i,k},w_{j,l})=0$ for $k+l=s$ by descending induction on $k$. Notice that by $(-1)^w$-symmetry we only have to prove this
for $k \geq l$.

The base case is to show that $P(w_{1,n},w_{j,s-n}) =0$ for $j \in \{1, \ldots , r_{s-n}\}$ (recall that $n+1 \leq s < 2n$). We have to distinguish two cases:\\

I. case: $w_{j,s-n}$ is not primitive. Thus there exists ${_1}w_{j,s-n}$ with $-c_1(\XSigA) \cup {_1}w_{j,s-n} = w_{j,s-n}$. We calculate
\begin{align}
0 &= (z\partial_z - (s-1)) P(w_{1,n}, {_1}w_{j,s-n}) \notag \\
&= P(n \cdot w_{1,n}, {_1}w_{j,s-n}) + P(w_{1,n}, (s-n-1){_1}w_{j,s-n}) + P(w_{1,n}, \frac{1}{z}w_{j,s-n})
- (s-1)P(w_{1,n},{_1}w_{j,s-n}) \notag \\
&= -\frac{1}{z}P(w_{1,n}, w_{j,s-n}). \notag
\end{align}

II. case: $w_{j,s-n}$ is primitive. This means that
\[
w_{j,s-n} \in H^{2s-2n}(\XSigA,\dC)_p = \ker\,\left(c_1(\XSigA)^{3n-2s+1}:H^{2s-2n}(\XSigA,\dC) \rightarrow H^{4n-2s+2}(\XSigA,\dC)\right).
\]
We have
\begin{align}
0 =& (z \partial_z - (s-1))P({_1}w_{1,n},w_{j,s-n}) \notag \\
=&P((n-1)\cdot {_1}w_{1,n}, w_{j,s-n}) + P(\frac{1}{z}w_{1,n},w_{j,s-n}) + P({_1}w_{1,n}, (s-n)\cdot w_{j,s-n}) \notag \\
+&P({_1}w_{1,n}, \frac{1}{z} (-c_1(\XSigA)) \cup w_{j,s-n}) - (s-1)P({_1}w_{1,n},w_{j,s-n}) \notag \\
=& P(\frac{1}{z}w_{1,n},w_{j,s-n}) + P({_1}w_{1,n}, \frac{1}{z} (-c_1(\XSigA)) \cup w_{j,s-n})\, , \notag
\end{align}
which gives $P(w_{1,n},w_{j,s-n}) = P({_1}w_{1,n}, c_1(\XSigA) \cup w_{j,s-n})$. Notice that $3n-2s < n$. Because $w_{1,n}$ has an $n-th$-primitive (this follows from
the Hard Lefschetz theorem: $c_1(\XSigA)^n: H^0(\XSigA,\dC) \overset{\simeq}{\longrightarrow} H^{2n}(\XSigA,\dC)$), we can repeat this step $3n-2s+1$ times to get
\[
P(w_{1_n},w_{j,s-n}) = P({_{(3n-2s+1)}}w_{1,n}, (-c_1(\XSigA))^{3n-2s+1} \cup w_{j,s-n}) = 0.
\]
This shows the second case.

We now assume that $P(w_{i,k},w_{j,l})=0$ for $k \geq t+1$ and $k+l =s$ as well as $P(w_{i,k},w_{j,l})=0$ for $k+l \geq s+1$. We have to prove $P(w_{i,t},w_{j,s-t})=0$ for $i \in \{1, \ldots r_t\}$ and $j \in \{1, \ldots , r_{t-s}\}$ and $t \geq s-t$ (the last restriction is allowed because of the $(-1)^w$-symmetry of $P$).

I. case: $w_{j,s-t}$ is not primitive: Thus there exists ${_1}w_{j,s-t}$ with $-c_1(\XSigA) \cup {_1}w_{j,s-t} = w_{j,s-t}$. We calculate
\begin{align}
0 =& (z\partial_z - (s-1)) P(w_{i,t},{_1}w_{j,s-t}) \notag \\
=& P(t \cdot w_{i,t},{_1}w_{j,s-t}) + P(\frac{1}{z}(-c_1(\XSigA) \cup w_{i,t}),{_1}w_{j,s-t}) + P(w_{i,t},(s-t-1)\cdot{_1}w_{j,s-t})\notag \\
+& P(w_{i,t},\frac{1}{z} w_{j,s-t}) - (s-1)P(w_{i,t},{_1}w_{j,s-t}) \notag \\
=& P(\frac{1}{z} (-c_1(\XSigA) \cup w_{i,t}),{_1}w_{j,s-t}) + P(w_{i,t},\frac{1}{z} w_{j,s-t}) \notag \\
=& P(w_{i,t},\frac{1}{z} w_{j,s-t}). \notag
\end{align}
Notice that $P(c_1(\XSigA) \cup w_{i,t},{_1}w_{j,s-t})$ vanishes because $c_1(\XSigA) \cup w_{i,t}$ is a linear combination of $\{w_{i,t+1}\}$ and $P(w_{i,t+1}, {_1}w_{j,s-t})$ vanishes for every $i \in \{1, \ldots , r_{t+1}\}$ by the induction hypothesis.

II. case: $w_{j,s-t}$ is primitive. This means that
\[
w_{j,s-t} \in H^{2s-2t}(\XSigA,\dC)_p = ker\left(c_1(\XSigA)^{n+ 2t -2s+1}: H^{2s-2t}(\XSigA,\dC) \rightarrow H^{2n-2s +2t +2}(\XSigA,\dC) \right).
\]
Notice that $w_{i,t}$ has a $(2t -n)$-th primitive and we have $2t-n \geq n +2t -2s +1$, because of $s \geq n+1$.
We calculate
\begin{align}
0 =& (z\partial_z - (s-1))P({_1}w_{i,t},w_{j,s-t}) \notag \\
=& P((t-1) \cdot {_1}w_{i,t},w_{j,s-t}) + P(\frac{1}{z} w_{i,t},w_{j,s-t}) + P({_1}w_{i,t},(s-t) \cdot w_{j,s-t}) \notag \\
+& P({_1}w_{i,t},\frac{1}{z} (-c_1(\XSigA)) \cup w_{j,s-t}) - (s-1)P({_1}w_{i,t},w_{j,s-t}) \notag \\
=& P(\frac{1}{z} w_{i,t},w_{j,s-t}) + P({_1}w_{i,t},\frac{1}{z} (-c_1(\XSigA)) \cup w_{j,s-t}) \notag
\end{align}
which gives $P(w_{i,t},w_{j,s-t}) = P({_1}w_{i,t}, (-c_1(\XSigA)) \cup w_{j,s-t})$. As $w_{i,t}$ has a $(2t-n)$-th primitive we can repeat this step $n +2t -2s+1$ times to get
\[
P(w_{i,t},w_{j,s-t}) = P({_{n+2t-2s+1}}w_{i,t}, (-c_1(\XSigA))^{n+2t-2s+1} \cup w_{j,s-t}) = 0\,.
\]
This finishes the induction over $t$. Thus we have shown that $P(w_{i,k},w_{j,l})=0$ if $k+l=s \geq n+1$ and $k \geq l$. The case $ k \leq l$ follows by symmetry and
this finishes the induction over $s$. This means that the pairing $P: F \times F \longrightarrow z^n \cO_{\dC_z}$ takes values in $z^n \dC$.

It remains to show that the pairing $z^{-n}P$ coincides, under the isomorphism $\alpha:1\otimes H^*(\XSigA,\dC)\rightarrow F$ and possibly up to a
non-zero constant, with the Poincar\'e pairing on the cohomology algebra. First notice that by construction, $z^{-n}P$, seen as defined
on $H^*(\XSigA,\dC)$ is multiplication invariant, i.e., $P(a,b)=P(1,a\cup b)$ for any two classes $a,b\in H^*(\XSigA,\dC)$.
This can be deduced from the flatness of $P$ on $\qM$, more precisely, by considering the restriction of $P$ defined
on the family of commutative algebras $\qMBLlog/z\cdot \qMBLlog$. Notice however that this argument holds a priori only modulo $z$,
and in order to obtain the multiplication invariance of $z^{-n}P$ on $1\otimes H^*(\XSigA,\dC)$ one first needs to know that it takes constant
values on that space. It suffices now
to show that $P(1,a)$ equals the value of the Poincar\'e pairing on $1$ and $a$. But as we have seen above, $P(1,a)$ can only be non-zero
if $a\in H^{2n}(\XSigA,\dC)$, so that the $P$ on $H^*(\XSigA,\dC)$ is entirely determined by the non-zero complex number $P(1,\mathit{PD}([\mathit{pt}]))$.

\end{enumerate}
\end{proof}
\textbf{Remark:} The value of the pairing $P$ at the point $(0,\underline{0})\in\dC_z\times U$ is determined, by the above
argument, up to multiplication by a non-zero complex number. In order to simplify the statements of the subsequent results, we will
without further mentioning assume that this number is chosen such that $P$ corresponds under the above identifications exactly
to the Poincar\'e pairing on $H^*(\XSigA,\dC)$. Such a choice is always possible by changing
the morphism $\phi:M_{\widetilde{A}}=M^{(1,\underline{0})}_{\widetilde{A}}\rightarrow \bD M_{\widetilde{A}}=M^{(0,\underline{0})}_{\widetilde{A}}$ from
the proof of theorem \ref{theo:Duality} by multiplication by a non-zero complex number (and these are the only non-trivial
morphisms between these two modules, due to \cite[theorem 3.3(3)]{SaitoMut1}).\\

We now show how to construct a specific basis of $\qMBLlog$
defining an extension to a family of trivial $\dP^1$ parameterized by an analytic neighborhood of the origin in $U$ and such that
the connection has a logarithmic pole at $z=\infty$.
As already mentioned in the introduction, the method goes back to \cite{Guest}, namely,
we first construct an extension of $E=(\qMBLlog)_{|\dC_z\times\{\underline{0}\}}$ to $\dP^1_z\times\{\underline{0}\}$
and then show that it can be extended to a family of $\dP^1$-bundles restricting to $\qMBLlog$ outside $z=\infty$.
At any point $\underline{q}$ near the origin in $U$ this yields
a solution to the Birkhoff problem (in other words, a good base in the sense of \cite{SM}) of the restriction
of $(\qMBL)_{|\dC_z\times\{\underline{q}\}}$, but it also gives an extension of the whole family $\qMBLlog$ taking into
account the logarithmic degeneration behavior at $D$.

\begin{proposition}\label{prop:ExtensionInftyAtZero}
Consider the $\cO_{\dC_z}$-module $E$ with the connection $\nabla^{\mathit{res},\underline{q}}$ and the subspace $F\subset E$ from
lemma \ref{lem:ModuleE}.
\begin{enumerate}
\item
The connection operator $\nabla^{\mathit{res},\underline{q}}:E\rightarrow z^{-2}\cdot E$ sends $F$ into $z^{-2}F\oplus z^{-1}\cdot F$.
\item
Let $\widehat{E}:=\cO_{\dP^1_z\times\{\underline{0}\}} \cdot F $ be an extension of $E$ to
a trivial $\dP^1$-bundle. Then the connection $\nabla^{\mathit{res},\underline{q}}$ has a logarithmic pole at $z=\infty$ with spectrum (i.e.,
set of residue eigenvalues) equal to the (algebraic) degrees of the cohomology classes of $H^*(\XSigA,\dC)$. This logarithmic extension
corresponds to an increasing filtration $F_\bullet$ on the local system $E_{|\dC_\tau^*}^{\mathit{an},\nabla^{\mathit{res},\underline{q}}}$ by
subsystems which are invariant under the monodromy of $\nabla^{\mathit{res},\underline{q}}$.
Let $j_\tau:\dC^*_\tau\hookrightarrow (\dP^1_z\backslash\{0\})$,
and put $E^\infty:=\psi_\tau j_{\tau,!} (E^{an})_{|\dC_\tau^*}^{\nabla^{\mathit{res},\underline{q}}}$,
where $\psi_\tau$ is Deligne's nearby cycle functor.
Then $F_\bullet$ is defined on $E^\infty$,
and there is an isomorphism $H^0(\dP^1_z,\widehat{E})=F\rightarrow E^\infty$.
\item
Write $N_a$ for the nilpotent part of the monodromy
of $(\qM)^{an,\nabla}$ around $\dC^*_\tau\times Z_a$, then $N_a$ acts on $E^\infty$ and satisfies
$N_a F_\bullet \subset F_{\bullet-1}$.
\item
The pairing $P$ on $E$ extends to a non-degenerate
pairing $P:\widehat{E}\otimes_{\cO_{\dP^1}} \iota^*\widehat{E} \rightarrow \cO_{\dP^1}(-n,n)$,
where $\cO_{\dP^1}(a,b)$ is the subsheaf of $\cO_{\dP^1}(*\{0,\infty\})$ consisting of meromorphic
functions with a pole of order $a$ at $0$ and a pole of order $b$ at $\infty$.
\end{enumerate}
\end{proposition}
\begin{proof}
\begin{enumerate}

\item
Let $w_1,\ldots,w_\mu$ be a $\dC$-basis of $F$ which consists
of monomials in $zq_a\partial_{q_a}$.
We will show that
\begin{equation}\label{eq:ConnectAtQ0}
(z^2\nabla^{\mathit{res},\underline{q}}_z)(\underline{w})=\underline{w} \cdot (A_0+zA_\infty),
\end{equation}
where $A_0,A_\infty\in M(\mu\times\mu,\dC)$ and that the eigenvalues of $A_\infty$ are
exactly the set (counted with multiplicity) of the (algebraic) degrees of the cohomology classes of $\XSigA$. First notice
that under the identification of $H^*(\XSigA,\dC)$ with the quotient
$\dC[(v_i)_{i=1,\ldots,m}]/\left((\sum_{i=1}^m a_{ki} v_i)_{k=1,\ldots,n}+(
v_{i_1}\cdot\ldots\cdot v_{i_p} )\right)$ in formula \eqref{eq:CohomToric}, a ray $v_i$
is mapped to the cohomology class in $H^2(\XSigA, \dC)$ of the torus invariant divisor
it determines.

From the definition of $\qMBL$ we see that
\begin{equation}\label{eq:ConnMatrixQ0B-Side}
\begin{array}{rcl}
(z^2\nabla^{\mathit{res},\underline{q}}_z)(zq_{b_i}\partial_{q_{b_i}})^{k_i} & = & (z^2\partial_z)\cdot (zq_{b_i}\partial_{q_{b_i}})^{k_i} \\ \\
& = & (zq_{b_i}\partial_{q_{b_i}})^{k_i}\cdot (z^2\partial_z) + {k_i}\cdot z \cdot (zq_{b_i}\partial_{q_{b_i}})^{k_i} \\ \\
& = & \left[-\sum\limits_{a=1}^r \rho(p^\vee_a) zq_a\partial_{q_a}\right]\cdot (zq_{b_i}\partial_{q_{b_i}})^{k_i} + {k_i}\cdot z \cdot (zq_{b_i}\partial_{q_{b_i}})^{k_i}
\end{array}
\end{equation}
Hence $A_\infty$ is diagonal with eigenvalues equal to the
algebraic cohomology degrees of $H^*(\XSigA,\dC)$.

As a by-product of the above calculation, we also see that the endomorphism of $E/z\cdot E$ represented by the
matrix $A_0$ is the multiplication with $-c_1(\XSigA)$, and hence, is nilpotent. With a little more work, this shows
that $\nabla^{\mathit{res},\underline{q}}$ has a regular singularity at $z=0$ on $E$. However, as we are not going to use this fact in the sequel, we will
not give the complete proof here. In any case, we see that $[A_\infty, A_0] = A_0$.

\item
Formula \eqref{eq:ConnectAtQ0} and formula \eqref{eq:ConnMatrixQ0B-Side} show that the connection $\nabla^{\mathit{res},\underline{q}}$ has a logarithmic pole at $z=\infty$ on $\widehat{E}$
with residue eigenvalues equal to the algebraic cohomology degrees of the cohomology classes of $H^*(\XSigA,\dC)$.
The correspondence between logarithmic extensions of flat bundles over a divisor and
filtrations on the corresponding local system is a general fact, see, e.g., \cite[III.1.ab]{Sa4} or \cite[lemma 7.6 and lemma 8.14]{He3}.
The isomorphism $F\rightarrow E^\infty$ is explicitly given by multiplication by $z^{-A_\infty}\cdot z^{-A_0}$.
\item
We have seen in the proof of theorem \ref{lem:ModuleE}, 4., that $E_{|\dC^*_\tau}\cong \gr^V_{\mathbf{0}}\qMlog$ as flat bundles. $N_a$ naturally acts on the latter one,
and is flat with respect to the residue connection $\nabla^{\mathit{res},\underline{q}}$, hence it acts on $E_{|\dC^*_\tau}^{\mathit{an},\nabla^{\mathit{res},\underline{q}}}$ and thus on $E^\infty$. Under the identification of 2.,
the filtration $F_\bullet$ is induced by
$$
F_p = \sum\limits_{|k|\geq -p} \dC\left((zq_1\partial_{q_1})^{k_1}\cdot\ldots\cdot(zq_r\partial_{q_r})^{k_r}\right).
$$
Notice that the only non-trivial filtration steps are those for $p\in[-n,0]$, which corresponds
to the residue eigenvalues of $z^{-1}\nabla_{z^{-1}}=-z\nabla_z$ on $\widehat{E}$ at $z=\infty$ (see formula \eqref{eq:ConnMatrixQ0B-Side} above).
By definition,
$N_a$, seen as defined on $F$ is simply the multiplication by $zq_a\partial_{q_a}$, from which it follows that $N_a F_\bullet\subset F_{\bullet-1}$.
\item
This follows directly from lemma \ref{lem:ModuleE}, 4. and from the definition of $\widehat{E}$.
\end{enumerate}
\end{proof}
The next result gives an extension of $\qMBLlog$ to a family of trivial $\dP^1$-bundles, possibly after restricting
to a smaller open subset inside $U$.
\begin{proposition}\label{prop:LogTrTLEP}
There is an analytic open subset $U^0\subset U^{an}$ still containing the origin of $\dC^r$ and a holomorphic bundle $\qMBLhat \rightarrow \dP^1_z\times U^0$
(notice that here $\;\widehat{\,}\;$ signifies an extension to $z=\infty$, this should not be confused with notation for the partial Fourier-Laplace transformation used before) such that
\begin{enumerate}
\item $(\qMBLhat)_{|\dC_z \times U^0} \cong (\qMBLlog^{an})_{|\dC_z \times U^0}$
\item $(\qMBLhat)_{|\dP^1_z\times \{0\}} \cong \widehat{E}$
\item $\qMBLhat$ is a family of trivial $\dP^1_z$-bundles, i.e., $\qMBLhat=p^*p_*(\qMBLhat)$,
where $p:\dP^1_z\times U^0
\rightarrow U^0$ is the projection.
\item The connection $\nabla$ has a logarithmic pole along $\widehat{\cZ}$ on $\qMBLhat$,
where $\widehat{\cZ}$ is the normal crossing divisor $\left(\{z=\infty\}\cup\bigcup_{a=1}^r \{q_a=0\}\right)\cap \dP^1_z\times U^0$.
\item
The given pairings $P:\qMBLlog \otimes \iota^*\qMBLlog \rightarrow z^n\cO_{\dC_z\times U}$ and $P:\widehat{E}\otimes_{\cO_{\dP^1_z}} \iota^*\widehat{E} \rightarrow \cO_{\dP^1_z}(-n,n)$ extend to a non-degenerate pairing
$P: \qMBLhat \otimes_{\cO_{\dP^1_z\times U^0}} \iota^* \qMBLhat \rightarrow \cO_{\dP^1_z\times U^0}(-n,n)$, where the latter sheaf is
defined as in point 4. of proposition \ref{prop:ExtensionInftyAtZero}.
\item
The residue connection
$$
\nabla^{\mathit{res},\tau}:\qMBLhat/\tau \cdot \qMBLhat \longrightarrow \qMBLhat/\tau \cdot \qMBLhat\otimes\Omega^1_{\{\infty\}\times U^0}(\log(\{\infty\}\times Z)).
$$
has trivial monodromy around $\{\infty\}\times Z$ and any element of $F\subset H^0(\dP^1_z\times U^0, \qMBLhat)$ is horizontal for $\nabla$.
\end{enumerate}
\end{proposition}
\begin{proof}Recall that $\qMlog$ is the restriction of
$\qMBLlog$ to $\dC^*_\tau\times U$.
The strategy of the proof is to show that there is a holomorphic bundle $\qMhat$ on $(\dP^1_z \backslash\{0\})\times B$ (where $B$ is the analytic neighborhood of $0\in\dC^r$ which was defined in lemma \ref{lem:LimitPointqCoord}) which is an extension of $(\qMlog^{an})_{|\dC^*_\tau\times B}$ over $z=\infty$ with
a logarithmic pole along $\widehat{\cZ}^{an}\cap (\dP^1_z\times B)$ and such that the bundle
obtained by gluing this extension to $\qMBLlog$
is a family of trivial $\dP^1_z$-bundles, possibly
after restricting to some open subset $\dP^1_z\times U^0$ of $\dP^1_z\times B$.

A logarithmic extension of $(\qM)^{an}_{|\dC^*_\tau\times (B\cap S^{an}_2)}$ over $\widehat{\cZ}^{an}\cap (\dP^1_z\times B)$ is given by a $\dZ^{r+1}$-filtration on the local system
$\cL=(\qM)^{an,\nabla}$ which is split iff the extension is locally free (see \cite[lemma 8.14]{He3}).
In our situation,
the bundle $\qMlog$ already yields a logarithmic extension over $\dC^*_\tau\times Z$
and we are seeking a bundle $\qMhat \rightarrow (\dP^1_z\backslash \{0\})\times B$
restricting to $\qMlog^{an}$ on $\dC^*_\tau \times B$. Moreover,
the $\dZ^r$-filtration $P_{\underline{\bullet}}$ corresponding to $\qMlog^{an}$ is trivial,
as this bundle is a Deligne extension due to lemma \ref{lem:ModuleE}, 4.
It follows that if we choose an extra single filtration $F_\bullet$ on $\cL$ (this will be the one which define the extension $\qMhat$ over $\{z=\infty\}$), then
 the corresponding
$\dZ^{r+1}$-filtration $\widetilde{P}_{\underline{\bullet}}:=(F_\bullet,P_{\underline{\bullet}})$ will automatically be split.
Write $L^\infty$ for the space $\psi_{\tau}(\psi_{q_1}(\ldots(\psi_{q_r}(j_!\cL)\ldots)))$, where $j:\dC_\tau^*\times (U\backslash Z)^{an}
\hookrightarrow (\dP^1_z\backslash\{0\})\times U^{an}$ i.e., $L^\infty$ is the space of multivalued flat sections of $\qM$. The
basic fact used in order to construct $F_\bullet$ is that we have $L^\infty = \psi_{\tau} j_{\tau,!}\left((\qMlog^{an})_{|\dC^*_\tau\times\{\underline{0}\}}^{\nabla^{\mathit{res},\underline{q}}}\right)$. This is again due to
lemma \ref{lem:ModuleE}, 4. More precisely, we have already seen that $V_{\mathbf{0}} \qM = \qMlog$, i.e.,
$\gr_{\mathbf{0}}^V(\qMlog) =(\qMlog)_{|\dC^*_\tau\times\{\underline{0}\}}=E_{|\dC^*_\tau}$,
where $V_{\underline{\bullet}}$ is
the canonical $V$-filtration on $\qM$ along the normal crossing divisor $Z$,
and then the statement follows from the comparison theorem for nearby cycles.

Now we have already constructed an extension of $(\qMlog)_{|\dC^*_\tau\times\{\underline{0}\}}$
to $(\dP^1_z\backslash\{0\})\times \{\underline{0}\}$: namely, the chart at $z=\infty$ of the bundle $\widehat{E}$ from proposition
\ref{prop:ExtensionInftyAtZero}, and we have seen in point 3 of this proposition that it is encoded
by a filtration $F_\bullet$ on $\psi_{\tau} j_{\tau,!}\left((\qMlog^{an})_{|\dC^*_\tau\times\{\underline{0}\}}^{\nabla^{\mathit{res},\underline{q}}}\right)$.
Hence we obtain a filtration $F_\bullet$ on $L^\infty$ that we are looking for.
As explained above, this yields a split $\dZ^{r+1}$-filtration $\widetilde{P}_{\underline{\bullet}}$ giving rise to
a bundle $\qMhat \rightarrow \left((\dP^1_z\backslash \{0\})\times B\right)$ with logarithmic poles along $\widehat{\cZ}^{an}\cap(\dP^1_z\times B)$, and by construction this bundle restricts to $\qMlog$ on $\dC^*_\tau \times B$ and to $\widehat{E}_{|(\dP^1_z\backslash\{0\})\times\{\underline{0}\}}$
on $(\dP^1_z\backslash\{0\})\times\{\underline{0}\}$. Hence we can glue $\qMhat$ and
$\qMBLlog^{an}$ on $\dC^*_\tau \times B$ to a holomorphic $\dP^1_z\times B$-bundle. Its restriction
to $\dP^1_z \times\{\underline{0}\}$ is trivial, namely, it is the bundle $\widehat{E}$
constructed in proposition \ref{prop:ExtensionInftyAtZero}. As triviality is an open condition,
there exists an open subset (with respect to the analytic topology) $U^0 \subset B$ such that
the restriction of this bundle to $\dP^1_z\times U^0$, which we call
$\qMBLhat$, is fibrewise trivial, i.e., satisfies $\qMBLhat = p^* p_* \qMBLhat$.
This shows the points 1. to 4.

Concerning the statement on the pairing, notice
that the flat pairing $P$ defined on $L^\infty$ gives rise to a pairing on
$\psi_{\tau} j_{\tau,!}\left((\qMlog^{an})_{|\dC^*_\tau\times\{\underline{0}\}}^{\nabla^{\mathit{res},\underline{q}}}\right)$.
Then the pole order property of $P$ on $\widehat{E}$ at $z=\infty$ can be encoded by an orthogonality
property of the filtration $F_\bullet$ with respect to that pairing (the one defined on
$\psi_{\tau} j_{\tau,!}\left((\qMlog^{an})_{|\dC^*_\tau\times\{\underline{0}\}}^{\nabla^{\mathit{res},\underline{q}}}\right)$)
see, e.g., \cite[theorem 7.17 and definition 7.18]{He4}. Hence the very same property must hold
for $P$ and $F_\bullet$, seen as defined on $L^\infty$, so that we conclude that
we obtain $P: \qMBLhat \otimes_{\cO_{\dP^1_z\times U^0}} \iota^* \qMBLhat \rightarrow \cO_{\dP^1_z\times U^0}(-n,n)$,
as required.

Finally, let us show the last statement:
It follows from the correspondence between
monodromy invariant filtrations and logarithmic poles used above
that the residue connection $\nabla^{\mathit{res},\tau}$ along $z=\infty$ on
$\qMBLhat/z^{-1} \qMBLhat$ has trivial monodromy around $Z$ if for any $a=1,\ldots,r$,
the nilpotent part $N_a$ of the monodromy of $\nabla$ on the local system $(\qM)^\nabla$
kills $\gr_\bullet^F$, i.e., $N_a F_\bullet \subset F_{\bullet-1}$. Now by the above identification,
we can see $F_\bullet$ as defined on $\psi_{\tau} j_{\tau,!}\left((\qMlog^{an})_{|\dC^*_\tau\times\{\underline{0}\}}^{\nabla^{\mathit{res},\underline{q}}}\right)$,
and then $N_a F_\bullet\subset F_{\bullet-1}$ has been shown in proposition \ref{prop:ExtensionInftyAtZero}, 3.
It follows directly from the above construction that all elements of $F$, seen as
global sections over $\dP^1_z\times U^0$ are horizontal for $\nabla^{\mathit{res},\tau}$.
\end{proof}
\textbf{Remark: } If the algebraic subset $\Delta_{S_2}=S_2\backslash S_2^0$, i.e., the subspace on which the Laurent
polynomial $W(-,\underline{q}):S_0 \rightarrow \dC_t$ is degenerate, is a divisor, then additional
monodromy phenomena may occur. For this reason, the bundle $\qMBLlog$ cannot in general be extended as an algebraic bundle
over a Zariski open subset of $\dP^1\times U$. Such an extension a priori can only be defined on some covering space
of a Zariski open subset of $\dP^1\times U$.
The choice of this covering space depends on the
structure of the fundamental group of $U$, which is not a priori known.
We therefore
restrict ourselves to the construction of an analytic extension parameterized by the ball $B$. Notice however that if
$\XSigA$ is Fano, then $\qMBLhat$ exists as an algebraic family of $\dP_z^1$-bundles on
some Zariski open subset of $\dC^r$. \\

At this point it is convenient to introduce the so-called $I$-function of the
toric variety $X_{\Sigma_A}$. We follow the definition of Givental (see \cite{Giv7}),
and relate this function to the hypergeometric module $\qM$ discussed above.
\begin{definition}\label{def:IFunction}
Define $I$ resp. $\widetilde{I}$ to be the $H^*(\XSigA,\dC)$-valued formal power series
$$
I= e^{\delta/z} \cdot \sum\limits_{\underline{l}\in\dL} q^{\underline{l}}\; \cdot
\prod\limits_{i=1}^m\frac{\prod_{\nu=-\infty}^0\left([D_i]+\nu z\right)}{
\prod_{\nu=-\infty}^{l_i}\left([D_i]+\nu z\right)}
\in H^*(\XSigA,\dC)[z][[q_1, \ldots , q_r]][[z^{-1},t_1,\ldots,t_r]].
$$
resp. $\widetilde{I}=z^{-\rho}\cdot z^{\mu} \cdot I$.
Here we have set $q^{\underline{l}} = \prod_{a=1}^r q_a^{p_a(\underline{l})}$ and $(t_1,\ldots,t_r)$ are the coordinates on $H^2(\XSigA,\dC)$ induced by the
basis $(p_1,\ldots,p_r)$ of $\dL^\vee$ which were chosen at the beginning of subsection \ref{subsec:QDModDown}. Notice that
$\delta=\sum_{a=1}^r t_a p_a$ is a cohomology class in $H^2(\XSigA,\dC)$.
Later we will set  $q_i = e^{t_i}$ for $i = 1, \ldots ,r$.
As before $\rho=\sum_{i=1}^m[D_i]\in\dL^\vee$ is the anti-canonical class of $\XSigA$ and we write $\mu \in \Aut(H^*(\XSigA,\dC))$ for the
grading automorphism which take the value $k\cdot c$ on a homogeneous class $c\in H^{2k}(\XSigA,\dC)$.
\end{definition}
We collect the main properties of the $I$-function that we will need in the sequel. Most of the statements of
the next proposition are well-known, but rather scattered in the literature.
\begin{proposition}\label{prop:IFunction}
\begin{enumerate}
\item
We have
\begin{equation}\label{eq:IFunctionSecondDef}
\widetilde{I} = \Gamma(T \XSigA)\cdot e^\delta \cdot z^{-\rho}\cdot
\sum\limits_{\underline{l}\in\dL}
\frac{q^{\underline{l}}\; \cdot z^{-\overline{l}}}{\prod_{i=1}^m \Gamma(D_i+l_i+1)},
\end{equation}
where $\Gamma(T \XSigA) := \prod_{i=1}^m \Gamma(1+D_i)$. Moreover,
\begin{equation}\label{eq:I-NoPoles}
e^{-\delta/z}\cdot I, z^\rho\cdot e^{-\delta}\cdot \widetilde{I} \in H^*(\XSigA,\dC)[[q_1,\ldots,q_r,z^{-1}]],
\end{equation}
that is, these series are univalued and have no poles in $\{z=\infty\}\cup\bigcup_{a=1}^r\{q_a=0\}$.
\item
$I$ has the development
$$
I=1+\gamma(q_1,\ldots,q_r) \cdot z^{-1} + o(z^{-1})
$$
where $\gamma= \delta + \gamma'(q_1,\ldots,q_r)$ lies in $\delta+H^2(\XSigA,\dC)[[q_1,\ldots,q_r]]$.
If $\XSigA$ is Fano, then $\gamma'=0$.
\item
There is an open neighborhood $S$ of $\underline{0}$ in $\dC^{r,an}$ such that both $e^{-\delta/z}\cdot I$ and $z^\rho\cdot e^{-\delta}\cdot \widetilde{I}$
are elements in $H^*(\XSigA,\dC)\otimes\cO^{an}_{\dC^*_\tau\times S^*}$, where
$S^*:=S \cap S^0_2$. In particular,
if we put $\kappa:=\underline{q}\cdot e^{\gamma'}$
then $\kappa$ lies in $(\cO^{an}_S)^r$ and defines a coordinate change on $S$.
Notice that in the Fano case, $\kappa$ is the identity, in general it is called the \textbf{mirror map}. It will reappear
in theorem \ref{theo:IgleichJ} and proposition \ref{prop:IsologtrTLEP}.
\item
Write $\pi:(\widetilde{\dC^*_\tau\times S^*})^{an}\rightarrow (\dC^*_\tau\times S^*)^{an}$ for the universal cover, then
for any linear function $h\in (H^*(\XSigA,\dC))^\vee$, we have
$$
h\circ \widetilde{I} \in \dH^0\left((\widetilde{\dC^*_\tau\times S^*})^{an},\pi^*{\cS\!}ol^\bullet(\qM)\right)
=
H^0\left((\widetilde{\dC^*_\tau\times S^*})^{an},
\pi^*{\cH\!}om_{\cD_{\dC^*_\tau\times S^*}}(\qM,\cO_{\dC_\tau^*\times S^*})\right)
$$
\item
For all $h\in (H^*(\XSigA,\dC))^\vee$, if $h\circ \widetilde{I}=0$, then $h=0$, in other words, $\widetilde{I}$ yields a fundamental system of solutions of $(\qM)_{|\dC^*_\tau\times S^*}$.
\end{enumerate}
\end{proposition}
\begin{proof}
\begin{enumerate}
\item
From $z^{\mu} \cdot \delta/z = \delta \cdot z^{\mu}$ and $z^{\mu} \cdot D_i /z = D_i \cdot z^{\mu}$ we deduce
{\small\begin{align}
z^{- \rho} \cdot z^{\mu} \cdot I &= z^{-\rho} \cdot e^{\delta} \cdot \sum\limits_{\underline{l}\in\dL} q^{\underline{l}}\;\cdot
\prod\limits_{i=1}^m\frac{\prod_{\nu=-\infty}^0 z \left([D_i] +\nu \right)}{
\prod_{\nu=-\infty}^{l_i} z \left([D_i] +\nu \right)} \notag \\
&= e^{\delta} \cdot z^{-\rho} \cdot \sum\limits_{\underline{l}\in\dL} q^{\underline{l}}\;\cdot
\prod\limits_{l_i \geq 0} \frac{\Gamma(D_i +1)}{\Gamma(D_i +1) \cdot \prod_{\nu =1}^{l_i} z \left( D_i + \nu\right)} \cdot \prod_{l_i < 0} \frac{\Gamma(D_i+1) \cdot \prod_{\nu = l_i +1}^{0} z \cdot \left( D_i + \nu \right)}{\Gamma(D_i +1)} \notag \\
&= e^{\delta} \cdot z^{-\rho} \cdot \prod_{i=1}^m \Gamma(D_i +1) \cdot \sum_{\underline{l} \in \dL} q^{\underline{l}} \prod_{l_i \geq 0} \frac{1}{\Gamma(D_i + l_i +1) \cdot z^{l_i}} \cdot \prod_{l_i < 0} \frac{z^{-l_i}}{\Gamma(D_i + l_i +1)}.\notag
\end{align}}
The identity $\prod_{i=1}^m\Gamma(D_i +1) = \Gamma(T \XSigA)$ yields
$$
\widetilde{I} = \Gamma(T \XSigA)\cdot e^\delta\cdot z^{-\rho}\cdot
\sum\limits_{\underline{l}\in\dL}
\frac{q^{\underline{l}}\; \cdot z^{-\overline{l}}}{\prod_{i=1}^m \Gamma(D_i+l_i+1)}
$$
For the second point, notice first that
$$
\widetilde{I} = \Gamma(T \XSigA)\cdot e^\delta\cdot z^{-\rho}\cdot
\sum\limits_{\underline{l}\in\dL\cap\textup{Eff}_{\XSigA}}
\frac{q^{\underline{l}}\; \cdot z^{-\overline{l}}}{\prod_{i=1}^m \Gamma(D_i+l_i+1)},
$$
where again $\textup{Eff}_{\XSigA}\subset \dL_\dR$ denotes the Mori cone of classes of effective curves in $\XSigA$. Indeed,
we will see that for any $\underline{l}^0$ outside $\Leff=\dL\cap\textup{Eff}_{\XSigA}$,
the term $\frac{q^{\underline{l}^0}\; \cdot z^{-\overline{l}^0}}{\prod_{i=1}^m \Gamma(D_i+l^0_i+1)}$ vanishes in $H^*(\XSigA,\dC)$.
Assume the contrary, and first notice that for $l^0_i < 0$ the factor $\frac{1}{\Gamma(D_i + l^0_i +1)}$ is divisible by $D_i$. For $\frac{q^{\underline{l}^0}\; \cdot z^{-\overline{l}^0}}{\prod_{i=1}^m \Gamma(D_i+l^0_i+1)}$ to be non-zero, there must be a maximal cone $\sigma^0$ containing the set of all $\underline{a}_i$ such that $l^0_i <0$,
as otherwise the term $\prod_{i:l^0_i<0} D_i$ which occurs as a factor in $\frac{q^{\underline{l}^0} \cdot z^{-\overline{l}^0}}{\prod_{i=1}^m \Gamma(D_i+l^0_i+1)}$
is zero in $H^*(\XSigA,\dC)$. We use again (see formula \eqref{eq:MoriCone}) that $\textup{Eff}_{\XSigA}= \sum_{\sigma \in \Sigma_A(n)} C_{\sigma}$,
where $C_{\sigma}$ is the cone generated by elements $\underline{l} =(l_1, \ldots, l_m)$ with $l_i \geq 0$ whenever $\dR_{\geq 0}\underline{a}_i$ is not a ray of $\sigma$. Thus $\underline{l}^0 \in C_{\sigma^0} \subset \textup{Eff}_{\XSigA}$, which shows the claim.
Now remember from the proof of theorem \ref{theo:LogExtQDMod} that for all $\underline{l}\in\Leff$ we have
$\overline{l}\geq 0$ as $\XSigA$ is weak Fano, hence,
$z^{-\overline{l}}$ has no poles at $z=\infty$.
Moreover, by the same argument $p_a(\underline{l})$ is non-negative for $\underline{l}\in\Leff$, which gives that
$q^{\underline{l}}$ has no poles along $\cup_{a=1}^r \{q_a=0\}$.
Hence we obtain $e^{-\delta/z}\cdot I, z^\rho\cdot e^{-\delta}\cdot\widetilde{I}\in H^*(\XSigA,\dC)[[q_1,\ldots,q_r,z^{-1}]]$.
\item
After what has been said before, it is evident that the $I$-function can be written as
$$
I= e^{\delta/z} \cdot \sum\limits_{\underline{l}\in\Leff} q^{\underline{l}}\;\cdot z^{-\overline{l}}\cdot
\prod\limits_{i=1}^m\frac{\prod_{\nu=-\infty}^0\left(\frac{[D_i]}{z}+\nu\right)}{
\prod_{\nu=-\infty}^{l_i}\left(\frac{[D_i]}{z}+\nu \right)}.
$$
Let us calculate the first terms in the $z^{-1}$-development of this expression: The constant term
can only get contributions from elements $\underline{l}\in\Leff$ with $\overline{l}=0$. The zero relation $\underline{l}=0$
gives the cohomology class $1$, on the other hand, for any $\underline{l}\neq 0$ with $\overline{l}=0$, there must be
at least one $i\in\{1,\ldots,m\}$ with $l_i<0$, and then constant coefficient in the product
$\frac{\prod_{\nu=-\infty}^0\left(\frac{[D_i]}{z}+\nu\right)}{
\prod_{\nu=-\infty}^{l_i}\left(\frac{[D_i]}{z}+\nu \right)}$ gets a factor $\nu=0$, i.e., is zero.
By a similar argument, the coefficient $\gamma$ of the $z^{-1}$-term cannot have a $H^0(\XSigA,\dC)$-component. One also sees immediately
that $\gamma$ has no components in $H^{>2}(\XSigA,\dC)$. Hence we are left to show that $\gamma (q_1,\ldots,q_r) = \delta + \gamma'(q_1,\ldots,q_r)$.
We have
$$
I = \left(1+\delta/z+o(z^{-1})\right) \cdot \sum\limits_{\underline{l}\in\Leff} q^{\underline{l}}\;\cdot z^{-\overline{l}}\cdot
\prod\limits_{i=1}^m\frac{\prod_{\nu=-\infty}^0\left(\frac{[D_i]}{z}+\nu\right)}{
\prod_{\nu=-\infty}^{l_i}\left(\frac{[D_i]}{z}+\nu \right)}.
$$
For the coefficient $\gamma$, we have a contribution from the $\delta/z$-term in the first factor, and if $\XSigA$ is Fano,
this is the only term as then $\overline{l}>0$ for all $\underline{l}\in\Leff\backslash\{0\}$. In the weak Fano case, any $\underline{l}\in\Leff\backslash\{0\}$ with
$\overline{l}=0$ give some extra contribution from the $[D_i]/z$-terms, but this part is multiplied by $q^{\underline{l}}\;$, i.e., a
univalued function in $q_1,\ldots,q_r$.

\item
As a first step, we show that there is a constant $L>0$ such that for any $x =(x_1, \ldots , x_m) \in \dC^m$, the expression
$$
\sum_{\underline{l} \in \Leff} \frac{q^{\underline{l}}\; \cdot z^{- \overline{l}}}{\prod_{i=1}^{m} \Gamma(x_i + l_i +1)} = \sum_{\underline{l} \in \Leff} z^{- \overline{l}} \frac{\prod_{a=1}^r q_a^{p_a(\underline{l})}}{\prod_{i=1}^{m} \Gamma(x_i + l_i +1)}
$$
is convergent on
$\{(z,q_1, \ldots ,q_r) \mid |z| \geq 1 , |q_a| \leq L\} \cap \dC_\tau^* \times S^0_2$. Using \cite[Lemma A.4]{BorHor}  we have
$$
\left| z^{- \overline{l}} \frac{\prod_{a=1}^r q_a^{p_a(\underline{l})}}{\prod_{i=1}^m \Gamma(x_i + l_i +1)}\right| \leq A(x) (4m)^{||\underline{l}||}\cdot e^{- \overline{l} \cdot \log|z| + \sum_{a=1}^r p_a(\underline{l}) \cdot \log|q_a|}
$$
Let $\epsilon > 0$, the series is absolutely and uniformly convergent if
\begin{equation}\label{Ifuncpropconv}
\|\underline{l} \| \cdot \log(4m) - \overline{l} \cdot \log|z| + \sum_{a=1}^r p_a(\underline{l}) \cdot \log|q_a| \leq - \epsilon ||\underline{l}||
\end{equation}
for all $\underline{l} \in \Leff$. This gives the condition
\[
\overline{l} \cdot \log|z| + \sum_{a=1}^r p_a(\underline{l}) \cdot (-\log|q_a|) \geq (\epsilon + \log(4m)) \cdot ||\underline{l}||
\]
Let $||M||$ be the norm of the matrix $(m_{ia})$. For $|z| \geq 1$ and $\underline{q} \in S_2^0$ we have
\begin{align}
&\overline{l} \cdot \log|z| + \sum_{a=1}^r p_a(\underline{l}) \cdot (-\log|q_a|) \geq \sum_{a=1}^r p_a(\underline{l}) \cdot (-\log|q_a|) \notag \\
\geq &\sum_{a=1}^r p_a(\underline{l}) \cdot \min_{a=1, \ldots ,r } (- \log |q_a|) \geq  \frac{1}{||M||} \cdot ||\underline{l}|| \cdot \min_{a=1, \ldots ,r } (- \log |q_a|) \notag
\end{align}
where we have used $\sum_{a=1}^r m_{ia} p_{a}(\underline{l}) = l_i$ and $p_a(\underline{l}) \geq 0$ for $\underline{l} \in \Leff$. Thus condition (\ref{Ifuncpropconv}) is satisfied for
\[
\max_{a=1, \ldots r} |q_a| \leq e^{-||M||( \epsilon + log(4m))} =:L
\]
This shows convergence of
$\sum_{\underline{l} \in \Leff} \frac{q^{\underline{l}}\; \cdot z^{- \overline{l}}}{\prod_{i=1}^{m} \Gamma(x_i + l_i +1)}$ on
$\widetilde{S}^*:=\{(z,q_1, \ldots ,q_r) \mid |z| \geq 1 , |q_a| \leq L\} \cap \dC_\tau^* \times S^0_2$. From the nilpotency of the operators $D_i \cup\in \mathit{End}(H^*(\XSigA,\dC)$
we see that
$$
\sum_{\underline{l} \in \Leff} \frac{q^{\underline{l}}\; \cdot z^{- \overline{l}}}{\prod_{i=1}^{m} \Gamma(D_i + l_i +1)}
\in H^*(\XSigA,\dC)\otimes \cO^{an}_{\widetilde{S}^*}.
$$

For the readers convenience, we recall next how to derive the identities
\begin{equation}\label{eq:OperatorsAnnihilateIFunc}
\begin{array}{rcccl}
\widetilde{\Box}_{\underline{l}} (\widetilde{I})  & = & 0 & \quad & \forall \underline{l} \in \dL\\ \\
\left(z\partial_z+\sum_{a=1}^r \rho(p^\vee_a) q_a\partial_{q_a}\right) (\widetilde{I}) & = & 0.
\end{array}
\end{equation}
Write $\widetilde{\Box}_{\underline{l}^0}:=\widetilde{\Box}^-_{\underline{l}^0} - \widetilde{\Box}^+_{\underline{l}^0}$ where
$$
\begin{array}{rcl}
\widetilde{\Box}^-_{\underline{l}^0}  & := &\prod\limits_{a:p_a(\underline{l}^0)>0} q_a^{p_a(\underline{l}^0)} \prod\limits_{i:l_i^0<0}
\prod\limits_{\nu=0}^{-l^0_i-1} (\sum_{a=1}^r m_{ia} zq_a\partial_{q_a} -\nu z)  \\ \\
\widetilde{\Box}^+_{\underline{l}^0} &:= & \prod\limits_{a:p_a(\underline{l}^0)<0} q_a^{-p_a(\underline{l}^0)} \prod\limits_{i:l^0_i>0}
\prod\limits_{\nu=0}^{l^0_i-1} (\sum_{a=1}^r m_{ia} zq_a\partial_{q_a} -\nu z)
\end{array}
$$
Using the fact that $z q_a \partial_{q_a} \widetilde{I} = z (p_a + p_a(\underline{l})) \cdot \widetilde{I}$ we get
\begin{align}
\widetilde{\Box}^-_{\underline{l}^0} \widetilde{I} &= \Gamma(T \XSigA)\cdot e^\delta \cdot z^{-\rho}\cdot
\sum\limits_{\underline{l}\in\dL}\left( \prod_{a:p_a(\underline{l}^0)>0} q_a^{p_a(\underline{l}^0)} \prod_{i:l_i^0 < 0} z^{-l^0_i}
\prod\limits_{\nu=l_i^0+1}^{0}(D_i + l_i+\nu) \right)
\frac{\prod_{a=1}^r q_a^{p_a(\underline{l})} \cdot z^{-\overline{l}}}{\prod_{i} \Gamma(D_i+ l_i +1)}  \notag \\
&= \Gamma(T \XSigA)\cdot e^\delta \cdot z^{-\rho}\cdot
\sum\limits_{\underline{l}\in\dL}
\frac{\prod_{a : p_a(\underline{l})>0} q_a^{p_a(\underline{l} + \underline{l}^0)} \cdot \prod_{a : p_a(\underline{l})<0} q_a^{p_a(\underline{l})} \cdot z^{-\overline{l}- \sum_{i:l_i^0 <0} l^0_i}}{\prod_{i : l_i^0 <0} \Gamma(D_i+ l_i + l_i^0 +1) \cdot \prod_{i: l_i \geq 0} \Gamma(D_i +l_i +1)} \notag \\
&= \Gamma(T \XSigA)\cdot e^\delta \cdot z^{-\rho}\cdot
\sum\limits_{\underline{l}\in\dL}
\frac{\prod_{a : p_a(\underline{l})>0} q_a^{p_a(\underline{l})} \cdot \prod_{a : p_a(\underline{l})<0} q_a^{p_a(\underline{l}- \underline{l}^0)} \cdot z^{-\overline{l}+ \sum_{i:l_i^0 >0} l^0_i}}{\prod_{i : l_i^0 <0} \Gamma(D_i+ l_i +1) \cdot \prod_{i: l_i \geq 0} \Gamma(D_i +l_i -l_i^0 +1)} \notag \\
&= \Gamma(T \XSigA)\cdot e^\delta \cdot z^{-\rho}\cdot
\sum\limits_{\underline{l}\in\dL}\left( \prod_{a:p_a(\underline{l}^0)<0} q_a^{p_a(\underline{l}^0)} \prod_{i:l_i^0 > 0} z^{l^0_i} \prod\limits_{\nu=1-l_i^0}^{0}(D_i + l_i+\nu) \right)
\frac{\prod_{a=1}^r q_a^{p_a(\underline{l})} \cdot z^{-\overline{l}}}{\prod_{i} \Gamma(D_i+ l_i +1)}  \notag \\
&= \widetilde{\Box}^+_{\underline{l}^0} \widetilde{I}
\end{align}
which shows $\widetilde{\Box}_{\underline{l}} (\widetilde{I})  = 0 $. The second one of the equations \eqref{eq:OperatorsAnnihilateIFunc} follows from
\[
\left( z \partial_z +\sum_{a=1}^r \rho(p_a^\vee) q_a \partial_a \right) \widetilde{I} = \left( (-\rho - \overline{l}) + \sum_{a=1}^r \rho(p_a^\vee) (p_a + p_a(\underline{l})) \right) \widetilde{I} = 0
\]
Now we conclude by a classical argument from the theory of ordinary differential equations
(see, e.g., \cite[Theorem 3.1]{CoddLev}): Fix $\underline{q}^0\in S_2^0$ with $|q^0_a|< L$,
then $\widetilde{I}(z^{-1},\underline{q}^0)$ satisfies a system of differential equations in $z^{-1}$ with a regular singularity at $z^{-1}=0$.
Hence $\widetilde{I}(z^{-1},\underline{q}^0)$ is a multivalued analytic function on all of $\dC^*_\tau\times \{\underline{q}_0\}$, that is,
$\widetilde{I}$ is (multivalued) analytic in $\dC^*_\tau \times S^*$, with $S=\{\underline{q}\in\dC^r\,|\,|q_a|< L\}$, this implies the statement
on $e^{-\delta/z}\cdot I$ and $z^\rho\cdot e^{-\delta}\cdot \widetilde{I}$ and obviously also the convergence of the
coordinate change $\kappa$.
\item
This is a direct consequence of the equations \eqref{eq:OperatorsAnnihilateIFunc}.
\item
We follow the argument in \cite[proposition 2.19]{BorHor}.
Let $h\in(H^*(\XSigA,\dC))^\vee\backslash\{0\}$ be given,  and let $c=p_1^{k_1}\cdot \ldots \cdot p_r^{k_r}
\in H^*(\XSigA,\dC)$ be a monomial cohomology class of maximal degree
such that $h(c)\neq 0$. Consider $\widetilde{I}$ as a multivalued section of the trivial bundle
$H^*(\XSigA,\dC)\times (\dC^*_\tau\times S^*) \twoheadrightarrow \dC^*_\tau\times S^*$, then as $e^{-\delta}\cdot \widetilde{I}$ is univalued, the monodromy
operator $M_a$ corresponding to a loop around $q_a=0$ sends $\widetilde{I}$ to  $e^{2\pi i p_a}\cdot \widetilde{I}$.
Hence we have
$$
\left(\log(M_1)^{k_1}\circ \ldots\circ \log(M_r)^{k_r}\right) h(\widetilde{I}) = h((2\pi i)^r \cdot p_1^{k_1}\cdot\ldots\cdot p_r^{kr} \cdot \widetilde{I}),
$$
and it suffices to show that the right hand side of this equation is not the zero function as then
$h(\widetilde{I})$ itself cannot be identically zero. We have
$$
h(p_1^{k_1}\cdot\ldots\cdot p_r^{kr} \cdot \widetilde{I}) =
\sum\limits_{l\in\Leff} q^{\underline{l}}\;\cdot z^{-\overline{l}} \cdot
h\left(p_1^{k_1}\cdot\ldots\cdot p_r^{k_r} \cdot \frac{\Gamma(T \XSigA)\cdot e^\delta}{\prod_{i=1}^m\Gamma(D_i+l_i+1)}\cdot z^{-\rho}\right)
$$
The contribution of $\underline{l}=(0,\ldots,0)\in\Leff$ is
$$
\begin{array}{rcl}
h\left(p_1^{k_1}\cdot\ldots\cdot p_r^{k_r} \cdot \dfrac{\Gamma(T \XSigA)\cdot}{\prod_{i=1}^m\Gamma(D_i+1)}\cdot e^\delta\cdot z^{-\rho}\right)
& = & h\left(p_1^{k_1}\cdot\ldots\cdot p_r^{k_r} \cdot e^\delta\cdot z^{-\rho}\right) \\ \\
& = & h\left(p_1^{k_1}\cdot\ldots\cdot p_r^{k_r} \cdot (1+\widetilde{c})\right).
\end{array}
$$
where $\widetilde{c} \in H^{>0}(\XSigA,\dC)[\log(z),\log(q_1),\ldots,\log(q_r)]$. As $h$ is zero
on any cohomology class of degree strictly bigger than $p_1^{k_1}\cdot\ldots\cdot p_r^{k_r}$, we get
$h\left(p_1^{k_1}\cdot\ldots\cdot p_r^{k_r} \cdot (1+\widetilde{c})\right)\neq 0$. On the other hand,
this term cannot be killed by a contribution from any $\underline{l}\in\Leff\backslash\{0\}$, as for such
an $\underline{l}$, $e^{\delta(\underline{l})}$ will have positive degree.
\end{enumerate}
\end{proof}
As an easy consequence, we obtain the following interpretation of the $I$- resp. the $\widetilde{I}$-function.
\begin{corollary}
For any homogeneous basis $T_0,T_1,\ldots, T_s$ of $H^*(\XSigA,\dC)$,
write $\widetilde{I} = \sum_{t=0}^s \widetilde{I}_t \cdot T_t$, so that
$\widetilde{I}_t\in \dH^0\left((\widetilde{\dC^*_\tau\times S^*})^{an},\pi^*{\cS\!}ol^\bullet(\qM)\right)$ by
proposition \ref{prop:IFunction}, 3. Moreover, $(\widetilde{I}_0,\ldots,\widetilde{I}_s)$
is a basis of $\dH^0\left((\widetilde{\dC^*_\tau\times S^*})^{an},\pi^*{\cS\!}ol^\bullet(\qM)\right)$
by proposition \ref{prop:IFunction}, 4.
Using the natural duality
$$
\begin{array}{rcl}
\dH^0\left((\widetilde{\dC^*_\tau\times S^*})^{an},\pi^*{\cS\!}ol^\bullet(\qM)\right)
& \stackrel{!}{=} &
\left(
\dH^0\left((\widetilde{\dC^*_\tau\times S^*})^{an},\pi^*\DR^\bullet(\qM)\right)
\right)^\vee \\ \\
& = &
H^0\left((\widetilde{\dC^*_\tau\times S^*})^{an},
\pi^*{\cH\!}om_{\cD_{\dC^*_\tau\times S^*}}(\cO_{\dC_\tau^*\times S^*},\qM)\right)^\vee,
\end{array}
$$
let $(f_0,\ldots,f_s) \in \left(
\dH^0\left((\widetilde{\dC^*_\tau\times S^*})^{an},\pi^*\DR^\bullet(\qM)\right)
\right)^{s+1}$ be the dual basis, then we have
$$
\id = \sum_{t=0}^s f_t \circ \widetilde{I}_t \in H^0\left((\dC^*_\tau\times S^*)^{an},{\cE\!}nd_{\cD_{(\dC^*_\tau\times S^*)^{an}}}(\qM)\right).
$$
In particular, seeing $\widetilde{I}_t$ (or, more precisely $\widetilde{I}_t(1)$) as a multivalued function in $\cO_{\dC^*_\tau\times S^*}$, we obtain
a representation
\begin{equation}\label{eq:RepresentationOfOne-BSide}
1 = \sum_{t=0}^s \widetilde{I}_t(z^{-1},q_1,\ldots,q_r) \cdot f_t
\end{equation}
of the element $1\in \qM$, where $f_t$ are multivalued sections of the local system $((\qM)^{an}_{|\dC^*_\tau\times S^*})^\nabla$.
\end{corollary}

\subsection{Logarithmic Frobenius structures}
\label{subsec:LogFrob}

We derive in this subsection the existence of a Frobenius manifold
with logarithmic poles associated to the Landau-Ginzburg model of
$\XSigA$. This extends, for the given class of functions, the construction from \cite{DS},
in the sense that we obtain a family of germs of Frobenius manifolds
along the space $U^0$ from the last subsection, with a logarithmic
degeneration behavior at the divisor $Z$.
For the readers convenience, we first recall briefly
the notion of a Frobenius structure with logarithmic poles,
and one of the main result from \cite{Reich1}, which produces
such structures starting from a set of initial data with specific
properties. In contrast to the earlier parts of the paper,
all objects in this subsection are analytic, unless otherwise stated.

\begin{deflemma}\label{deflem:logFrob-logtrTLEP}
Let $M$ be a complex manifold of dimension bigger or equal to one, and $Z\subset M$ be a simple normal crossing divisor.
\begin{enumerate}
\item
Suppose that $(M\backslash Z, \circ, g, e,E)$ is a Frobenius manifold. Then we say that it
has a logarithmic pole along $Z$ (or that $(M,Z,\circ, g, e,E)$ is a logarithmic Frobenius manifold for short) if
$\circ\in\Omega^1_M(\log\,Z)^{\otimes 2}\otimes \Theta_M(\log\,Z)$, $g\in\Omega^1_M(\log\,Z)^{\otimes 2}$,
$E,e\in\Theta(\log\,Z)$ and if $g$ is non-degenerate on $\Theta_M(\log\,Z)$.
\item
A $\trTLEPlog$-structure on $M$ is a holomorphic vector bundle $\cH\rightarrow \dP_z^1\times M$ such that $p^*p_* \cH = \cH$ (where $p:\dP_z^1\times M\twoheadrightarrow M$ is the projection) which is equipped with an integrable connection $\nabla$ with a pole of type $1$ along $\{0\}\times M$ and
a logarithmic pole along $(\dP^1_z \times Z) \cup (\{\infty\}\times M)$ and a flat, $(-1)^n$-symmetric, non-degenerate pairing
$P:\cH\otimes \iota^*\cH\rightarrow \cO_{\dP^1_z\times M}(-n,n)$.
\item
Any logarithmic Frobenius manifold gives rise to a $\trTLEPlog$-structures on $M$, basically by setting $\cH:=p^*\Theta(\log\,Z)$,
$\nabla:=\nabla^{\mathit{LC}}-\frac1{z}\circ + \left(\frac{\cU}{z}-\cV\right)\frac{dz}{z}$, where
$\nabla^{\mathit{LC}}$ is the Levi-Civita connection of $g$ on $\Theta(\log\,Z)$, $\cU:=E\circ \in{\cE\!}nd(\Theta(\log\,Z))$
and $\cV:=\nabla^{\mathit{LC}}_\bullet E-\textup{Id}\in{\cE\!}nd(\Theta(\log\,Z))$ (see \cite[proposition 1.7 and proposition 1.10]{Reich1} for more details).
\end{enumerate}
\end{deflemma}
Under certain conditions, a given $\trTLEPlog$-structure can be unfolded to a logarithmic Frobenius manifold.
This is summarized in the following theorem which we extract from \cite[theorem 1.12]{Reich1}, notice that
a non-logarithmic version of it was shown in \cite{HM}, and goes back to earlier work of Dubrovin and Malgrange (see
the references in \cite{HM}).
\begin{theorem}\label{theo:logUnfoldingTheorem}
Let $(N,0)$ be a germ of a complex manifold and $(Z,0)\subset (N,0)$ a normal crossing divisor.
Let $(\cH,0)$ be a germ of a $\trTLEPlog$-structure on $N$. Suppose moreover that
there is a global section $\xi\in H^0(\dP^1\times N, \cH)$ whose restriction to $\{\infty\}\times N$ is horizontal
for the residue connection $\nabla^{\mathit{res},\tau}:\cH/\tau \cH\rightarrow\cH/\tau \cH\otimes\Omega^1_{\{\infty\}\times N}(\log\,(\{\infty\}\times Z))$
and which satisfies the following three conditions
\begin{enumerate}
\item
The map from $\Theta(\log\,Z)_{|\underline{0}} \rightarrow p_*\cH_{|\underline{0}}$ induced by the Higgs field $[z\nabla_\bullet](\xi):\Theta(\log\,Z)\rightarrow p_*\cH$ is injective \textup{(injectivity condition (IC))},
\item
The vector space $p_*\cH_{|\underline{0}}$ is generated by $\xi_{|(0,\underline{0})}$ and its images under
iteration of the elements of $\mathit{End}(p_*\cH_{|\underline{0}})$ induced by $\cU$ and by $[z\nabla_X]\in$ for any $X\in\Theta(\log\,Z)$
\textup{(generation condition (GC))},
\item
$\xi$ is an eigenvector for the residue endomorphism $\cV \in {\cE\!nd}_{\cO_{\{\infty\}\times N}}(\cH/z^{-1}\cH)$ \textup{(eigenvector condition (EC))}.
\end{enumerate}
Then there exists a germ of a logarithmic Frobenius manifold $(M,\widetilde{Z})$, which is unique
up to canonical isomorphism, a unique embedding $i:N\hookrightarrow M$ with $i(M) \cap \widetilde{Z} = i(Z)$ and a unique isomorphism
$\cH \rightarrow (\id_{\dP^1_z}\times i)^* p^* \Theta_M(\log\,\widetilde{Z})$ of $\trTLEPlog$-structures.
\end{theorem}
Using proposition \ref{prop:LogTrTLEP}, we show now how to associate a logarithmic Frobenius manifold
to the Landau-Ginzburg model $(W,q)$ of the toric manifold $\XSigA$.
\begin{theorem}
\label{theo:logFrob-BSide}
\begin{enumerate}
\item
Let $\XSigA$ be a smooth toric weak Fano manifold, defined by a fan $\Sigma_A$. Let
$(W,q):S_1 \rightarrow \dC_t \times S_2$ be the Landau-Ginzburg model of $\XSigA$ and let $q_1,\ldots,q_r$ be the coordinates
on $S_2$ defined by the choice of a nef basis $p_1,\ldots,p_r$ of $\dL^\vee$.
Consider the tuple $(\qMBLhat, \nabla, P)$ associated to $(W,q)$ by proposition \ref{prop:LogTrTLEP}.
Then the corresponding analytic object
$(\qMBLhat, \nabla, P)^{an}$ is a $\trTLEPlog$-structure on $U^{0,an}$.
\item
There is a canonical Frobenius structure on $(U^{0,an} \times \dC^{\mu-r},0)$ with a logarithmic
pole along $(Z \times \dC^{\mu-r},0)$, where, as before, $Z=\bigcup_{a=1}^r \{q_a=0\}\subset U^{0,an}\subset \dC^r$.
\end{enumerate}
\end{theorem}
\begin{proof}
\begin{enumerate}
\item
This follows directly from the properties of $\qMBLhat$, $\nabla$ and $P$ as described
in proposition \ref{prop:LogTrTLEP}.
\item
We apply theorem \ref{theo:logUnfoldingTheorem} to the germ $(N,0):=(U^{0,an},0)$ and
the germ of the $\trTLEPlog$-structure $(\qMBLhat, \nabla, P)^{an}$. Define the section $\xi$ to be the class
of $1$ in $F\subset H^0(\dP^1_z\times U^0, \qMBLhat)$, recall that $F\cong H^0(\dP^1\times\{\underline{0}\},(\qMhat)_{|\dP^1_z\times\{\underline{0}\}})$
was defined as the subspace of $E\cong (\qMBLlog)_{|\dC_z\times\{\underline{0}\}}$ generated by monomials in $(zq_a\partial_{q_a})_{a=1,\ldots,r}$. The $\nabla^{\mathit{res},\tau}$-flatness of $\xi$ follows from proposition
\ref{prop:LogTrTLEP}, 6. Conditions (IC) and (GC) are a consequence of the identification
of $(\qMBLlog)_{|(0,\underline{0})}$ with $(H^*(\XSigA,\dC),\cup)$ (lemma \ref{lem:ModuleE}, 1.) and the fact that
the latter algebra is ``$H^2$-generated'', i.e., from the description given by formula \eqref{eq:CohomToric}.
More precisely, the action of the logarithmic Higgs fields $[zq_a\partial_{q_a}]$ on $H^0(\dP^1_z,\widehat{E})\cong F\cong(\qMBLlog)_{(0,\underline{0})}$
correspond, under the isomorphism $\alpha$ from lemma \ref{lem:ModuleE} exactly to the multiplication with the divisors classes $D_a\in H^2(\XSigA,\dC)$ on $H^*(\XSigA,\dC)$, and
$H^2$-generation implies that the images under iteration of these multiplications generate the whole vector space $(\qMBLlog)_{|(0,\underline{0})}$.
Finally, condition (EC) follows from proposition \ref{prop:ExtensionInftyAtZero}, 2. Hence the conditions of theorem \ref{theo:logUnfoldingTheorem} are satisfied
and yield the existence of a Frobenius structure on a germ $(N\times \dC^{\mu-r},0)$, which is canonical
in the sense that it does not depend on any further choice, and which is universal for chosen section $\xi$ by
the universality property of theorem \ref{theo:logUnfoldingTheorem}.
\end{enumerate}
\end{proof}
\textbf{Remark: } It follows from conditions (GC) and (EC) that $\xi$
is a primitive and homogeneous section in the sense of \cite{DS} (this notion
goes back to the theory of ``primitive forms'' of K.~Saito). In particular,
for a representative $U^{0,an}$ of the germ $(U^{0,an},0)$ and any point
$\underline{q}\in    U^{0,an}\backslash Z$, the Frobenius structure from theorem
\ref{theo:logUnfoldingTheorem}, 2., is one of those constructed in loc.cit.
It is a natural to ask the following
\begin{question}\label{quest:CanonicalFrob}
Is the (restriction of the) Frobenius structure from above to
a small neighborhood of $\underline{q} \in U^{0,an}\backslash Z$ the
canonical Frobenius structure of the map $\widetilde{W}(-,\underline{q}):S_0\rightarrow \dC_t$ from
\cite{DS} (see also \cite{DouaiMZ})?
\end{question}
Notice that for $\XSigA=\dP^n$, it follows from the computations done in \cite{DS2} (which concern
the more general case of weighted projective spaces), that this question can be answered in the affirmative.

\section{The quantum $\cD$-module and the mirror correspondence}
\label{sec:AModel}

We start this section by recalling for the readers convenience
some well-known constructions from quantum cohomology of smooth
projective varieties, mainly in order to fix the notations.
In particular, we explain the so-called quantum $\cD$-module (resp. the Givental connection) and
the $J$-function. We next show that the quantum $\cD$-module can be identified with the object
$\qMBLhat$ constructed in the last section.
This identification uses the famous $I\!=\!J$-theorem of Givental and can be seen as the essence of
the mirror correspondence for smooth toric weak Fano varieties. As
a consequence, using the results of subsection \ref{subsec:LogFrob}, we obtain
a mirror correspondence as an isomorphism of logarithmic Frobenius manifolds.

\subsection{Quantum cohomology and Givental connection}
\label{subsec:GivConnection}

We review very briefly some well known constructions from
quantum cohomology of smooth projective complex varieties
and explain the the so-called quantum
$\cD$-module, also called Givental connection.
\begin{deflemma}
Let $X$ be smooth and projective over $\dC$ with $\dim_\dC(X)=n$. Choose once and
for all a homogeneous basis $T_0,T_1,\ldots,T_r,T_{r+1},\ldots,T_s$
of $H^{2*}(X,\dC)$, where $T_0=1\in H^0(X,\dC)$, $T_1,\ldots,T_r$ are nef classes in $H^2(X,\dZ)$
(here and in what follows, we consider without mentioning only the torsion free parts
of the integer cohomology groups)
and $T_i\in H^{2k}(X,\dC)$ with $k>2$ for all $i>r$. If $X=\XSigA$ is
toric and weak Fano, then we suppose moreover that $T_i=p_i$, i.e,
that the basis $T_0,\ldots,T_s$ extends the basis of $\dL^\vee\cong H^2(\XSigA,\dC)$ chosen
at the beginning of section \ref{subsec:QDModDown}. We write $t_0,\ldots,t_s$ for the
coordinates induced on the space $H^{2*}(X,\dC)$. We denote by $(-,-)$ the Poincar\'e pairing
on $H^{2*}(X,\dC)$ and by $(T^k)_{k=0,\ldots,s}$ the dual basis with respect to $(-,-)$.
\begin{enumerate}
\item
For any effective class $\beta\in H_2(X,\dZ)/\mathit{Tors}$ denote by
$\overline{\cM}_{0,n,\beta}(X)$ the Deligne-Mumford stack of stable
maps $f:C\rightarrow X$ from rational nodal pointed curves $C$ to $X$ such that
$f_*([C])=[\beta]$. For any $i=1,\ldots,n$, let
$\omega_\pi$ be the relative dualizing sheaf of the
``forgetful'' morphism
$\pi:\overline{\cM}_{0,n+1,\beta}(X) \rightarrow \overline{\cM}_{0,n,\beta}(X)$
(i.e., the morphism forgetting the $i$-th point and stabilizing if necessary) which represents the universal family.
Define a Cartier divisor $L_i:=x_i^*(\omega_\pi)$
on $\overline{\cM}_{0,n,\beta}(X)$, where $x_i:\overline{\cM}_{0,n,\beta}(X)\rightarrow \overline{\cM}_{0,n+1,\beta}(X)$
is the $i$-th marked point, and put $\psi_i=c_1(L_i)$.
\item
For any tuple $\alpha_1,\ldots,\alpha_n\in H^{2*}(X,\dC)$, let
$$
\langle \psi_1^{i_1}\alpha_1,\ldots,\psi_n^{i_n}\alpha_n\rangle_{0,n,\beta}:=\int_{[\overline{\cM}_{0,n,\beta}(X)]^{virt}}
\psi_1^{i_1}\mathit{ev}_1(\alpha_1)\cup\ldots\cup\psi_n^{i_n}\mathit{ev}_n(\alpha_n)
$$
and put $\langle \alpha_1,\ldots,\alpha_n \rangle_{0,n,\beta}:=
\langle \psi_1^0 \alpha_1,\ldots,\psi_n^0 \alpha_n\rangle_{0,n,\beta}$. Here $ev_i:\overline{\cM}_{0,n,\beta}(X)\rightarrow
X$ is the $i$-th evaluation morphism $ev_i([C,f,(x_1,\ldots,x_n)]):=f(x_i)$ and
$[\overline{\cM}_{0,n,\beta}(X)]^{virt}$ is the so-called virtual fundamental class of $\overline{\cM}_{0,n,\beta}(X)$,
which has dimension $\dim_\dC(X)+\int_\beta c_1(X)+n-3$. $\langle \alpha_1,\ldots,\alpha_n\rangle_{0,n,\beta}$ is called
a Gromov-Witten invariant and
$\langle \psi_1^{i_1}\alpha_1,\ldots,\psi_n^{i_n}\alpha_n\rangle_{0,n,\beta}$ is a Gromov-Witten invariant with gravitational descendent.
\item
Let $\alpha,\gamma,\tau\in H^{2*}(X,\dC)$ be given, write $\tau=\tau'+\delta$ where
$\delta\in H^2(X,\dC)$ and $\tau'\in H^0(X,\dC)\oplus H^{>2}(X,\dC)$. Define the \emph{big quantum product} to be
\begin{align}
\alpha \circ_\tau \gamma :&= \sum_{\beta\in\textup{Eff}_X} \sum_{n,k\geq 0} \frac1{n!} \langle \alpha,\gamma,\underbrace{\tau,\ldots,\tau}_{n-\textup{times}},T_k \rangle_{0,n+3,\beta}T^k Q^\beta
\notag\\
&=\sum_{\beta\in\textup{Eff}_X} \sum_{n,k\geq 0} \frac{e^{\delta(\beta)}}{n!} \langle \alpha,\gamma,\underbrace{\tau',\ldots,\tau'}_{n-\textup{times}},T_k \rangle_{0,n+3,\beta}T^k Q^\beta
\in H^{2*}(X,\dC)\otimes \dC[[\underline{t}]][[\textup{Eff}_X]]
\end{align}
where $\textup{Eff}_X$ is the semigroup of effective classes in $H_2(X,\dZ)$, i.e., the
intersection of $H_2(X,\dZ)$ with the Mori cone in $H_2(X,\dR)$. Notice that in order to obtain
the last equality, we have used the divisor axiom for Gromov-Witten invariants, see, e.g., \cite[section 7.3.1]{CK}.

The Novikov ring $\dC[[\textup{Eff}_X]]$ was introduced to split the contribution
of the different $\beta \in \textup{Eff}_X$, as otherwise the formula above would not be a formal power series. However, if one knows the convergence of this power series, one can set $Q=1$.\item
Suppose that as before $\alpha,\gamma\in H^{2*}(X,\dC)$ and that $\delta\in H^2(X,\dC)$. Define the \emph{small
quantum product} by
$$
\alpha \star_\delta \gamma :=\sum_{k=0}^s \sum_{\beta\in\textup{Eff}_X}e^{\delta(\beta)}\langle \alpha,\gamma,T_k\rangle_{0,3,\beta} T^k Q^\beta \in H^{2*}(X,\dC) \otimes \cO^{an}_{H^2(X,\dC)}[[\textup{Eff}_X]].
$$
\end{enumerate}

\end{deflemma}

As we have seen, the quantum product exists
as defined only formally near the origin in $H^{2*}(X,\dC)$. However, we will need
to consider the asymptotic behavior of the quantum product in another limit direction inside this cohomology space. For that purpose
we will use the following
\begin{theorem}[{\cite[theorem 1.3]{IrConv}}]
The quantum product for a projective smooth toric variety is convergent on a simply connected neighborhood
W of
$$
\left\{
\tau = \tau'+\delta \in H^{2*}(X,\dC)\,|\, \Re(\delta(\beta)) < -M \;\;\forall \beta\in\textup{Eff}_X\backslash\{0\},\;\|\tau'\| < e^{-M}
\right\}
$$
for some $M\gg 0$, here $\|\cdot \|$ can be taken to be the standard hermitian norm on $H^{2*}(X,\dC)$ induced by the basis $T_0,\ldots,T_s$.
\end{theorem}

If $\alpha$ and $\gamma$ are seen as sections
of the tangent bundle of the cohomology space, we also write $\alpha\circ\gamma$ for the quantum product, which is also a section of $T H^{2*}(X,\dC)$.

The next step is to define the Givental connection, also known as the quantum $\cD$-module. For a smooth toric weak Fano manifold, this is
the object that we will compare to the various hypergeometric differential systems constructed
in the last section from the Landau-Ginzburg model of this variety.
\begin{deflemma}
\begin{enumerate}
\item
Write $p:\dP^1_z\times W \twoheadrightarrow W$ for the projection, and
let $\cF^{\mathit{big}}:=p^* T W$ be the pull-back of the tangent bundle of $W$. Define
a connection with a logarithmic pole along $\{\infty\}\times W$ and with
pole of type $1$ along $\{0\}\times W$ on $\cF^{\mathit{big}}$ by putting
for any $s\in H^0(\dP^1_z\times W,\cF^{\mathit{big}})$
\begin{equation}\label{eq:QDmod}
\begin{array}{rcl}
\nabla^{\mathit{Giv}}_{\partial_{t_k}} s& := & \nabla^{\mathit{res},z^{-1}}_{\partial_{t_k}}(s) - \frac1z \cdot T_k \circ T_l\\ \\
\nabla^{\mathit{Giv}}_{\partial_z} s & := & \frac1z\left(\frac{E\circ s }{z}+\mu(s)\right)
\end{array}
\end{equation}
where $\mu\in\mathit{End}_\dC(H^{2*}(X,\dC))$ is the grading operator already used in definition \ref{def:IFunction},
$$
E:=\sum_{i=0}^s \left(1-\frac{\deg(T_i)}{2}\right) + \sum_{a=1}^r k_a\partial_{T_a}
$$
is the so-called Euler field which is defined by $\sum_{a=1}^r k_a T_a = c_1(X)$
 and where $\nabla^{\mathit{res},z^{-1}}$ is the connection on $T W$ defined by the affine structure on $H^{2*}(X,\dC)$.
Notice that by its very definition, the residue connection of $\nabla^{Giv}$ along $z^{-1}=0$ is $\nabla^{\mathit{res},z^{-1}}$, whence its name. We have that $(\nabla^{\mathit{Giv}})^2=0$, and
this integrability condition encodes many of the properties of the quantum product (most
notably its associativity, which is expressed by a system of partial differential equations, known as Witten-Dijkgraaf-Verlinde-Verlinde equations).
We sometimes use the dual Givental connection, which is defined by $\hat{\nabla}^{\mathit{Giv}} := \iota^* \nabla^{\mathit{Giv}}$,
recall that $\iota(z,\underline{t})=(-z,\underline{t})$.
\item
Define the pairing
\begin{equation}\label{eq:PairingQCohom}
\begin{array}{rcl}
P:\cF^{\mathit{big}}\otimes \iota^*\cF^{\mathit{big}} & \longrightarrow & \cO_{\dP^1_z\times W} (-n,n) \\ \\
(a,b) & \longmapsto & z^n(a(z), b(-z))
\end{array}
\end{equation}
\item
The tuple $(\cF^{\mathit{big}},\nabla^{\mathit{Giv}},P)$ is a $\trTLEP$-structure on $W$ in the sense of \cite[definition 4.1]{HM} (i.e.,
the non-logarithmic version of definition-lemma \ref{deflem:logFrob-logtrTLEP}, 2.). We call it
the quantum $\cD$-module or Givental connection of $H^{2*}(X,\dC)$.
\item
Write $W':=\{\tau\in W\,|\, \tau'=0\}$ and let $\cF:=p^* (T H^{2*}(X,\dC)_{|W'})$. We equip $\cF$ with a connection and a pairing defined by formulas \eqref{eq:QDmod} and \eqref{eq:PairingQCohom}. Then
$(\cF,\nabla^{\mathit{Giv}},P)$ is a $\trTLEP$-structure on $W' \subset H^2(X,\dC)$, which we call the small quantum $\cD$-module. We have
$(\cF,\nabla^{\mathit{Giv}},P) = (\cF^{\mathit{big}},\nabla^{\mathit{Giv}},P)_{|\dP^1_z\times W'}$.
\end{enumerate}
\end{deflemma}
Next we show that the small quantum $\cD$-module can be considered
in a natural way as a bundle over the partial compactification of the
K\"ahler moduli space that we already encountered in the last section.
\begin{lemma}
\begin{enumerate}
\item
Consider the natural action of $ 2\pi i H^2(X,\dZ)$ on $H^{2*}(X,\dC)$ by translation. Then the set $W$ is
invariant under this action. Write $V_0$ for the quotient space, and $\pi:W\twoheadrightarrow V_0$
for the projection map. Then there is a $\trTLEP$-structure $(\cG^{\mathit{big}},\nabla^{\mathit{Giv}},P)$ on $V_0$ such that
$\pi^*(\cG^{\mathit{big}},\nabla^{\mathit{Giv}},P)=(\cF^{\mathit{big}},\nabla^{\mathit{Giv}},P)$. $(\cG^{\mathit{big}},\nabla^{\mathit{Giv}},P)$ is also called quantum $\cD$-module
of $X$.
\item
The algebraic quotient of $H^2(X,\dC)$ by $2 \pi i H^2(X,\dZ)$ is the torus $\Spec\dC[H^2(X,\dZ)]$,
which we call $S_2$ to be consistent with the notation of the previous section in case that
$X$ is toric weak Fano. Then the small quantum $\cD$-module descends to  $V'_0= S_2^{an} \cap V_0$, i.e,
there is a vector bundle $\cG$ on $\dP^1_z\times V'_0$, a connection $\nabla^{\mathit{Giv}}$ and a pairing $P$ such that $(\cG,\nabla^{\mathit{Giv}},P)$ is a $\trTLEP$-structure on $V'_0$ and such that $\pi^*(\cG,\nabla^{\mathit{Giv}},P)=(\cF,\nabla^{\mathit{Giv}},P)$, where $\pi: W'\twoheadrightarrow V'_0$ is again the projection map. We also call $(\cG,\nabla^{\mathit{Giv}},P)$ the small quantum $\cD$-module.
Obviously, we have again that
$(\cG,\nabla^{\mathit{Giv}},P) = (\cG^{\mathit{big}},\nabla^{\mathit{Giv}},P)_{|\dP^1_z\times V'_0}$.

If $X$ is Fano, then $(\cG,\nabla^{\mathit{Giv}},P)$ has an algebraic structure, i.e., it is defined as an algebraic
bundle over $\dP^1_z\times S_2$.
\end{enumerate}
\end{lemma}
\begin{proof}
The first statement and the first part of the second one are immediate consequences of the divisor axiom already mentioned above.
If $X$ is Fano, then as $\int_\beta c_1(X)>0$ for all $\beta\in\textup{Eff}_X$, for fixed $n$ only finitely many Gromov-Witten invariants
can be non-zero, this implies the algebraicity of $\cG$.
\end{proof}

\begin{corollary}\label{cor:logFrob-ASide}
Using the choice of the nef basis $T_1,\ldots,T_r$ of $H^2(X,\dZ)$ (consisting of the classes $p_1,\ldots,p_r \in \dL^\vee$
if $X=\XSigA$ is toric weak Fano), we obtain an embedding $H^2(X,\dC)/2 \pi i H^2(X,\dZ)\hookrightarrow \dC^r$, with complement
a normal crossing divisor $Z=\bigcup_{a=1}^r \{q_a=0\}$, if $q_a=e^{t_a}$ for $a=1,\ldots,r$. Denote by $V'$ the closure of the image of $V'_0$ under this embedding. Then
there is an extension $(\overline{\cG},\hat{\nabla}^{\mathit{Giv}},P)$ of $(\cG,\hat{\nabla}^{\mathit{Giv}},P)$ to a
$\trTLEPlog$-structure on $V'$. Moreover, consider the partial
compactification
$$
\begin{array}{rcl}
V& := &\left\{(t_0,q_1,\ldots,q_r,t_{r+1},\ldots,t_s\} \;|\; \|\underline{q}\|<e^{-M}, \|(t_0,t_{r+1},\ldots,t_s)\|<e^{-M}\right\} \\ \\
& \subset &
H^0(X,\dC)\oplus \dC^r \oplus \bigoplus_{k>1} H^{2k}(X,\dC)
\end{array}
$$
of $V_0$, then there is a structure of a logarithmic Frobenius manifold on $V$ restricting to the germ of a Frobenius manifold
defined by the quantum product at any point $(t_0,q_1,\ldots,q_r,t_{r+1},\ldots,t_s)\in H^0(X,\dC)\oplus H^2(X,\dC)/2 \pi i H^2(X,\dZ)\oplus \bigoplus_{k>1} H^{2k}(X,\dC)$.

\end{corollary}
\begin{proof}
Both statements follow from \cite[section 2.2, proposition 1.7 and proposition 1.10]{Reich1}.
\end{proof}

\subsection{$J$-function, Givental's theorem and mirror correspondence}
\label{subsec:MirrorCorr}

In order to compare the quantum $\cD$-module $\cG$ to the hypergeometric system
$\qMBLhat$ from the last section, we will use a particular
multivalued section of $\cG$, called the $J$-function.
Givental has shown in \cite{Giv7} that $I=J$ for Fano varieties
and that equality holds after a change of coordinates in the weak Fano case.
We use this equality to identify the two $\trTLEPlog$-structures
and deduce an isomorphism of Frobenius manifolds with logarithmic poles.

Actually, Givental's theorem is broader as it also treats the case of
nef complete intersections in toric varieties, however, the
B-model has a different shape for those varieties, the most prominent
example being the quintic hypersurface in $\dP^4$. In this case (this is
true whenever the complete intersection is Calabi-Yau) the
mirror is an ordinary variation of pure polarized Hodge structures, whereas in our
situation the Landau-Ginzburg model gives rise to a \emph{non-commutative Hodge structure}
as discussed in section \ref{sec:ncHodge}. We plan to discuss the
relation between the B-model of a (weak) Fano variety and that of its subvarieties in a subsequent paper.

We start with the definition of the $J$-function. It is convenient to introduce at the same time
an endomorphism valued series which is closely related $J$. We suppose from now
on that $X=\XSigA$ is a smooth toric weak Fano variety.
\begin{definition}
\label{def:JFunction}
\begin{enumerate}
\item
Define a $End(H^*(\XSigA,\dC))$-valued power series in $z^{-1},t_1, \ldots ,t_r$ by
$$
L(\delta,z^{-1})(T_a):=e^{-\delta/z}T_a
-
\sum_{
\begin{array}{c}
\SC \beta\in\textup{Eff}_{\XSigA}\backslash\{0\}\\
\SC j=0,\ldots,s\end{array}
}
e^{\delta(\beta)}
\left\langle
\frac{e^{-\delta/z}T_a}{z+\psi_1},T_j
\right\rangle_{0,2,\beta} T^j,
$$
here the gravitational descendent GW-invariant $\langle
\frac{T_j}{z+\psi_1},1\rangle_{0,2,\beta}$ has to be understood as the formal sum
$-\sum_{k\geq0} (-z)^{-k-1}\langle\psi_1^k T_j,1\rangle_{0,2,\beta}$.
\item
Define the $H^*(\XSigA,\dC)$-valued power series $J$ by
$$
J(\delta,z^{-1}):=e^{\frac{\delta}{z}}\cdot\left(1+\sum\limits_{
\begin{array}{c}
\SC \beta\in\textup{Eff}_{\XSigA}\backslash\{0\}\\
\SC j=0,\ldots,s\end{array}
}
e^{\delta(\beta)}
\left\langle
\frac{T_j}{z-\psi_1},1
\right\rangle_{0,2,\beta}
T^j
\right).
$$
\end{enumerate}
Notice that any product of cohomology classes appearing in the definition of $L$ and $J$ is the cup product.

\end{definition}

Observe that $L$ has the factorization $L = S \circ (e^{-\delta/z})$ where $S$ is the following $End(H^*(\XSigA,\dC))$-valued power series
\[
S(\delta,z^{-1})(T_a):= T_a
-
\sum_{
\begin{array}{c}
\SC \beta\in\textup{Eff}_{\XSigA}\backslash\{0\}\\
\SC j=0,\ldots,s\end{array}
}
e^{\delta(\beta)}
\left\langle
\frac{T_a}{z+\psi_1},T_j
\right\rangle_{0,2,\beta} T^j,
\]
The main tool we are going to use to identify the quantum $\cD$-module with
a hypergeometric system from the last chapter
is the following famous result of Givental.
\begin{theorem}[{\cite[theorem 0.1]{Giv7}}]
\label{theo:IgleichJ}
The coordinate change $\kappa$ from
\ref{prop:IFunction}, 3., transforms the $I$-function into the $J$-function, i.e., we have
$I=(\id_{\dC_\tau} \times\kappa)^*J$. In particular, it follows from proposition \ref{prop:IFunction}, 3. that $J$ defines a (multivalued)
holomorphic mapping from $\dC_\tau\times S^*$ to $H^*(\XSigA,\dC)$. If $\XSigA$ is Fano, then
$I=J$.
\end{theorem}
Denote by $\mathbf{S}$ the matrix-valued function which represents the endomorphism function $S$ with respect to the basis $T_0, \ldots ,T_s$. Similarly, $\mathbf{K}_i$ is the constant matrix representing the cup product with $T_i$, $\mathbf{\Omega}_i$ is the connection matrix of $\hat{\nabla}^{\mathit{Giv}}_{z q_i \partial q_i}$ and $\mathbf{V}$ the matrix $\diag(\deg(T_0), \ldots, \deg(T_s))$. We have the following
\begin{lemma}[{\cite[lemma 2.1,2.2]{IrFloer}}]
\label{lem:diffeqS}\begin{enumerate}
\item The matrix-valued function $\mathbf{S}$ satisfies
the following differential equations:
\begin{align}
&z q_i \frac{\partial \mathbf{S}}{\partial q_i}  - \mathbf{S}
\cdot \mathbf{K}_i  + \mathbf {\Omega}_i \cdot \mathbf{S}  = 0 \; , \notag \\
&\left(2 z \frac{\partial}{\partial z} +\sum_{i=1}^r (\deg q_i)q_i \frac{\partial}{\partial q_i}\right)\mathbf{S} + [\mathbf{V},\mathbf{S}] = 0 \; . \notag
\end{align}
\item
The $End(H^*(\XSigA,\dC))$-valued power series $S$ satisfies $S^*(\delta,z^{-1}) \cdot S(\delta,-z^{-1})=\id$, where $(-)^*$ denotes the adjoint with respect to the Poincar\'{e} pairing. In particular $S$ is invertible.
\end{enumerate}
\end{lemma}

The main properties of the $J$-function and of the endomorphism function $L$
are summarized in the following proposition.
\begin{proposition}\label{prop:ConvergenceLIJ}
\begin{enumerate}
\item
For any $\alpha\in H^*(\XSigA,\dC)$, we have
$$
\begin{array}{rcl}
\hat{\nabla}^{\mathit{Giv}}_{\partial_{t_k}} L\cdot \alpha & = & \hat{\nabla}^{\mathit{Giv}}_{q_k \partial_{q_k}} L\cdot \alpha =0 \\ \\
\hat{\nabla}^{\mathit{Giv}}_{z^2\partial_z} L\cdot \alpha& = &  L \cdot\left(z\mu- c_1(\XSigA)\cup \right)\cdot\alpha\\ \\
\end{array}
$$
\item The endomorphism-valued function $L$ is invertible.
\item
We have $J =L^{-1}(T_0) = \sum_{t=0}^s (s_t, T_0) T_t$, with $s_t = L(T_t)$:
\item
Both $L$ and $J$ are convergent on $\dP^1_z \setminus \{ 0\} \times (S^* \cap V'_0)$.
\end{enumerate}
\end{proposition}
\begin{proof}
\begin{enumerate}
\item
The first formula can be found in \cite[equation (25)]{Pan} and the second follows from lemma \ref{lem:diffeqS} by a straightforward calculation.

\item This follows from the second point of \ref{lem:diffeqS}.
\item See, e.g. \cite[lemma 10.3.3]{CK}.

\item
The multivalued functions $(s_t,T_0)$ are holomorphic in $\dC_\tau\times S^*$ as this is true for $J$ by theorem \ref{theo:IgleichJ} and proposition \ref{prop:IFunction}, 3. Using the formula $\hat{\nabla}^{\mathit{Giv}}_{q_a \partial_{q_a}}(s_t, T_l) = (s_t, T_a \circ T_l)$ we conclude that $s_t$ is a multivalued section of $\cG$ which is holomorphic in $\dC_\tau \times (S^* \cap V'_0)$, because monomials of the form $T_1^{n_1} \circ \ldots \circ T_r^{n_r}$ are a basis of $\cG$ in this domain.
\end{enumerate}
\end{proof}
Next we will define a twist of the endomorphism-valued function $L$ to produce truly flat sections of the Givental connection.
Define $\widetilde{L} = L \circ z^{-\mu} \circ z^\rho= S \circ e^{-\delta/z} \circ z^{-\mu} \circ z^\rho$. If we set $\widetilde{s}_t = \widetilde{L}(T_t)$,
where as before $\rho=c_1(\XSigA)\in H^2(\XSigA,\dC)=\dL^\vee$,
then it is a straightforward computation to see that $\hat{\nabla}^{\mathit{Giv}} \widetilde{s}_t =0$ for $t=0, \ldots,s$. As $L$ resp. $\widetilde{L}$ is invertible, we obtain that $\widetilde{s}_t$ is a basis of multivalued flat sections.

We also need to define a twisted $J$-function, namely $\widetilde{J}:= \sum_{t=0}^t \widetilde{J}_t T_t := \sum_{t=0}^{s}(\widetilde{s}_t,T_0) T_t = \widetilde{L}^{-1}(T_0)$. This yields, similarly to equation \eqref{eq:RepresentationOfOne-BSide}, the following formula
\begin{equation}\label{eq:Section1A-Side}
1=T_0 \stackrel{!}{=} \sum_{t=0}^s \widetilde{J}_t \widetilde{s}_t \in H^0(\dC^*_\tau\times V'_0,\overline{\cG})
\end{equation}
The following proposition uses all the previous results to identify the differential systems defined on both sides
of the mirror correspondence.
\begin{proposition}
\label{prop:IsologtrTLEP}
Let $W_0$ be a sufficiently small open neighborhood of $0\in \dC^{r,an}$ which is contained
in $S \cap V' \cap U^{0,an}$ and such that $\kappa$ induces an automorphism of $W_0$.
There is an isomorphism
$$
\phi:(\qMBLhat)^{an}_{|\dP^1_z\times W_0} \longrightarrow (\id_{\dP^1_z}\times\kappa)^*\overline{\cG}_{|\dP^1_z\times W_0}
$$
of $\trTLEPlog$-structures on $W^0$.
\end{proposition}
\begin{proof}
Define a morphism of vector bundles with connection
$$
\begin{array}{rcl}
\varphi:\left((\qMBLlog)^{an}_{|\dC_z\times W^0},\nabla\right) & \longrightarrow & (\id_{\dC_z}\times\kappa)^* \left(\overline{\cG}_{|\dC_z\times W^0},\hat{\nabla}^{\mathit{Giv}}\right) \\ \\
1 & \longmapsto & 1=T_0,
\end{array}
$$
where the connection operator $\nabla$ on the left hand side is the one from theorem \ref{theo:LogExtQDMod}.
The first task is to show that $\varphi$ is well-defined, i.e., that the following holds:
Put $\widetilde{\Box}'_{\underline{l}}:=(\id_{\dC_z}\times \kappa)_* \widetilde{\Box}_{\underline{l}}$
and $\widetilde{E}':=(\id_{\dC_z}\times \kappa)_*(z^2\partial_z+\sum_{a=1}^r \rho(p^\vee_a) zq_a\partial_{q_a})$, then
we have to show that
$$
\begin{array}{rcl}
\widetilde{\Box}'_{\underline{l}}(q_1,\ldots,q_r,z,\hat{\nabla}^{\mathit{Giv}}_{zq_1\partial_{q_1}},\ldots,\hat{\nabla}^{\mathit{Giv}}_{zq_r\partial_{q_r}}) (1)
& = & 0 \quad \forall \underline{l}\in\dL \\ \\
\widetilde{E}'\left(q_1,\ldots,q_r,z,\hat{\nabla}^{\mathit{Giv}}_{z^2\partial_z},\hat{\nabla}^{\mathit{Giv}}_{zq_1\partial_{q_1}},\ldots,\hat{\nabla}^{\mathit{Giv}}_{zq_r\partial_{q_r}}\right)(1)
& = & 0.
\end{array}
$$
Obviously, the objects on the left hand side of these equations are sections of
$(\id_{\dC_z}\times\kappa)^*\overline{\cG}_{|\dC_z\times W^0}$, i.e., they cannot have support on a proper subvariety, hence,
it suffices to show that they are zero on $\dC^*_\tau\times (W_0\cap S_2^0)$.
On that subspace we can use the presentation
$1 = \sum_{t=0}^s \widetilde{J}_t\cdot \widetilde{s}_t$
from equation \eqref{eq:Section1A-Side}. As the multivalued sections $\widetilde{s}_t$ are flat for $\hat{\nabla}^{\mathit{Giv}}$ it follows
that we have to show that
$$
\begin{array}{rllcl}
\widetilde{\Box}'_{\underline{l}}(\widetilde{J}_t) & =&
\widetilde{\Box}_{\underline{l}}((\id_{\dC^*_\tau}\times\kappa)^*\widetilde{J}_t) & = & 0 \\ \\
\widetilde{E}'(\widetilde{J}_t) & = &
\left(z^2\partial_z+\sum_{a=1}^r \rho(p^\vee_a) zq_a\partial_{q_a}\right)((\id_{\dC^*_\tau}\times\kappa)^*\widetilde{J}_t)&=&0.
\end{array}
$$
This is obvious by theorem
\ref{theo:IgleichJ} and by the equations \eqref{eq:OperatorsAnnihilateIFunc} in the proof of proposition \ref{prop:IFunction}.
Hence we obtain that $\varphi$ is a well-defined morphism of locally free sheaves compatible with the
connection operators on both sides.

Next we show the the surjectivity of $\varphi$: As we are allowed to replace $W_0$ by
a smaller open neighborhood of $\underline{0}\in \dC^r$, one easily sees that it suffices to show that
$\varphi$ is surjectiv on the germs at $(0,\underline{0})$ of both modules. Namely, we have flat structures
on $\dC^*_\tau\times (W_0\cap S_2^0)$ and on $\dC^*_\tau\times Z_a$ for all $a=1,\ldots,r$, so that if $\varphi$ is surjective
at some point in $\dC^*_\tau\times (W_0\cap S_2^0)$ resp. at some point in $\dC^*_\tau\times Z_a$, it will be surjective
on all of $\dC^*_\tau\times (W_0\cap S_2^0)$ resp.  $\dC^*_\tau\times Z_a$. By Nakayama's lemma, surjectivity on the germs at $(0,\underline{0})$ is guaranteed once we have surjectivity
at the fibre at $(0,\underline{0})$, which is evident as both fibres are canonically isomorphic
to $H^*(\XSigA,\dC)$ (for $\overline{\cG}_{|(0,\underline{0})}$, this isomorphism holds by definition, and for
$(\qMBLlog)^{an}_{|(0,\underline{0})}$, this is lemma \ref{lem:ModuleE}, 1.). Now by comparison of ranks,
we obtain that $\varphi$ is an isomorphism.

It remains to show that $\varphi$ can be extended to an isomorphism of $\trTLEPlog$-structures on $W_0$. First notice that
$\varphi$ yields an identification of the local systems $(\qM)^{\nabla}_{|\dC^*_\tau\times (W_0\cap S_2^0)}$
and $\cG^{\hat{\nabla}^{\mathit{Giv}}}_{|\dC^*_\tau\times (W_0\cap S_2^0)}$. In particular, it follows then from
lemma \ref{lem:ModuleE}, 3. that the monodromy
of $\cG^{\hat{\nabla}^{\mathit{Giv}}}_{|\dC^*_\tau\times (W_0\cap S_2^0)}$ around $Z_a=\{q_a=0\}$ is unipotent (this can also be shown by a direct calculation). Hence
by using the the same arguments as in proposition \ref{prop:LogTrTLEP} it suffices to identify
the punctual $\trTLEP$-structures $(\qMhat)_{|\dP^1\times\{\underline{0}\}}$ and $\overline{\cG}_{|\dP^1\times\{\underline{0}\}}$.
We already have such an identification on $\dC_z\times\{\underline{0}\}$ by restricting the above isomorphism
$\varphi$ to $\dC_z\times\{0\}$. Moreover, consider a basis $w_1,\ldots,w_\mu$ of $(\qMBLlog)_{|\dC_z\times\{\underline{0}\}}$ as in
the proof of proposition \ref{prop:ExtensionInftyAtZero}, 1., which extends the basis $T_0,T_1,\ldots,T_r,T_{r+1},\ldots,T_s$ of $H^*(\XSigA,\dC)=(\qMBLlog)_{|(0,\underline{0})}$.
Then by the definition of the Givental connection and of the morphism $\varphi$, the restriction
$\varphi_{|\dC_z\times\{\underline{0}\}}$ maps this basis is to $T_0,\ldots, T_s\in\overline{\cG}_{|\dC_z\times\{\underline{0}\}}\cong\oplus_{t=0}^s\cO_{\dC_z}T_t$. Remark also that the connection matrices in these bases of
$\nabla$ on $(\qMBLlog)_{|\dC_z\times\{\underline{0}\}}$ resp. $\hat{\nabla}^{\mathit{Giv}}$ on $\overline{\cG}_{|(0,\underline{0})}$
are equal, this follows from formula \eqref{eq:ConnMatrixQ0B-Side} resp. formula \eqref{eq:QDmod}. Hence
$\varphi$ extends to an isomorphism of $\dP^1_z$-bundles $\widetilde{\phi}:(\qMhat)_{|\dP^1_z\times\{\underline{0}\}} \rightarrow \overline{\cG}_{|\dP^1_z\times\{\underline{0}\}} =
\left((\id_{\dP^1_z}\times\kappa)^*\overline{\cG}\right)_{|\dP^1_z\times\{\underline{0}\}}$, compatible with the connections. By the same argument, this isomorphism
also respects the pairings $P$ on both sides, as it restricts to the identity at $z=0$.

As discussed above, we obtain from $\varphi$ and $\widetilde{\phi}$ an isomorphism
$$
\phi:(\qMBLhat)^{an}_{|\dP^1_z\times W_0} \longrightarrow (\id_{\dP^1_z}\times\kappa)^*\overline{\cG}_{|\dP^1_z\times W_0}
$$
of $\trTLEPlog$-structures on $W^0$, as required.
\end{proof}
As a consequence, we can now deduce an isomorphism of logarithmic Frobenius structures defined by the quantum
product resp. by the Landau-Ginzburg model (through the construction from subsection \ref{subsec:LogFrob}) of $\XSigA$.
\begin{theorem}\label{theo:MirrorSymmetry}
There is a unique isomorphism germ $Mir:(W_0\times \dC^{\mu-r},0) \rightarrow (V,0)$ which maps
the logarithmic Frobenius manifold from corollary \ref{cor:logFrob-ASide} (A-side) to that of theorem \ref{theo:logFrob-BSide} (B-side)
and whose restriction to $W_0$ corresponds to the isomorphism $\phi$ of $\trTLEPlog$-structures from above.
In particular, it induces the identity on the tangent spaces
at the origin, i.e., on $(H^*(X_{\Sigma_A},\dC),\cup)$.
\end{theorem}
\begin{proof}
This is a direct consequence of the uniqueness statement in theorem \ref{theo:logUnfoldingTheorem}, using the last proposition.
\end{proof}

\section{Non-commutative Hodge structures}
\label{sec:ncHodge}

In this section we will use the results from the previous parts of the
paper to show, via the fundamental theorem \cite[theorem 4.10]{Sa8}, that the quantum
$\cD$-module on the K\"ahler moduli space underlies a variation of pure polarized non-commutative Hodge structures.
Moreover, we study the asymptotic behavior near the large radius limit point and show
that the associated harmonic bundle is tame in the sense of Mochizuki and Simpson (see, e.g., \cite[definition 4.4]{Mo1}) along the boundary divisor.
We start by recalling briefly the necessary definitions.
\begin{definition}[{\cite[definition 2.12]{He4},\cite[definition 2.1]{HS4},\cite[definition 2.7]{KKP}}]
\label{def:ncHodge}
Let $M$ be a complex manifold and $n\in\dZ$ be an integer. A variation
of TERP-structures on $M$ of weight $n$ consists of the following set
of data.
\begin{enumerate}
\item
A holomorphic vector bundle $\cH$ on $\dC_z\times M$ with an algebraic structure in $z$-direction,
i.e., a locally free $\cO_M[z]$-module.
\item
A $\dR$-local system $\cL$ on $\dC_z^*\times M$, together with an isomorphism
$$
\iso:\cL\otimes_\dR \cO^{an}_{\dC_z^*\times M} \rightarrow \cH^{an}_{|\dC_z^*\times M}
$$
such that the connection $\nabla$ induced by $\iso$ has a pole of type $1$ along $\{0\}\times M$
and a regular singularity along $\{\infty\}\times M$.
\item
A polarizing form $P:\cL\otimes \iota^*\cL\rightarrow i^n\underline{\dR}_{\dC_z^*\times M}$,
which is $(-1)^n$-symmetric and which induces a non-degenerate pairing
$$
P:\cH\otimes_{\cO_{\dC_z\times M}}\iota^*\cH\rightarrow z^n \cO_{\dC_z\times M},
$$
here non-degenerate means that we obtain a non-degenerate symmetric pairing
$[z^{-n}P]:\cH/z\cH \times \cH/z\cH \rightarrow \cO_M$.
\end{enumerate}
\end{definition}
We also recall the notions of pure and pure polarized TERP-structures.
\begin{definition}\label{def:ppncHodge}
Let $(\cH,\cL,P,n)$ be a variation of TERP-structures on $M$. Write $\gamma:\dP^1\times M\rightarrow \dP^1\times M$ for the
involution $(z,x)\mapsto (\overline{z}^{-1},x)$ and consider $\overline{\gamma^*\cH}$, which is a holomorphic vector bundle
over $(\dP^1\backslash\{0\})\times \overline{M}$. Define a locally free $\cO_{\dP^1}\cC^{an}_M$-module $\widehat{\cH}$, where
$\cO_{\dP^1}\cC^{an}_M$ is the subsheaf of $\cC^{an}_{\dP^1\times M}$ consisting of functions annihilated by $\partial_{\overline{z}}$
by gluing $\cH$ and $\overline{\gamma^*\cH}$ via the following identification on $\dC_z^*\times M$. Let $x\in M$ and $z\in\dC_z^*$ and define
$$
\begin{array}{rcl}
c: \cH_{|(z,x)} & \longrightarrow & (\overline{\gamma^*\cH})_{|(z,x)} \\ \\
a &\longmapsto & \nabla\textup{-parallel transport of }\overline{z^{-n} \cdot a}.
\end{array}
$$
Then $c$ is an anti-linear involution and identifies $\cH_{|\dC_z^*\times M}$ with $\overline{\gamma^*\cH}_{|\dC_z^*\times \overline{M}}$.
Notice that $c$ restricts to the complex conjugation (with respect to $\cL$) in the fibres over $S^1\times M$.
\begin{enumerate}
\item
$(\cH,\cL,P,n)$ is called pure iff $\widehat{\cH} = p^* p_* \widehat{\cH}$, where $p:\dP^1\times M \twoheadrightarrow M$.
A variation of pure TERP-structures is also called variation of (pure) non-commutative Hodge structures (\textbf{ncHodge structure} for short).
\item
Let $(\cH,\cL,P,n)$ be pure, then by putting
$$
\begin{array}{rcl}
h: p_*\widehat{\cH} \otimes_{\cC^{an}_M}  p_*\widehat{\cH}  & \longrightarrow & \cC^{an}_M \\ \\
(s,t) & \longmapsto & z^{-n}P(s,c(t))
\end{array}
$$
we obtain a hermitian form on $p_*\widehat{\cH}$. We call $(\cH,\cL,P,n)$ a pure polarized TERP resp. ncHodge structure if this form
is positive definite (at each point $x\in M$).
\end{enumerate}
\end{definition}
\textbf{Remarks:} We comment on the differences between this definition and those
in \cite{HS4} resp. \cite{KKP}.
\begin{enumerate}
\item
One may want, depending on the actual geometric situation to be considered, the local system $\cL$ to be defined
over $\dQ$ (as in \cite{KKP}) or even over $\dZ$. This corresponds to the notion of real resp. rational Hodge structures
and to the choice of a lattice for them in ordinary Hodge theory.
\item
The reason for considering TERP-structures, and not only ncHodge structures, which are pure by definition
(this condition is called \emph{opposedness} condition in \cite{KKP})
is that there are natural examples of TERP-structures which are not pure (see, e.g., \cite[section 9]{HS4}).
\item
A ncHodge structure in the sense of \cite{KKP} does not contain any polarization data. However, the structures
we are considering, i.e., those defined by (families of) algebraic functions are polarizable in a natural way,
so that it seems reasonable to include these data in the definition.
\item
We did not put the \emph{$\dQ$-structure axiom} from \cite{KKP} in the definition of an ncHodge structure. This property,
roughly stating that the Stokes structure defined by the pole of $\nabla$ along $z=0$ (in case it is irregular) is already
defined on the local system $\cL$, and not only on its complexification $\cL\otimes_{\underline{\dR}}\underline{\dC}$
was part of the definition of a mixed TERP-structure in \cite{HS1}. It turns out that in some situations (see, e.g., \cite[section 8]{Mo6}),
this property is actually something to be proved, which is why we exclude this condition from the definition of a ncHodge structure. Notice
however that in the geometric situations we are studying, this condition will always be satisfied.
\end{enumerate}
The following theorem is the first result of this section.
\begin{theorem}
\label{theo:GKZunderliesNCHodge}
The restriction to $\dC_z\times (W_0 \cap S^0_2)$ of the quantum $\cD$-module $\cG$ underlies a variation of (pure) polarized
ncHodge structures of weight $n$ on $W_0 \cap S^0_2$.
\end{theorem}
\begin{proof}
We will show that
$\qMBL$ is a polarized ncHodge structure on $S^0_2$, then the statement follows from
proposition \ref{prop:IsologtrTLEP}.
We first show that $\qMBL$ is equipped with structures as in definition
\ref{def:ncHodge}, that is, that it underlies a variation of TERP-structures. Then we deduce from
\cite{Sa8} that this structure is pure and polarized.

It follows from corollary \ref{cor:BrieskornDualityDown} that $\qMBL$ is a locally free $\cO_{\dC_z\times S_2^0}$-module, equipped with a connection operator with a pole of type $1$ along $\{0\}\times S_2^0$ and that moreover we have a non-degenerate pairing $P:\qMBL \otimes \iota^* \qMBL \rightarrow z^n\cO_{\dC_z\times S_2^0}$. Recall also from the proof of theorem \ref{theo:GM-GKZUp} and of corollary
\ref{cor:GM-GKZDown} that the $\cD_{\dP^1_z\times S_2^0}$-module $\qM \otimes_{\cO_{\dP^1_z\times S_2}}\cO_{\dP^1_z\times S^0_2}$
equals $\FL^\tau_t(\cH^0(W,q)_+\cO_{S_1^0})$. Now the Riemann-Hilbert correspondence
gives $\DR^\bullet(\cH^0(W,q)_+\cO_{S_1^0}) = {^p}\cH^0 R^\bullet(W,q)_*\underline{\dC}_{S_1^0}$, where
${^p}\cH^\bullet$ is the perverse cohomology functor (see, e.g., \cite{Di}). Hence
$\DR^\bullet(\cH^0(W,q)_+\cO_{S_1^0})$ carries a real (resp. rational) structure, namely,
${^p}\cH^0 R^\bullet(W,q)_*\underline{\dR}_{S_1^0}$ (resp. ${^p}\cH^0 R^\bullet(W,q)_*\underline{\dQ}_{S_1^0}$).
We then deduce from \cite[theorem 2.2]{Sa1} that the the local system of flat sections of $((\qM)^{an},\nabla)$
is equipped with a real or even rational structure.
One could also invoke the recent preprint \cite{Mo7} and show that $\cH^0(W,q)_+\cO_{S_1^0}$ is a $\dR$-(or $\dQ$-)holonomic
$\cD$-module in the sense of \cite[definition 7.6]{Mo7}, which holds due to the regularity of $\cH^0(W,q)_+\cO_{S_1^0}$.
It then follows from loc.cit., section 9, that this real or rational structure is preserved
under the standard functors (direct image, inverse image, tensor product) in particular, under (partial) Fourier-Laplace transformation
(the elementary irregular rank one module has an obvious real/rational structure).
Hence $\FL^\tau_t(\cH^0(W,q)_+\cO_{S_1^0})$ has a real (resp. rational) structure,
which shows that $\qMBL$ underlies a variation of TERP-structures on $S_2^0$.

It remains to show that this structure is pure and polarized in the sense of definition \ref{def:ppncHodge}. It is sufficient to
do this for the restriction $(\qMBL)_{|\dC_z\times\{\underline{q}\}}$ for all $\underline{q} \in S_2^0$.
Write $W_{\underline{q}}$ for the restriction $W_{|\pr^{-1}(\underline{q})}:q^{-1}(\underline{q})\rightarrow \dC_t$,
then the restriction of the tuple $(\qM,\qMBL,P)$ to $\dC_z\times \{\underline{q}\}$ is exactly the tuple $(G,G_0,\widehat{P})$ associated to $W_{q}$
which was considered in \cite[theorem 4.10]{Sa8}, where one has to use the comparison result \cite[lemma 5.9]{Sa11} to identify (possibly
up to a non-zero constant, see the remark after the proof of lemma \ref{lem:ModuleE}) the pairing $P$ defined on $\qMBL$ with the pairing $\widehat{P}$ from \cite[theorem 4.10]{Sa8}.
Then it is shown in loc.cit. that one can associated to  $(G,G_0,\widehat{P})$ an integrable polarized
twistor structure, which means exactly that the variation of TERP-structures $(\qMBL)_{|\dC_z\times\{\underline{q}\}}$
is pure polarized, i.e., that it is a variation of (pure) polarized ncHodge structures.
\end{proof}

In order to state the second result of this section, recall the following fact (see, e.g., \cite[lemma 3.12]{HS1}).
\begin{proposition}
Let $(\cH,\cL,P)$ be a variation of polarized ncHodge structures of weight $n$ on $M$. Put
$E:=p_*\widehat{\cH}$, which is a real-analytic bundle
equipped with a holomorphic structure defined by the isomorphism
$E \cong \cH/z\cH \otimes_{\cO_M} \cC^{an}_M$, a Higgs field $\theta:=[z\nabla_z]
\in{\cE\!}nd_{\cO_M}(\cH/z\cH)\otimes\Omega^1_M$ and the hermitian metric $h$ from above.
Then the tuple $(E,\overline{\partial},\theta, h)$ (where $\overline{\partial}$
is the operator defining the holomorphic structure on $E$) is a harmonic bundle in the sense of \cite{Si1}.
\end{proposition}
Let $(E,\overline{\partial},\theta,h)$ be the harmonic bundle associated by the last proposition
to the ncHodge structure $\qMBL$ on $S_2^0$ (resp. $\cG$ on $W_0\cap S^0_2$). The next result concerns the asymptotic behavior of $E$ along the boundary
divisor $Z=\bigcup_{a=1}^r \{q_a=0\}$.
\begin{theorem}\label{theo:TameHarmonic}
Put $\widetilde{U}:=(U \backslash Z)^{an} \subset S_2^{0,an}$. Then
the restriction of the harmonic bundle $(E,\overline{\partial},\theta,h)$ to $\widetilde{U}$ is tame along $Z$ in the sense of \cite[definition 4.4]{Mo1}.
\end{theorem}
\begin{proof}
Recall that the tameness property of a harmonic bundle defined by a variation of polarized ncHodge structures
can be expressed in the chosen coordinates $q_1,\ldots,q_r$ as follows: Write the Higgs field
$\theta\in{\cE\!}nd_{\cO_{\widetilde{U}}}(\cH/z\cH)\otimes\Omega^1_{\widetilde{U}}$ as
$$
\theta = \sum_{a=1}^r \theta_a \frac{d q_a}{q_a}
$$
with $\theta_i\in{\cE\!}nd_{\cO_{\widetilde{U}}}(\cH/z\cH)$. Then $(E,\overline{\partial},\theta,h)$ is called
tame iff the coefficients of the characteristic polynomials of all $\theta_i$ extend to holomorphic functions on $U^{an}$.
Now consider the locally free $\cO_{\dC_z\times U}$-module $\qMBLlog$ from theorem
\ref{theo:LogExtQDMod}. The connection
$$
\nabla:\qMBLlog \longrightarrow \qMBLlog \otimes z^{-1}\Omega^1_{\dC_z\times U}\left(\log \left((\{0\}\times U) \cup (\dC_z\times Z)\right)\right)
$$
induces
$$
\theta':=[z\nabla]\in {\cE\!}nd_{\cO_{U^{an}}}\left((\qMBLlog)^{an}_{|\{0\}\times U^{an}}\right)\otimes\Omega^1_{U^{an}}(\log Z)
$$
As $\theta'$ restricts to $\theta$ on $\widetilde{U}$, we see that
if we write
$\theta'=\sum_{a=1}^r \theta'_a \frac{dq_a}{q_a}$, then $\theta'_a$ is the holomorphic extension
of $\theta_a$ we
 are looking for.
\end{proof}

\bibliographystyle{amsalpha}
\def\cprime{$'$} \def\cprime{$'$}
\providecommand{\bysame}{\leavevmode\hbox to3em{\hrulefill}\thinspace}
\providecommand{\MR}{\relax\ifhmode\unskip\space\fi MR }
\providecommand{\MRhref}[2]{  \href{http://www.ams.org/mathscinet-getitem?mr=#1}{#2}
}
\providecommand{\href}[2]{#2}

\vspace*{1cm}

\nd
Lehrstuhl f\"ur Mathematik VI \\
Institut f\"ur Mathematik\\
Universit\"at Mannheim,
A 5, 6 \\
68131 Mannheim\\
Germany

\vspace*{1cm}

\nd
Thomas.Reichelt@math.uni-mannheim.de\\
Christian.Sevenheck@math.uni-mannheim.de

\end{document}